\documentclass[leqno, 11pt]{amsart}
\usepackage{kotex}

\usepackage{amsfonts}
\usepackage{amssymb}
\usepackage{amsmath}
\usepackage{float}
\usepackage{tabularx}
\usepackage{array}

\newcolumntype{Y}{>{\centering\arraybackslash}X}
\newcolumntype{P}[1]{>{\centering\arraybackslash}m{#1}}

\usepackage{graphicx}
\usepackage{subcaption}
\usepackage{enumitem}
\usepackage{amsthm}
\usepackage{amsbsy}

\usepackage{tikz}
\usepackage{tikz-cd}

\usepackage{soul}
\usepackage[colorlinks=true, citecolor=blue]{hyperref} 
\usepackage{ulem}
\usepackage{marginnote}
\usepackage[margin=1in]{geometry}

\usepackage[export]{adjustbox}

\usepackage{multirow}
\usepackage{makecell}

\usepackage[dvipsnames]{xcolor}
\usepackage{color}

\numberwithin{equation}{section}

\newtheorem{theorem}{Theorem}[section]
\newtheorem{lemma}[theorem]{Lemma}

\usepackage{lmodern}

\title{Local Well-Posedness of a Modified NSCH–Oldroyd System: PINN-Based Numerical Computation}
\author{WOOJEONG KIM\textsuperscript{1} \\ \textsuperscript{1}Department of Mathematics and Institute for Scientific Computing and Applied Mathematics, \\ Indiana University, Bloomington, IN 47405, USA}
\date{}

\begin{document}
\pagestyle{plain} 
\raggedbottom
\raggedbottom

\begin{abstract}
Motivated by thrombus modeling, we study a modified Navier–Stokes–Cahn–Hilliard system and consider PINN-based numerical illustrations for the modified system. To enable the analysis, we introduce a diffusion-enhanced system for the deformation variable while preserving the associated dissipative energy structure. We prove local well-posedness for this new system. We also present PINN-based numerical illustrations for representative thrombus cases and report residual losses and benchmark errors obtained with Metropolis--Hastings sampling based on the energy decay.
\end{abstract}

\maketitle

\section{Introduction}
Recently, fluid--structure interaction (FSI) models have been developed in increasingly detailed forms to describe the interaction between deformable structures and surrounding fluids. Among such models, diffuse-interface Navier--Stokes--Cahn--Hilliard (NSCH) systems have been widely used as continuum models for binary-fluid dynamics \cite{dsb23}. In \cite{maa18}, a fully Eulerian formulation was introduced for the velocity field. Moreover, NSCH thrombus model \cite{zyl20,xxk17} incorporated additional Oldroyd-B-type equations to represent elastic stress. Based on these developments, more detailed NSCH-type systems have been proposed for thrombus dynamics using physical reference data. However, in those works, the simulations are guided by external data, while rigorous mathematical results such as well-posedness or strong-solution theory are not established.
\\
\par
By contrast, the same thrombus model of \cite{ktt22} established local well-posedness and uniqueness of strong solutions for the governing system under suitable initial data. That paper proved a microstructural energy-dissipation property, which provides a useful analytical framework for understanding deformation dynamics. However, from the viewpoint of stabilization and simulation, the governing system in the paper lacks a diffusion term in the deformation variable $F$. This absence makes the system less robust both analytically and computationally, since diffusion terms are technically useful for controlling higher-order terms in the a priori estimates and for simulating stabilized systems.
\\
\par
To address this issue, we modify the system in \cite{ktt22} by adding a small diffusion term to the deformation equation while preserving the underlying physical and analytical stability of the model. A theoretical contribution of this paper is the construction of a modified governing system with an additional diffusion term for the deformation variable $F$. This framework provides a more stable basis for analysis and for future data-assimilation-type applications. The main analytical contribution of this paper is the proof of local well-posedness for the resulting system.
\\
\par
Motivated by this analytical development, we also perform numerical simulations for the coupled NSCH system using physics-informed neural networks (PINNs). As pointed out in \cite{aj25,qh22}, one of the main difficulties in PINN simulation is the accurate resolution of regions in which the solution exhibits sharp gradients. In phase-field models, this difficulty is concentrated near the interface of the phase variable, where rapid spatial variation occurs. Nevertheless, for a governing system of complexity comparable to that in \cite{qh22}, we obtain stable PINN simulations without relying on the reference-data setting used there for bubble dynamics. This serves as a supplementary numerical illustration for the paper.
\\
\par
Therefore, the main objective of this paper is to establish local well-posedness for a diffusion-enhanced NSCH--Oldroyd-type thrombus model and . As a secondary component, we present PINN-based numerical illustrations for representative thrombus cases, including challenging interfacial regimes.
\\
\par
The remainder of this paper is organized as follows. In Section 2, we derive the energy-dissipation law and establish the a priori estimates for the governing system of the proposed model. In particular, the stabilizing effect of the new model is reflected in the additional diffusion term and in the higher regularity obtained for the strong solutions in the main well-posedness theorem. In Section 3, we prove the local well-posedness of the governing system. In Section 4, we present PINN-based numerical illustrations for the modified system. Finally, Section 5 contains concluding remarks.
\\
\par
As orientation-preserving diffeomorphisms, let $x(t, \cdot)$ be a time-dependent family. Also, $X(\cdot, t)$ is the inverse of $x(\cdot, t)$ and thus $X(x, t)$ be the corresponding reference map. The velocity field is denoted as $u(x, t)$ and
\[ u(x, t) = \left.\frac{dx(t, X)}{dt}\right|_{X = X(x,t)}. \]
\\
The deformation gradient is denoted as $F$:
\[ F(x, t) = \left.\frac{\partial x(t, X)}{\partial X}\right|_{X = X(x,t)}. \]
\par
Assume that this $F$ satisfies the following:
\[ \partial_t F - { k (\nu (\phi) \Delta F +2 \nabla \nu (\phi) \cdot \nabla F ) }+ u \cdot \nabla F = \nabla u F, \]
where $[\nabla u]_{ij} = \frac{\partial u^i}{\partial x^j}$. This is written in component-wise as:
\begin{equation}\label{F_elementwise}
F^{ij}_t - { k (\nu (\phi) \Delta F^{ij} +2 \nabla \nu (\phi) \cdot \nabla F^{ij} ) }+ \sum_{k=1}^d u^k\partial_k F^{ij} = \sum_{k=1}^d \partial_k u^i F^{kj}, \quad 1 \leq i,j \leq d.
\end{equation}
\\
In the analytical part, we do not impose the incompressibility constraint $\det F = 1$ for the modified diffusion-enhanced system. In the numerical experiments, however, we enforce $\det F = 1$ throughout the space-time training domain.\\
\par
Let $\Omega \subset \mathbb{R}^d$ be a bounded domain with a sufficiently smooth boundary $\partial\Omega$. On $\Omega \times (0, T)$, we consider the following governing equations:

\begin{equation}\label{governing_system}
\left\{ 
\begin{array}{rcl}
\rho(\frac{\partial u}{\partial t} + u \cdot \nabla u) + \nabla p - \nabla \cdot (\eta(\phi) \nabla u) &=& -\lambda \nabla \cdot (\nabla \phi \otimes \nabla \phi) \\[5pt]
+ \nabla \cdot \left({ \nu(\phi)}(FF^T - I)\right) &- &\dfrac{\eta(\phi)(1-\phi)u}{\kappa(\phi)}, \\[5pt]
\nabla \cdot u &=& 0, \\[5pt]
\frac{\partial F}{\partial t} - { k (\nu (\phi) \Delta F +2 \nabla \nu (\phi) \cdot \nabla F ) }+u \cdot \nabla F &=& \nabla u F, \\[5pt]
\frac{\partial \phi}{\partial t}  + u \cdot \nabla \phi &=& \tau \Delta \mu, \\[5pt]
\mu &=& -\lambda \Delta \phi + \lambda \gamma f'(\phi) + \dfrac{{ \nu'(\phi)}}{2} \text{tr}(FF^T - I),
\end{array}
\right.
\end{equation}
\\
$u, \phi, F$ and $p$ are variables of velocity, phase-field variable, deformation gradient and pressure. The blood region is represented by $\phi = 1$, the mixture of blood and thrombus is denoted as $0 < \phi <1$ and the thrombus region is represented by $\phi = 0$. For the other parameters, $\rho, \eta, \kappa$ are mass density, dynamic viscosity and permeability. The variable $\mu$ denotes the chemical potential. 
\\
The parameters $\gamma, \tau$ and $\lambda$ are positive constants representing the interfacial mobility, relaxation parameter, and mixing energy density, respectively. The function $f$ is denoted as the double-well potential $f(\phi) = \frac{(\phi-1)^2\phi^2}{4h^2}$ with the interfacial thickness $h$.
The governing system is supplemented with the following initial and boundary conditions:
\begin{equation}\label{bdry}
\left\{
\begin{array}{lcl}
    &u = \mathbf{0}, \quad \partial_n \mu = \partial_n \phi = {  \partial_n F^{i,j}} = 0 ~ (1 \leq i, j \leq d)\quad \text{on } \partial\Omega \times (0,T), \\[5pt]
    &u(\cdot,0) = u_0, \quad \phi(\cdot,0) = \phi_0, \quad F(\cdot,0) = F_0 \quad \text{in } \Omega,
\end{array}
\right.
\end{equation}
\\
where $n$ denotes the outward unit normal vector on $\partial\Omega$.
\\
\par
To derive the a priori estimates, we assume that $\eta, \kappa \in C^1$, $\nu \in C^3$ and that for some $\alpha, \beta > 0$,
\begin{equation}\label{visco_diff}
\begin{aligned}
&\alpha \leq \eta(x), { \nu(x)}, \kappa(x) \leq \beta \quad \forall x \in \mathbb{R}.
\end{aligned}
\end{equation}
\par
Note that $|\cdot|$ is written as the space $L^2(\Omega)$ norm and $(\cdot, \cdot)$ is the corresponding inner product. $L^2(\Omega)^d$ and $L^2(\Omega)^{d\times d}$ are written in the same way.
For details of the deformation gradient inner product, if $F = [\xi_{ij}] \in L^2(\Omega)^{d\times d}$ where $1 \leq i,j \leq d$, the inner product is $|F|^2 = (F, F) = \int_{\Omega} \text{tr}(FF^T)dx$.
\\
\par
To specify the solution spaces for our governing system, we introduce the following sets. Solenoidal vector field is defined as $\mathcal{V} = \{ u \in C_0^\infty(\Omega)^d,\ \text{div}\, u = 0 \}$ and the closures of $\mathcal{V}$ in $L^2(\Omega)^d$ and $H_0^1(\Omega)^d$ are $H$ and $V$ respectively. Equivalently, these spaces can be written as follows:
\begin{align*}
H &= \{ u \in L^2(\Omega)^d,\ \text{div}\, u = 0,\ u \cdot n = 0 \text{ on } \partial\Omega \}, \\
V &= \{ u \in H_0^1(\Omega)^d,\ \text{div}\, u = 0 \}.
\end{align*}
\par
Additionally, for Stokes operator $A := -\mathbb{P}\Delta$ where $\mathbb{P}$ is the Helmholtz-Leray orthogonal projection from $L^2(\Omega)^d$ onto $H$, let $D(A) \subset V$ be the domain of the $A$. Then $D(A) = H^2(\Omega)^d \cap V$ as it is well known.
\\
\par
One of our main results is the following theorem, which states strong well-posedness of the solution in \eqref{governing_system} - \eqref{visco_diff}.
\begin{theorem}\label{Thm_local_wp}
For $d=2,3$, let $\Omega \subset \mathbb{R}^d$ be a bounded open set with a sufficiently smooth boundary. We are given $u_0 \in D(A)$, $\phi_0 \in H^5(\Omega)$ such that $\partial_n \phi_0 = 0$, and $F_0 \in H^3(\Omega)^{d\times d}$ such that $\partial_n F_0^{i,j} = 0$, $\partial_n \Delta F_0^{i,j} = 0$ $(1 \leq i,j \leq d)$. Then there exists $0 < T_0 \leq T$ such that (\ref{governing_system})--(\ref{visco_diff}) has a unique solution $(u, F, \phi, p)$ on $[0, T_0]$ such that
\begin{equation*}
\begin{aligned}
    &u \in C([0, T_0]; D(A)) \cap L^2(0, T_0; H^3(\Omega)^d), \quad
    \partial_t u \in C([0, T_0]; H) \cap L^2(0, T_0; D(A)), \\
    &
    F \in C([0, T_0]; H^2(\Omega)^{d\times d}) \cap L^2(0, T_0; H^3(\Omega)^{d\times d}) \cap H^1(0, T_0; { H^2}(\Omega)^{d \times d}), \\
    &
    \phi \in C([0, T_0]; H^3(\Omega)) \cap L^2(0, T_0; H^4(\Omega)) \cap H^1(0, T_0; { H^3}(\Omega)), \\
    &\nabla p \in L^2(0, T_0; H^1(\Omega)).
\end{aligned}
\end{equation*}
\end{theorem}
\section{A Priori Estimates}\label{section_ape}

In this section, we derive the a priori estimates which will be used in the Faedo–Galerkin approximation in the next section. We use $(u, \phi, F)$ instead of using subscript $n$ for the elements in the $n-$dimensional finite subspace of Galerkin scheme which we define in section \ref{Galerkin_def}.
\\
\par
The shorthand $|\cdot|_s$ is written to abbreviate $|\cdot|_{H^s(\Omega)}$, $|\cdot|_{H^s(\Omega)^d}$ or $|\cdot|_{H^s(\Omega)^{d\times d}}$. For dimension $d$, we simplify the processes to show the case of $d =3$ since this will prove $d=2$ case as well.
\\
\par
To estimate derivatives of $u \in V$ or $u \in D(A)$, we use basic properties of $u$. The Stokes operator $A$ is a self-adjoint and positive-definite operator with compact inverse $A^{-1} : H \rightarrow H$. Therefore, there is a positive orthonormal basis associated with eigenvalues 
\begin{equation}\label{Stokes_e_values}
0 < \lambda_1 \leq \lambda_2 \leq \lambda_3 \leq .... 
\end{equation}
For the smallest eigenvalue $\lambda_1$, we have 
\begin{equation}\label{poincare_u}
\begin{aligned}
    &\lambda_1 |u|^2 \leq |\nabla u|^2, \text{\qquad if $u \in V$}
    \\&
    \lambda_1 |\nabla u|^2 \leq |A u|^2, \text{\qquad if $u \in D(A)$}
\end{aligned}    
\end{equation}
Also, since $A$ is the isomorphism of $D(A)$ onto $H$ and $D(A) = H^2(\Omega) \cap V$, we can derive
\begin{equation}\label{u_H2}
\begin{aligned}
    & |u|_2 \leq C|A u|^2, \text{\qquad if $u \in D(A)$}
\end{aligned}    
\end{equation}
\\
\par
With denote the average of $\phi$ over $\Omega$ as
\begin{equation*}
\langle \phi \rangle_{\Omega} := \frac{1}{|\Omega|} \int_{\Omega} \phi dx,
\end{equation*}
integration on $\eqref{governing_system}_4$ with divergence free condition and boundary condition yields
\begin{equation}\label{average}
\langle \phi_t \rangle_{\Omega} = \langle \Delta \phi \rangle_{\Omega} = \langle \Delta \phi_t \rangle_{\Omega} = 0. 
\end{equation}
Therefore, integrating in time yields
\begin{equation}\label{average_const}
|\langle \phi(t) \rangle_{\Omega}| = |\langle \phi_0 \rangle_{\Omega}| := K_0.
\end{equation}
We also use several auxiliary lemmas from \cite{ktt22} directly, which are needed in the Faedo–Galerkin argument.
\\
\par
The following Lemma will be applied to bound $|u|_3$ in estimates in this section (Refer \cite{c61}, \cite{ktt22}).
\begin{lemma}\label{lemma_u}
Suppose $\Omega \subset \mathbf{R}^3$ is an open bounded set with a sufficiently smooth boundary and $(u,p) \in V \times L^2(\Omega)$ is the weak solution of the following problem
\begin{equation}
    \begin{aligned}
        - \nabla \cdot (\eta(\phi) \nabla u) + \nabla p &= f \qquad \text{in } \Omega
        \\
        div u &= 0 \qquad \text{in } \Omega
        \\
        \int_{\Omega} \frac{p}{\eta(\phi)} dx &= 0 \qquad \text{in } \Omega
    \end{aligned}
\end{equation}
where $\eta \in C^2(\mathbf{R})$ is given as \eqref{visco_diff} and $\phi \in H^3(\Omega)$. Then $u \in (H^3 (\Omega)^d  \cap V )$ and the corresponding constant $C > 0$ exists s.t.
\begin{equation}
    |u|_3 + |\frac{p}{\eta (\phi)}|_1 \leq C |f|_1 
    \left(
    1 + (1 + |\phi|_2^2)(|\phi|_2^2 + |\phi|_2^{\frac 12} + |\phi|_3^{\frac 12})
    \right).
\end{equation}
\end{lemma}
\begin{proof}
    This is derived in \cite{ktt22} Lemma 2.2.
\end{proof}
Also, there is another lemma for bounding the higher degree of $\phi$.
\begin{lemma}\label{lemma_phi}
Suppose $\Omega \subset \mathbf{R}^3$ is an open bounded set with a sufficiently smooth boundary and $f \in L^2(\Omega)$, $g \in H^{\frac 12}(\partial \Omega)$ and $\phi \in H^2(\Omega)$ are the weak solution of the following biharmonic inhomogeneous Neumann boundary value problem
\begin{equation}
    \begin{aligned}
        \Delta^2 \phi &= f \qquad \text{in }\Omega
        \\
        \partial_n \phi = 0, ~ \partial_n \Delta \phi &= g \qquad \text{on } \partial \Omega
    \end{aligned}
\end{equation}
where the following compatibility condition holds for $f$ and $g$ as
\begin{equation}
    \int_{\Omega} f dx = \int_{\partial \Omega} g d\Gamma.
\end{equation}
Then $\phi \in H^4(\Omega)$ and there exists some constant $\tilde{C} > 0$ independent of $\phi$, $f$, $g$ s.t.
\begin{equation}
    \begin{aligned}
        &|\phi|_{H^4(\Omega) / \mathbf{R}}^2 \leq \tilde{C} \left( |f|^2 + |g|_{H^{\frac 12}(\partial \Omega) }^2\right), ~ \text{and}
        \\&
        |\phi|_{H^4(\Omega)}^2 \leq \tilde{C} \left( |f|^2 + |g|_{H^{\frac 12}(\partial \Omega)}^2 + |\int_{\Omega} \phi dx|^2 \right)
    \end{aligned}
\end{equation}
\end{lemma}
\begin{proof}
    This is derived in \cite{ktt22} Lemma 2.1.(Refer \cite{adn59})
\end{proof}
\qquad \\
Similar to Lemma \ref{lemma_phi}, note that if $\phi \in H^2(\Omega)$ is such that $\partial_n \phi = 0$ on $\partial\Omega$, then, for some constant $C > 0$,
\begin{equation}\label{H2_phi}
|\phi|_{H^2(\Omega)} \leq C( |\Delta \phi| + \left| \int_{\Omega} \phi dx \right|).
\end{equation}
\qquad \\
\par
Moreover, based on Neumann conditions and mean zero conditions on $\phi$ and $F$, we obtain several simple inequalities by using generalized poincare theorem.\\
From Neumann boundary conditions on $\phi_t$ in \eqref{bdry} with $\langle \phi_t \rangle_{\Omega} = 0$ as \eqref{average},
\begin{equation}\label{phit_H1}
    |\phi_t|_1 \leq C |\nabla \phi_t| 
\end{equation}
by on the generalized Poincare inequality.
With the Neumann condition on $F$ in \eqref{bdry},
\begin{equation}\label{DelF_H1}
\begin{aligned}
&|\Delta F^{i,j}|_1 \leq |\nabla \Delta F^{i,j}| + |\int \Delta F^{i,j} dx| = |\nabla \Delta F^{i,j}|.
\\&\text{and in the same way,}
\\&
|\Delta F_t^{i,j}|_1 \leq |\nabla \Delta F_t^{i,j}| + |\int \Delta F_t^{i,j} dx| = |\nabla \Delta F_t^{i,j}|
\end{aligned}
\end{equation}
for $1 \leq i, j \leq d$.\\
Also, with the Neumann boundary condition on $\phi$ in \eqref{bdry}, the previous $H^2$ norm bound \eqref{H2_phi} yields 
\begin{equation}\label{gradphi_H1}
\begin{aligned}
&|\nabla \phi|_1 = |\nabla (\phi - \int_{\Omega} \phi dx)|_1 
 \leq | (\phi - \int_{\Omega} \phi dx)|_2
 \leq | \Delta \phi |.
\end{aligned}
\end{equation}
Also, from \eqref{H2_phi} and Neumann boundary condition on $\phi_t$ in \eqref{average},
\begin{equation}\label{gradphit_H1}
\begin{aligned}
&|\nabla \phi_t|_1  \leq | \phi_t |_2
 \leq | \Delta \phi_t |.
\end{aligned}
\end{equation}
Similarly, from the generalized Poincare inequality and Neumann boundary condition on $\phi$ in \eqref{bdry},
\begin{equation}\label{Delphi_H1}
\begin{aligned}
&|\Delta \phi|_1  \leq C | \nabla \Delta \phi| + |\int_{\Omega} \Delta \phi dx|
 \leq C | \nabla \Delta \phi|.
\end{aligned}
\end{equation}
Likewise, from the Neumann boundary condition on $\phi_t$ in \eqref{bdry} and generalized Poincare theorem,
\begin{equation}\label{Delphit_H1}
\begin{aligned}
&|\Delta \phi_t|_1  \leq C | \nabla \Delta \phi_t|.
\end{aligned}
\end{equation}
\qquad
\\
\par
We next introduce several simple estimates in dual Sobolev spaces. To estimate the functionals having negative Sobolev space norm, we consider the several simple inequalities.
\\
For $\Xi \in L^2(\Omega)^{d \times d}$, we define the negative norm of the dual space $H^{-1}(\Omega)$ as 
\begin{equation}\label{def_neg_sobol}
    |\Xi|_{H^{-1}( \Omega)} = sup_{\omega \in {H_0^{1}}^{d \times d}( \Omega)}  \frac{(  \Xi , \omega ) }{|\omega|_1} 
\end{equation}
This definition follows \cite{ErnGuermond2004}.
Then, for given $f \in L^2(\Omega)$ and for any $w \in H_0^{1}(\Omega)$,
\begin{equation*}
    \begin{aligned}
        &(\nabla g, w) = -(g, \nabla w) \leq |g||w| 
        \leq C |g||w|_1,
    \end{aligned}
\end{equation*}
which implies
\begin{equation}\label{H-1norm_div}
    |\nabla g|_{H^{-1}(\Omega)} \leq C |g|.
\end{equation}
In the same way, we get 
\begin{equation}\label{H-1norm_Del}
    \begin{aligned}
        |\Delta g|_{H^{-1}(\Omega)} \leq C |\nabla g|.
    \end{aligned}
\end{equation}
\\
\par
Additionally, there is lemma on adjusting the degree of derivative of the functionals in the dual space of the Sobolev spaces.
\begin{lemma}\label{lem_neg_sobolev}
Let $\Omega$ be a $d-$dimensional bounded domain with sufficiently smooth boundary. If $F \in H^2(\Omega)^{d \times d}$, then it satisfies the following inequality.
\begin{equation}
    |\partial_n F^{ij}|_{H^{-1/2}(\partial \Omega)} \leq C_d(|F|_{H^{1}} + |\Delta F|_{H^{-1}})
\end{equation}
\begin{proof}
From definition of the negative order space as the dual of the Sobolev space, we can compute the following. From the definition of the dual space for sobolev spaces as in \cite{ErnGuermond2004},
    \begin{equation}\label{def_neg_sobolev}
        \begin{aligned}
            &|\partial_n F^{ij}|_{H^{-1/2}(\partial \Omega)} = sup_{g \in H_0^{1/2}(\partial \Omega)} \{ \frac{\int_{\partial \Omega} \partial_n F g d\Gamma }{|g|_{H^{1/2}(\partial \Omega)}}\} = sup_{g \in H_0^{1/2}(\partial \Omega)} \{ \frac{\int_{\partial \Omega}  \partial_n F g d\Gamma }{|g|_{H^{1/2}(\partial \Omega)}}\}
        \end{aligned}
    \end{equation}
We want to discuss the bound with norm on the whole domain $\Omega$ for this inequality. From the generalized trace theorem as proved in Theorem 5 (i) in \cite{m87}, we know that there is trace operator $\gamma : H^1(\Omega) \rightarrow H^{1/2} (\partial \Omega)$  which is surjective and bounded. Furthermore, from Theorem 5 (ii) in \cite{m87}, there is bounded linear right inverse $\epsilon$ s.t. $\epsilon(g) = \tilde{g}$ for any $g \in H^{1/2} (\partial \Omega)$ when $\tilde{g} \in H^{1} ( \Omega) $ satisfies $\gamma (\tilde{g}) = g$. From the boundedness of this right inverse $\epsilon$ of $\gamma$, there is constant $\hat{C}$ s.t. $| \tilde{g} |_{H^{1} ( \Omega)} \leq \hat{C}|\gamma(\tilde{g}) |_{H^{1/2} (\partial \Omega)}$ for any $\tilde{g} \in H^{1} ( \Omega)$.
\\
\par
Now we can continue the norm inequality from \eqref{def_neg_sobolev} as
    \begin{equation}
        \begin{aligned}
            & sup_{g \in H_0^{1/2}(\partial \Omega)} \{ \frac{\int_{\partial \Omega}  \partial_n F^{ij} g d\Gamma }{|g|_{H^{1/2}(\partial \Omega)}}\}
            \\&
            \leq C sup_{\epsilon (g) \in H^{1} ( \Omega)} \{ \frac{\int_{ \Omega}  \nabla \cdot (\nabla F^{ij}  \epsilon (g)) dx}{|\epsilon (g)|_{H^{1}( \Omega)}}\} 
            = C { sup_{\tilde{g} \in H^{1}} ( \Omega)} \{ \frac{\int_{ \Omega}  \Delta F^{ij}  \tilde{g} + \nabla F^{ij} \cdot \nabla \tilde{g} dx}{|\tilde{g}|_{H^{1}( \Omega)}}\}
            \\&
            \leq C sup_{ |\tilde{g}|_1 \leq 1 } \{ \frac{\int_{ \Omega}  \Delta F^{ij}  \tilde{g} + \nabla F^{ij} \cdot \nabla \tilde{g} dx}{|\tilde{g}|_{H^{1}( \Omega)}}\}
            \leq 
            C(|\Delta F^{ij}|_{H^{-1}} +|F^{ij}|_{H^{1}} ).
        \end{aligned}
    \end{equation}
\end{proof}
\end{lemma}
\subsection{Energy estimate}
Now we begin energy estimate for the system \eqref{governing_system} by testing equation $\eqref{governing_system}_1$ with $u$,
\begin{equation}\label{1st_test}
\begin{aligned}
\frac{1}{2} \frac{d}{dt} \int_{\Omega} |u|^2 dx + \int_{\Omega} \eta(\phi) |\nabla u|^2 dx &= - (\lambda \nabla \cdot (\nabla \phi \otimes \nabla \phi), u) \\
&\quad + (\nabla \cdot ({  \nu(\phi)}(FF^T - I)), u) - \left( \frac{\eta(\phi)}{\kappa(\phi)} (1 - \phi) u, u \right).
\end{aligned}
\end{equation}
{For the first term in the} right-hand side of $\eqref{governing_system}_1$, it can be expanded as
\begin{equation}\label{2_17}
\begin{aligned}
- \lambda \nabla \cdot (\nabla \phi \otimes \nabla \phi) &= -\lambda \Delta \phi \nabla \phi - \lambda \nabla \left( \frac{1}{2} |\nabla \phi|^2 \right).
\end{aligned}
\end{equation}
And $\eqref{governing_system}_5$ further gives
\begin{equation}\label{2_18}
= \mu \nabla \phi -\lambda \gamma f'(\phi) \nabla \phi - { \frac{ 1}{2} \nu'(\phi) }\text{tr}(FF^T - I) \nabla \phi  - \lambda \nabla \left( \frac{1}{2} |\nabla \phi|^2 \right).
\end{equation}
To simplify the divergence of the matrix form, if matrix $\Xi = (b)_{ij} ~(i,j = 1, 2, 3)~$ and vector $u = (u^1, u^2, u^3)$ are given,
\begin{equation}\label{eqn:mat}
    \begin{split}
        \int_\Omega div(\Xi ) u dx &=  \int_\Omega (\partial_1, \partial_2, \partial_3) 
	\begin{pmatrix} 
	b_{11} & b_{12} & b_{13} \\
	b_{21} & b_{22} & b_{23}\\
	b_{31} & b_{32} & b_{33} \\
	\end{pmatrix}
	\begin{pmatrix} 
	u^1 \\
	u^2 \\
	u^3 \\
	\end{pmatrix}
	 dx
    \\ ~~    &= 
     - \int_{\Omega}
     tr \left[
 	{ \begin{pmatrix} 
	\partial_1 u^1 & \partial_2 u^1 & \partial_3 u^1 \\
	\partial_1 u^2 & \partial_2 u^2 & \partial_3 u^2 \\
	\partial_1 u^3 & \partial_2 u^3 & \partial_3 u^3 \\
	\end{pmatrix}}
  	\begin{pmatrix} 
	b_{11} & b_{12} & b_{13} \\
	b_{21} & b_{22} & b_{23}\\
	b_{31} & b_{32} & b_{33} \\
	\end{pmatrix}
    \right]
    \\ ~~   &= - \int_\Omega  tr (\nabla u \cdot \Xi) dx
            .\\
    \end{split}
\end{equation}


By taking this form on a term in the equation \eqref{1st_test},
\begin{align*}
\int_{\Omega}  \text{div}({  \nu(\phi)}(FF^T - I))  u \, dx &= - \int_{\Omega} \text{tr}(\nabla u \cdot (FF^T - I){  \nu(\phi)}) \, dx \\
&= - \int_{\Omega} \text{tr}(\nabla u FF^T){  \nu(\phi)} \, dx +  \int_{\Omega} \text{tr}(\nabla u){  \nu(\phi)} \ dx.
\end{align*}

Using $\text{tr}(\nabla u) = \text{div} \, u = 0$ yields
\[ = - \int_{\Omega} \text{tr}(\nabla u FF^T){  \nu(\phi)} \, dx. \]

Therefore, we obtain
\begin{equation}\label{2_19}
\begin{aligned}
\frac{1}{2} \frac{d}{dt} &\int_{\Omega} |u|^2 dx + \int_{\Omega} \eta(\phi) |\nabla u|^2 dx 
\\
&= \int_{\Omega} \mu \nabla \phi \cdot u dx  - \lambda \gamma \int_{\Omega} f'(\phi) \nabla \phi \cdot u dx - \frac{1}{2} \int_{\Omega} {  \nu'(\phi)} \nabla \phi \text{tr}(FF^T-I) \cdot u \, dx  
\\
&  - \lambda \int_{\Omega} \nabla  \left( \frac{1}{2} |\nabla \phi|^2 \right)\cdot u dx -  \int_{\Omega} \text{tr}(\nabla u FF^T){  \nu(\phi)} \, dx - \int_{\Omega} \frac{\eta(\phi)}{\kappa(\phi)} (1 - \phi) |u|^2 \, dx
\\
&= \int_{\Omega} \mu \nabla \phi \cdot u dx  - \frac{1}{2} \int_{\Omega} {  \nu'(\phi)} \nabla \phi \text{tr}(FF^T-I) \cdot u \, dx  
\\
&-  \int_{\Omega} \text{tr}(\nabla u FF^T){  \nu(\phi)} \, dx - \int_{\Omega} \frac{\eta(\phi)}{\kappa(\phi)} (1 - \phi) |u|^2 \, dx.
\end{aligned}
\end{equation}
For the last equality in above equation, we used divergence-free condition.
\\
\par
For the next estimate, we test $\eqref{governing_system}_4$ with $\mu$ as
\begin{equation}\label{2_20}
\left( \phi_t, \mu \right) + (u \cdot \nabla \phi, \mu) + \tau |\nabla \mu|^2 = 0.
\end{equation}

To expand this form by the expression of $\mu$ in $\eqref{governing_system}_5$,
\begin{equation}\label{2_21}
\begin{aligned}
(u \cdot \nabla \phi, \mu) &= - \left( \phi_t, \mu \right) -\tau |\nabla \mu|^2 
\\
&= \int_{\Omega} \left( \lambda \Delta \phi - \lambda \gamma f'(\phi) - \frac{1}{2} {  \nu'(\phi)}\text{tr}(FF^T - I) \right) \phi_t \, dx  -\tau |\nabla \mu|^2.
\end{aligned}
\end{equation}
by using $\langle \phi_t \rangle_{\Omega} = 0$ in \eqref{average}.

Equalities \eqref{2_19}-\eqref{2_21} are combined as
\begin{equation}\label{2_22}
\begin{aligned}
&\frac{1}{2} \frac{d}{dt} \int_{\Omega} \left( |u|^2 + \lambda |\nabla \phi|^2 + 2 \lambda \gamma f(\phi) \right) dx + \int_{\Omega} (\eta(\phi) |\nabla u|^2 + \tau |\nabla \mu|^2) dx 
\\
&= - \frac{1}{2}  \int_{\Omega} {  \nu'(\phi)}\frac{d \phi}{dt} \text{tr}(FF^T - I) dx   - \frac{1}{2} \int_{\Omega} {  \nu'(\phi)} \nabla \phi \text{tr}(FF^T-I) \cdot u \, dx  
\\
& - \int_{\Omega} \text{tr}(\nabla u FF^T){  \nu (\phi)} dx - \int_{\Omega} \frac{\eta(\phi)}{\kappa(\phi)} (1 - \phi) |u|^2 dx.
\end{aligned}
\end{equation}
\\
\par
As the next test, take innerproduct between $\eqref{governing_system}_3$ and $F{ \nu(\phi)}$. This means that $\eqref{governing_system}_3$ is multiplied by $F^T{ \nu(\phi)}$, and then we take the trace and integrate this over $\Omega$.
\begin{equation}\label{2_23}
\begin{aligned}
\int_{\Omega} & \text{tr} \left( F_t F^T \right) \nu(\phi)dx + \int_{\Omega} \text{tr}((u \cdot \nabla F)F^T) \nu(\phi)dx 
\\
&- k \int_{\Omega}tr(({  \nu (\phi) \Delta F +2 \nabla \nu (\phi) \cdot \nabla F ) F^T \nu(\phi)} ) dx = \int_{\Omega} \text{tr}(\nabla u FF^T){ \nu(\phi)} dx.
\end{aligned}
\end{equation}
Let us estimate each term. The first term is
\begin{equation}\label{2_24}
\text{tr} \left( F_t F^T \right) = \frac{1}{2} \frac{d}{dt} \text{tr}(FF^T).
\end{equation}
To other term is written as
\[ \text{tr}((u \cdot \nabla F)F^T) = \frac{1}{2} u \cdot \nabla \text{tr}(FF^T) \]
and this gives
\begin{align}\label{2_25}
\int_{\Omega} \text{tr}((u \cdot \nabla F)F^T){ \nu(\phi)}  dx &= \frac{1}{2} \int_{\Omega} u \cdot \nabla \text{tr}(FF^T){ \nu(\phi)}  dx 
\end{align}
For the other term, let us observe $(i,j)-$ component of $F$ where $1 \leq i, j \leq d$ as \eqref{F_elementwise}. Then integration by part and \eqref{bdry} yields
\begin{align}\label{k_estimate}
- k \int_{\Omega}tr(({  \nu (\phi) \Delta F +2 \nabla \nu (\phi) \cdot \nabla F ) F^T \nu(\phi)} ) dx &=  k \int_{\Omega}{  \nu (\phi)} \sum_{i,j} (\nabla F^{i,j})^2 {  \nu(\phi)}  dx 
\end{align}
Using \eqref{2_24} - \eqref{k_estimate} in \eqref{2_23} gives,
\begin{equation}\label{2_26}
\begin{aligned}
&\frac{1}{2} \int_{\Omega} \left( \frac{\partial}{\partial t} \text{tr}(FF^T) \right){  \nu(\phi)} dx + \frac{1}{2} \int_{\Omega} u \cdot \nabla \text{tr}(FF^T) {  \nu(\phi)}  dx 
\\
&+ k \int_{\Omega}{  \nu (\phi)} \sum_{i,j} (\nabla F^{i,j})^2 {  \nu(\phi)}  dx 
= \int_{\Omega} \text{tr}(\nabla u FF^T){  \nu(\phi)} dx.
\end{aligned}
\end{equation}
In \eqref{2_22}, substituting the right-hand side of the equation \eqref{2_26} gives
\begin{align*}
&\frac{1}{2} \frac{d}{dt} \int_{\Omega} (|u|^2 + \lambda |\nabla \phi|^2 + 2 \lambda \gamma f(\phi)) dx + \int_{\Omega} (\eta(\phi) |\nabla u|^2 + \tau |\nabla \mu|^2) dx \\
&= -\frac{1}{2} \int_{\Omega} {  \nu'(\phi)} \phi_t \text{tr}(FF^T) dx   - \frac{1}{2} \int_{\Omega} {  \nu'(\phi)} \nabla \phi \text{tr}(FF^T-I) \cdot u \, dx   \\
&- \frac{{1}}{2} \int_{\Omega} \left( \frac{\partial}{\partial t} \text{tr}(FF^T-I) \right) {  \nu(\phi)} dx - \frac{1}{2} \int_{\Omega} u \cdot \nabla \text{tr}(FF^T) {  \nu(\phi)}  dx 
\\
&- k \int_{\Omega}{  \nu (\phi) \sum_{i,j} (\nabla F^{i,j})^2  \nu(\phi)}  dx - \int_{\Omega} \frac{\eta(\phi)}{\kappa(\phi)} (1 - \phi) |u|^2 dx.
\end{align*}
Hence, the following energy-dissipation identity holds:
\begin{align}
&\frac{1}{2} \frac{d}{dt} \int_{\Omega} |u|^2 + \lambda |\nabla \phi|^2 + 2 \lambda \gamma f(\phi)  + {  \nu(\phi)} tr(FF^T -I ) dx + \int_{\Omega} (\eta(\phi) |\nabla u|^2 + \tau |\nabla \mu|^2) dx
\\
&  =
-k \int_{\Omega}{  \nu^2(\phi) \sum_{i,j} (\nabla F^{i,j})^2 } dx - \int_{\Omega} \frac{\eta(\phi)}{\kappa(\phi)} (1 - \phi) u^2 dx.
\end{align}
{ This is equal to
\begin{equation}\label{2_27}
\begin{aligned}
&\frac{1}{2} \frac{d}{dt} \int_{\Omega} |u|^2 + \lambda |\nabla \phi|^2 + 2 \lambda \gamma f(\phi)  + { \nu(\phi)} tr(FF^T -I ) dx 
\\
&  =
- \int_{\Omega} (\eta(\phi) |\nabla u|^2 + \tau |\nabla \mu|^2) dx-k \int_{\Omega}{  \nu^2 (\phi)\sum_{i,j} (\nabla F^{i,j})^2 } dx - \int_{\Omega} \frac{\eta(\phi)}{\kappa(\phi)} (1 - \phi) u^2 dx 
\end{aligned}
\end{equation}
Note that we derived dissipative energy which is on left hand side with non-positive parts on the right hand side. We define the total energy by
\begin{equation}\label{Total_Energy}
    E(x,t) :=  \int_{\Omega} |u|^2 + \lambda |\nabla \phi|^2 + 2 \lambda \gamma f(\phi)  + { \nu(\phi)} tr(FF^T -I ) dx.
\end{equation}
We can decompose this total energy as kinetic energy $E_k(x,t)$, mixed energy $E_m(x,t)$ and elastic energy $E_e(x,t)$ as follows.
\begin{equation}\label{E_decom}
\begin{aligned}
    &E_k(x,t) :=  \int_{\Omega} |u|^2  dx,
    \\& 
    E_m(x,t) :=  \int_{\Omega} \lambda |\nabla \phi|^2 + 2 \lambda \gamma f(\phi) dx,
    \\& 
    E_e(x,t) :=  \int_{\Omega} { \nu(\phi)} tr(FF^T -I ) dx.
\end{aligned}
\end{equation}
\subsection{A priori estimate for well-posedness property}\label{Galerkin_def}
The previous energy estimate is not enough to discuss the well-posedness of the solution since we need additional functional regularity for $\partial_t$ of the solution variables $u, \phi, F$ to apply the Aubin-Lions compactness theorem as standard method.
Observe that \eqref{2_27} does not provide a closed estimate.
\\
\par
Before deriving higher-order estimates, we introduce the finite-dimensional spaces used in the standard Faedo–Galerkin approximation for the system \eqref{governing_system}. This will be estimate on the Faedo-Galerking scheme as standard method. Therefore, let us define orthonormal basis sets $(w_k)_{k=1}^{\infty}$ of $H$ which are the eigenvectors of the Stokes operator $A$, $(e_k)_{k=1}^{\infty}$ of $L^2(\Omega)^d$ which are the Neumann eigenvalues of the operator $-\Delta + I$ and $(M_n)_{n=1}^{\infty}$ of $L^2(\Omega)^{d \times d}$ which are the Neumann eigenfunctions of the Laplace operator.
\\
\par
Then, we can construct the $n$-dimensional subspaces of $H$, $L^2(\Omega)^d$ and $L^2(\Omega)^{d \times d}$ where $n \geq 1$.
\[
V_n^1 := \text{span}\{w_1, ..., w_n\}, \qquad
V_n^2 := \text{span}\{e_1, ..., e_n\}, \qquad
V_n^3 := \text{span}\{M_1, ..., M_n\}.
\]
Let $\mathbf{P}_n^1 : H \to V_n^1$, $\mathbf{P}_n^2 : L^2(\Omega) \to V_n^2$ and $\mathbf{P}_n^3 : L^2(\Omega)^{d \times d} \to V_n^3$ be the projections associated with $L^2$ inner product. In this section, we assume all the variables $(u, \phi, F)$ are in these corresponding $n$-dimensional subspaces abbreviating the subscript of $n$.
\\
\par
To begin higher-order estimates for $u, \phi, F, u_t, \phi_t$ and $F_t$, let us denote the following.
\begin{equation}\label{Z}
\mathcal{Z} := |\nabla u|^2 + |\nabla u_t|^2 + |\Delta \phi|^2 + |\nabla \phi_t|^2 + |\Delta F|^2+\sum_{i,j}|\nabla F^{i,j}|^2+| F|^2+\sum_{i,j}|\nabla F_t^{i,j}|^2+ |F_t|^2
\end{equation}

Next, we test $\eqref{governing_system}_1$ with $Au$.
\begin{equation}\label{u_estimate}
\begin{aligned}
&\frac{1}{2} \frac{d}{dt} \int_{\Omega} |\nabla u|^2 dx + (u \cdot \nabla u, Au) - (\text{div}(\eta(\phi) \nabla u), Au)
\\ &= - (\lambda \Delta \phi \nabla \phi, Au) 
 +  (\text{div}({  \nu(\phi)}(FF^T - I)), Au) - \left( \frac{\eta(\phi)}{\kappa(\phi)} (1 - \phi) u, Au \right).
\end{aligned}
\end{equation}
By using \eqref{2_17}, we rewrite the first term on the right hand side.
\\
\par
From the classical Stokes theory, there is $p^* \in L^2(0,T;H^1(\Omega))$ such that $\langle p^* \rangle_{\Omega} = 0$ and $(u, p^*)$ satisfies the equation
\begin{equation}\label{Au_decomposition}
-\Delta u + \nabla p^* = Au \text{~for~a.e.~} t.
\end{equation}
Moreover, there is $C > 0$ such that
\begin{equation}\label{p*}
|p^*|_1 \leq C |Au|.
\end{equation}
And, from {Proposition 1.2 in \cite{t84}}, we have
\begin{equation}\label{p*_uH1}
|p^*|_{L^2(\Omega)/\mathbb{R}} \leq |\nabla p|_{H^{-1}(\Omega)} \leq |Au|_{H^{-1}(\Omega)} \leq |u|_1.
\end{equation}
\\
\par
Then, we get the following inequality by using above inequalities.
\begin{align*}
- (\text{div}(\eta(\phi) \nabla u), Au) &= - (\eta(\phi) \Delta u, Au) - (\eta'(\phi) \nabla \phi \cdot \nabla u, Au) \\
&= (\eta(\phi) Au, Au) - (\eta(\phi) \nabla p^*, Au) - (\eta'(\phi) \nabla \phi \cdot \nabla u, Au) \\
&\geq \alpha |Au|^2 + (\eta'(\phi) \nabla \phi \, p^*, Au) - (\eta'(\phi) \nabla \phi \cdot \nabla u, Au).
\end{align*}
Sobolev embeddings $H^{1/2}(\Omega) \subset L^3(\Omega), H^1(\Omega) \subset L^6(\Omega)$ enable the following estimates.
\begin{align*}
\left| (\eta'(\phi) \nabla \phi \, p^*, Au) \right| &\leq C|Au| |p^*|^{1/2} |p^*|_1^{1/2} |\phi|_2 \\
&\leq C|Au|^{3/2} |u|_1^{1/2} |\phi|_2 \\
&\leq \frac{\alpha}{12} |Au|^2 + C |u|_1^2 |\phi|_2^4,
\end{align*}
\begin{align*}
\left| (\eta'(\phi) \nabla \phi \cdot \nabla u, Au) \right| &\leq C |\nabla \phi|_1 | u|_1^{1/2} | u|_2^{1/2} |Au| \\
&\leq \frac{\alpha}{12} |Au|^2 + C |\phi|_2^2 | u|_2^2.
\end{align*}
Additionally, from \eqref{Au_decomposition} and \eqref{p*},
\begin{align*}
|(u \cdot \nabla u, Au)| &\leq |u|_{L^6(\Omega)} |\nabla u|_{L^3(\Omega)} |Au| \\
&\leq |u|_1 |u|_1^{1/2} |u|_2^{1/2} |Au| \\
&\leq \frac{\alpha}{12} |Au|^2 + C |u|_1^6.
\end{align*}

Continuously, we estimate the other terms in similar way.
\begin{align*}
|\Delta \phi \nabla \phi, Au| &\leq C |\Delta \phi|_{L^3} |\nabla \phi|_{L^6} |Au| \\
&\leq \frac{\alpha}{12} |Au|^2 + C |\Delta \phi||\Delta \phi|_1 |\nabla \phi|_1^2 \\
&\leq \frac{\alpha}{12} |Au|^2 + \dfrac{\tau \lambda}{8\tilde{C}} |\Delta \phi|_1^2 + C| \phi|_2^6 \\
&\leq \frac{\alpha}{12} |Au|^2 + \frac{\tau \lambda}{8} |\Delta^2 \phi|^2 + C |\partial_n \Delta \phi|_{H^{\frac{1}{2}}(\partial \Omega)}^2 + C |\phi|_2^6.
\end{align*}

The above $\tilde{C} > 0$ comes from Lemma \ref{lemma_phi}. From $\eqref{governing_system}_5$ with the boundary conditions (1.2), we know that
\begin{align*}
\partial_n \mu &= -\lambda \partial_n \Delta \phi + \lambda \gamma f''(\phi) \partial_n \phi + \partial_n (\frac{\nu'(\phi)}{2\lambda} \text{tr}(FF^T)),
\end{align*}

and this will be
\[
\partial_n \Delta \phi =  \partial_n (\frac{\nu'(\phi)}{2\lambda}\text{tr}(FF^T)).
\]

A general {trace theorem in \cite{lm00}} yields
\begin{equation}\label{bdry_Delphi}
\begin{aligned}
&|\partial_n \Delta \phi|_{H^{\frac{1}{2}}(\partial \Omega)} 
=  |\partial_n (\frac{\nu'(\phi)}{2\lambda} \text{tr}(FF^T))|_{H^{\frac{1}{2}}(\partial \Omega)} 
\\
&\leq C |\nu'(\phi)\text{tr}(FF^T))|_2
\\
&\leq C (|\nu'(\phi) \text{tr}(FF^T)|+|\nu''(\phi)\nabla \phi \text{tr}(FF^T)|+|\nu'(\phi) \text{tr}(\sum_{i,j}\nabla F^{i,j}( F^{i,j})^T)|
\\
& +\left|\nu'''(\phi)\|\nabla \phi\|^2 \text{tr}(FF^T)\right| +|\nu''(\phi)\Delta \phi \text{tr}(FF^T)|
+\sum_{i,j}|\nu''(\phi)\nabla \phi \nabla F^{i,j}F^{i,j}|
\\
& +|\nu'(\phi) (\text{tr}(\Delta FF^T)+\sum_{i,j}(\nabla F^{i,j})^2|)
\\
&\leq C (1 + |\nabla \phi| + |\nabla \phi|^2 + |\Delta \phi|)|F|_2^2.
\end{aligned}
\end{equation}
We obtain the last inequality from \eqref{H2_phi} and generalized Poincaré theorem with \eqref{average_const}.
Combining the estimates and then we get
\begin{align*}
|(\Delta \phi \nabla \phi, Au)| &\leq \frac{\alpha}{12} |Au|^2 + \frac{\tau \lambda}{8} |\Delta^2 \phi|^2 + C (1 + |\nabla \phi| + |\nabla \phi|^2 + |\Delta \phi|)^2|F|_2^4+ C |\phi|_2^6.
\end{align*}

From $H^2(\Omega) \subset L^\infty(\Omega)$,
\begin{align*}
&|(\nabla \cdot ({  \nu(\phi)}(FF^T - I)), Au)| \leq |( {  \nu'(\phi)} \nabla \phi (FF^T - I), Au)| + |({  \nu(\phi)}\nabla \cdot (FF^T), Au)| \\
&\leq  |\nu'(\phi)|_{L^{\infty}} |{\nabla \phi}|_{L^{6}}(|F|_{L^{3}}^2 + C)|Au| + |\nu (\phi)|_{L^{\infty}}|\nabla \cdot (FF^T)| |Au|
\\
&\leq  |\nu'(\phi)|_{L^{\infty}} |{\nabla \phi}|_1(|F|_1^2 + C)|Au| + |\nu (\phi)|_{L^{\infty}}|F|_2^2 |Au|
\\
&\leq \frac{\alpha}{12} |Au|^2 + C |\phi|_2^2 + C |\phi|_2^2 |F|_1^4+ C|F|_2^4.
\end{align*}
Likewise,
\begin{align*}
\left| \left( \frac{\eta(\phi)}{\kappa(\phi)} (1 - \phi) u, Au \right) \right| &\lesssim \frac{\beta}{\alpha} |Au| |u| |\phi|_2 \\
&\leq \frac{\alpha}{12} |Au|^2 + C |u|^2 + C |u|^2 |\phi|_2^2.
\end{align*}

By combining all the estimates after applying \eqref{H2_phi}, the \eqref{u_estimate} can be bounded as
\begin{align*}
&\frac{1}{2} \frac{d}{dt} |\nabla u|^2 + \frac{\alpha}{2} |Au|^2 
\\
&\leq \frac{\tau \lambda}{8} |\Delta^2 \phi|^2 + C (1 + |u|_1^2 |\phi|_2^4 + |\phi|_2^2 | u|_2^2 + |u|_1^6+ |F|_2^4( 1+ |\nabla \phi| + |\nabla \phi|^2 + |\Delta \phi|)^2
\\
&+ |\Delta \phi|^6 + |\Delta \phi|^2 + |\Delta \phi|^2 |F|_1^4+ |F|_2^4 + |u|^2 + |u|^2 |\Delta \phi|^2 ).
\end{align*}

Therefore, we can organize the above with $\mathcal{Z}$ as defined in \eqref{Z}
\begin{equation}\label{2_31}
\frac{1}{2} \frac{d}{dt} |\nabla u|^2 + \alpha (1 - \frac{7}{16}) |Au|^2 \leq \frac{\tau \lambda}{12} |\Delta^2 \phi|^2 + C (1 + \mathcal{Z})^3.
\end{equation}

As next estimate, $\eqref{governing_system}_4$ is tested with $\Delta^2 \phi$. Using \eqref{2_18} and $\eqref{governing_system}_5$ brings
\begin{align*}
&\frac{1}{2} \frac{d}{dt} |\Delta \phi|^2 + \int_{\partial \Omega} \phi_t \partial_n \Delta \phi d\Gamma + (u \cdot \nabla \phi, \Delta^2 \phi) + \tau \lambda |\Delta^2 \phi|^2  \\
&=  \left(  \Delta (\frac{\nu'(\phi)}{2}\text{tr}(FF^T - I)), \Delta^2 \phi \right) + \tau \lambda \gamma \left( f''(\phi) \Delta \phi + f'''(\phi) \|\nabla \phi\|^2, \Delta^2 \phi \right).
\end{align*}
For the integration term on the boundary of domain, we use the boundary condition $\partial_n \phi_t = 0$ as written in \eqref{bdry}. Also employing the divergence-free condition of $u$ yields
\begin{align*}
\left| \int_{\partial \Omega} \phi_t \partial_n \Delta \phi d\Gamma \right|
&= \left| \int_{\Omega} \Delta (\phi_t \Delta \phi) dx \right| \\
&= \left| \int_{\Omega} \Delta \phi_t \Delta \phi + 2 \nabla \phi_t \cdot \nabla \Delta \phi + \phi_t \Delta^2 \phi dx \right| \\
&= \left| \int_{\Omega} \nabla \Delta \phi \cdot \nabla \phi_t + \phi_t \Delta^2 \phi dx \right| \\
&\leq C |\nabla \Delta \phi| |\nabla \phi_t| + |\phi_t| |\Delta^2 \phi| \\
&\leq \frac{\tau \lambda}{8} |\Delta^2 \phi|^2 + C (1 + |\nabla \phi|^2 + |\nabla \phi|^4 + |\Delta \phi|^2)|F|_2^ + C |\nabla \phi_t|^2 +C |\phi_t|^2.
\end{align*}
from Lemma \eqref{lemma_phi} and \eqref{bdry_Delphi}.
Additionally, Sobolev embedding is used again as
\begin{align*}
&|(u \cdot \nabla \phi, \Delta^2 \phi)| \leq \frac{\tau \lambda}{8} |\Delta^2 \phi|^2 + C |u|_2^2 |\phi|_2^2,
\\
&
\\
&|(\Delta ({\nu'(\phi)}\text{tr}(FF^T - I)), \Delta^2 \phi)|
\\
&\leq 
|(\Delta (\nu'(\phi)\text{tr}(FF^T )), \Delta^2 \phi)| 
+|(\Delta \nu'(\phi), \Delta^2 \phi)| \\
&\leq 
|
((\nu'''(\phi)\sum_{k}\phi_k^2  + \nu''(\phi)\sum_{k}\phi_{kk} )tr(FF^T), \Delta^2 \phi)
| \\ & 
+\sum_{k}|(
\nu''(\phi)\phi_ktr(2F_kF^T), \Delta^2 \phi
)|\\ & 
+|(
\nu'(\phi)tr(2\Delta FF^T + 2\sum_{k}F_kF_k^T), \Delta^2 \phi
)| + |(\Delta \nu'(\phi), \Delta^2 \phi)|
\\ 
& 
\leq
(|\nu'''(\phi)|_{L^{\infty}}|\nabla \phi| + |\nu''(\phi)|_{L^{\infty}}|\Delta \phi|)|tr(FF^T))|_{L^{\infty}} |\Delta^2 \phi|
 \\ & 
+
C|\nu''(\phi)|_{L^{\infty}}|\nabla \phi|_{L^{6}}|F|_{L^{6}}^2 |\Delta^2 \phi|
\\ & 
+C|(
|\nu'(\phi)|_{L^{\infty}}(|\Delta F||F|_{L^{\infty}} + |\nabla F|_{L^{3}}|\nabla F|_{L^{6}})|\Delta^2 \phi|
)| \\ &
+ |\Delta \nu'(\phi)|| \Delta^2 \phi| \\
&\leq
(|\nabla \phi|_1^2 + |\Delta \phi|)|F|_2^2|\Delta^2 \phi|
+
C|\nabla \phi|_1|F|_1^2 |\Delta^2 \phi|
\\ & 
+C
(|F|_2^2 + |\nabla F|_1|\nabla F|_1)|\Delta^2 \phi|
+ | \nu'(\phi)|_2| \Delta^2 \phi| \\ &
\leq 
\frac{\tau \lambda}{12} |\Delta^2 \phi|^2 +C(|\nabla \phi|_1^4 + |\phi|_2^2)|F|_2^2+
C|\nabla \phi|_1^2|F|_1^4 +C|F|_2^4 + C| F|_2^4+ | \nu'(\phi)|_2^2
\end{align*}
For the last term, we use \eqref{visco_diff} and thus get
\begin{align*}
\leq 
&\frac{\tau \lambda}{12} | \Delta^2 \phi|^2 +C(|\nabla \phi|_1^4 + |\phi|_2^2)|F|_2^2+
C|\nabla \phi|_1^2|F|_1^4 +C|F|_2^4 + C| F|_2^4\\
& + ( 1+ |\nabla \phi| + |\nabla \phi|^2 + |\Delta \phi|)^2
\end{align*}
As we defined $f$ as the double-well potential of 4th order polynomial, we estimate as
\begin{align*}
|(f''(\phi) \Delta \phi + f'''(\phi) |\nabla \phi|^2, \Delta^2 \phi)| &\lesssim |\phi|_{L^\infty(\Omega)} |\Delta \phi| |\Delta^2 \phi| + |\phi|_{L^\infty(\Omega)} |\nabla \phi|^2_{L^4(\Omega)} |\Delta^2 \phi| \\
&\lesssim |\phi|_2 |\Delta \phi| |\Delta^2 \phi| + |\phi|_2 |\nabla \phi|_1^2 |\Delta^2 \phi| \\
&\leq \frac{\tau \lambda}{12} |\Delta^2 \phi|^2 + C |\phi|_2^6.
\end{align*}

And, with $H^2(\Omega)$ norm bound as in \eqref{H2_phi}, combining the estimates above gives
\begin{equation}\label{2_32}
\frac{1}{2} \frac{d}{dt} |\Delta \phi|^2 + \tau \lambda(1-\frac{4}{12}) |\Delta^2 \phi|^2 \leq \frac{\tau \lambda}{12} |\nabla \Delta \phi_t|^2 + C(1 + \mathcal{Z})^3.
\end{equation}

Next, take $\partial_t$ to $\eqref{governing_system}_1$ and test with $u_t$ by using \eqref{2_17}.
\begin{align*}
&(u_{tt}, u_t) + (u_t \cdot \nabla u, u_t) + (\eta'(\phi) \phi_t \nabla u, \nabla u_t) + \int \eta(\phi) |\nabla u_t|^2 dx
\\&= -\lambda (\Delta \phi_t \nabla \phi + \Delta \phi \nabla \phi_t, u_t) - (\nu'(\phi) \phi_t  (FF^T - I) + \partial_t(FF^T)\nu(\phi), \nabla u_t) \\
&\quad - \left( \left( \frac{\eta'(\phi)}{\kappa(\phi)} - \frac{\eta(\phi)}{\kappa^2(\phi)} \kappa'(\phi) \right) \phi_t (1 - \phi) u - \eta(\phi) \phi_t u + \eta(\phi) (1 - \phi) u_t, u_t \right).
\end{align*}

For the terms above, we can estimate each term as follows.
\begin{align*}
|(u_t \cdot \nabla u, u_t)| &\lesssim |u_t|_{L^3(\Omega)} |\nabla u| |u_t|_{L^6(\Omega)} \\
&\lesssim |u_t|^{1/2} |u_t|_1^{3/2} |u|_1 \\
&\leq \frac{\alpha}{10} |u_t|_1^2 + C |u_t|^2 |u|_1^4.
\end{align*}
In a similar way, we can apply sobolev interpolation inequalities as
\begin{align*}
|(\eta'(\phi) \phi_t \nabla u, \nabla u_t)| &\lesssim |\eta'|_{L^\infty(\Omega)} |\phi_t|_{L^6(\Omega)} |\nabla u|_{L^3(\Omega)} |\nabla u_t| \\
&\lesssim |\eta'|_{L^\infty(\Omega)} |\phi_t|_1 |\nabla u|_1^{\frac 12} |\nabla u|^{\frac 12} |u_t|_1 \\
&\leq \frac{\alpha}{10} |u_t|_1^2 + \frac{\alpha}{4} |u|_2^2 + C |\phi_t|_1^4 |u|_1^2.
\end{align*}
Also,
\begin{align*}
|(\Delta \phi_t \nabla \phi + \Delta \phi \nabla \phi_t, u_t)| &\leq |(\Delta \phi_t, u_t \cdot \nabla \phi)| + |(\Delta \phi, u_t \cdot \nabla \phi_t)| \\
&\lesssim |\Delta \phi_t||u_t|_{L^3(\Omega)} |\nabla \phi_t|_{L^6(\Omega)} + |\nabla \phi_t||u_t|_{L^6(\Omega)} |\Delta \phi|_{L^3(\Omega)} \\
&\lesssim |\Delta \phi_t| |u_t|^{1/2}|\nabla u_t|^{1/2} |\phi|_2 + |\nabla \phi_t| |u_t|_1|\Delta \phi|^{1/2} |\Delta \phi|_1^{1/2}.
\end{align*}
We get following estimate continuously by using the generalized Poincaré inequality with $\langle \Delta \phi_t \rangle_\Omega = 0$,
\begin{align*}
&\leq \frac{\tau \lambda}{8} |\nabla \Delta \phi_t|^2 + C |u_t||u_t|_1 |\phi|_2^2 + \frac{\tau \lambda}{8\tilde{C}} |\Delta \phi|_1 + C |\nabla \phi_t|^2 |u_t|_1^2 |\Delta \phi| \\
&\leq \frac{\tau \lambda}{8} |\nabla \Delta \phi_t|^2 + \frac{\alpha}{10} |u_t|_1^2 + C |u_t|^2 |\phi|_2^4 + \frac{\tau \lambda}{8\tilde{C}} |\phi|_{H^4(\Omega)/R}^2 + C |\nabla \phi_t|^4 |\Delta \phi|^2.
\end{align*}
Applying Lemma \ref{lemma_phi} and \eqref{bdry_Delphi}, we obtain
\begin{align*}
&\leq \frac{\tau \lambda}{8} |\nabla \Delta \phi_t|^2 + \frac{\alpha}{10} |u_t|_1^2 + \frac{\tau \lambda}{8} |\Delta^2 \phi|^2 + C |\partial_n \Delta \phi|_{\mathbf{H}^{\frac 12}(\partial \Omega)}^2+ C |u_t|^2 |\phi|_2^4 + C |\nabla \phi_t|^4 |\Delta \phi|^2 \\
&\leq \frac{\tau \lambda}{8} |\nabla \Delta \phi_t|^2 + \frac{\alpha}{10} |u_t|_1^2 + \frac{\tau \lambda}{8} |\Delta^2 \phi|^2 \\
&+ |F|_2^4+ |u_t|^2 |\phi|_2^4 + C |\nabla \phi_t|^4 |\Delta \phi|^2.
\end{align*}
Also from $H^2(\Omega) \subset L^\infty(\Omega)$, we get
\begin{align*}
&|(\nu'(\phi)\phi_t(FF^T - I) + \partial_t(FF^T)\nu (\phi), \nabla u_t)| \\
&\leq |\nu'(\phi)|_{L^{\infty}}|\phi_t| (|F|_2^2+C) |\nabla u_t| +2 |F|_{L^\infty(\Omega)} |F_t| |\nu(\phi)|_{L^{\infty}} |\nabla u_t| \\
&\leq \frac{\alpha}{10} |u_t|_1^2 + C |F|_2^4 |\phi_t|^2 +C |\phi_t|^2+ C |F|_2^2 |F_t|^2 .
\end{align*}
And,
\begin{align*}
&\left| \left( (\frac{\eta'(\phi)}{\kappa(\phi)} - \frac{\eta(\phi)}{\kappa^2(\phi)} \kappa'(\phi) ) \phi_t (1 - \phi) u - \eta(\phi) \phi_t u+ \eta(\phi)  (1 - \phi) u_t, u_t \right) \right| 
\\
&\lesssim |\phi|_2 |\phi_t|_1 |u_t| |u|_1 + |\phi|_2 |u_t|^2.
\end{align*}
\\
Combine the above bounds by applying \eqref{H2_phi} with \eqref{average_const} and poincare inequality on $|u_t|$ from \eqref{poincare_u}. Then
\begin{equation}\label{2_33_}
\frac{1}{2} \frac{d}{dt} |u_t|^2 + \alpha(1-\frac{4}{10}) |u_t|_1^2 \leq \frac{\alpha}{4} |u|_2^2 + \frac{\tau \lambda}{8} |\Delta^2 \phi|^2 + \frac{\tau \lambda}{8} |\nabla \Delta \phi_t|^2 + C (1 + \mathcal{Z})^3.
\end{equation}
Next, we apply $\partial_t$ to $\eqref{governing_system}_1$ with \eqref{2_17} and then test with $A u_t$.
\begin{align*}
&{ \frac{1}{2}|\nabla u_{t}|^2} + (u_t \cdot \nabla u, A u_t) + (\nabla \cdot (\eta'(\phi) \phi_t \nabla u), \Delta  u_t)  -(\nabla \cdot (\eta(\phi)  \nabla u_t), A  u_t)
\\& = - \lambda (\Delta \phi_t \nabla \phi + \Delta \phi \nabla \phi_t, A u_t) + (\nabla \cdot (\nu'(\phi) \phi_t  (FF^T - I) + \partial_t(FF^T)\nu(\phi)), A  u_t) \\
&\quad + \left( \left( \frac{\eta'(\phi)}{\kappa(\phi)} - \frac{\eta(\phi)}{\kappa^2(\phi)} \kappa'(\phi) \right) \phi_t (1 - \phi) u - \eta(\phi) \phi_t u + \eta(\phi) (1 - \phi) u_t, A u_t \right).
\end{align*}
Each terms are estimated as follows. Using \eqref{Au_decomposition} and \eqref{p*} gives
\begin{align*}
- (\text{div}(\eta(\phi) \nabla u_t), Au_t) &= - (\eta(\phi) \Delta u_t, Au_t) - (\eta'(\phi) \nabla \phi \cdot \nabla u_t, Au_t) \\
&= (\eta(\phi) Au_t, Au_t) - (\eta(\phi) \nabla p^*, Au_t) - (\eta'(\phi) \nabla \phi \cdot \nabla u_t, Au_t) \\
&\geq \alpha |Au_t|^2 + (\eta'(\phi)  \nabla p^*, Au_t) - (\eta'(\phi) \nabla \phi \cdot \nabla u_t, Au_t).
\end{align*}
And, applying \eqref{p*_uH1} yields
\begin{align*}
\left| (\eta(\phi)  \nabla p_t^*, Au_t) \right| &\leq |\eta(\phi) |_{L^\infty} |p_t^*|_1  |Au_t| \\
&\leq C |p_t^*|_1  |Au_t| \\
&\leq \frac{\alpha}{12} |Au_t|^2 + C |u_t|_1^2 ,
\end{align*}
\begin{align*}
\left| (\eta'(\phi) \nabla \phi \cdot \nabla u_t, Au_t) \right| &\leq C |\nabla \phi|_1 | u_t|_1^{1/2} | u_t|_2^{1/2} |Au_t| \\
&\leq \frac{\alpha}{14} |Au_t|^2 + C |\phi|_2^2 | u_t|_2^2.
\end{align*}
\begin{align*}
|(u_t \cdot \nabla u, \Delta u_t)| &\leq |u_t|_{L^3(\Omega)} |\nabla u|_{L^6(\Omega)} |\Delta u_t| \\
&\leq \frac{\alpha}{14} |A u_t|^2 + C |u_t|_1^2 |u|_3^2.
\end{align*}
In the similar way,
\begin{align*}
&|(\eta''(\phi) \nabla \phi \phi_t \nabla u + \eta'(\phi) \nabla \phi_t \nabla u+ \eta'(\phi)\phi_t \Delta u, A u_t)| 
\\
&\leq (
|\eta''|_{L^\infty(\Omega)} |\nabla \phi|_{L^6(\Omega)}|\phi_t|_{L^6(\Omega)} |\nabla u|_{L^6(\Omega)}
+|\eta'|_{L^\infty(\Omega)} |\nabla \phi_t| |u|_{L^\infty(\Omega)}
\\
&+|\eta'|_{L^\infty(\Omega)} |\phi_t|_{L^3(\Omega)} |u|_{L^6(\Omega)}
)|A u_t| \\
&\leq \frac{\alpha}{14} |A u_t|^2 + C(|\Delta \phi|^2|\nabla \phi_t|^2|u|_3^2 +|\nabla \phi_t|^2|u|_2^2 + |\nabla \phi_t|^2|u|_1^2).
\end{align*}
In the last equality, inequalities \eqref{H2_phi} and \eqref{phit_H1} were used.\\
And, for the formula below
\begin{equation}\label{est_Au}
    (\Delta \phi_t \nabla \phi + \Delta \phi \nabla \phi_t, { A }u_t) = -(\Delta \phi_t \nabla \phi + \Delta \phi \nabla \phi_t, { \Delta }u_t) +(\Delta \phi_t \nabla \phi + \Delta \phi \nabla \phi_t, p^*),
\end{equation}
detailed estimate with \eqref{p*_uH1} is
\begin{align*}
&|(\Delta \phi_t \nabla \phi + \Delta \phi \nabla \phi_t, p^*)| \leq |(\Delta \phi_t \nabla \phi, p^* )| + |(\Delta \phi \nabla \phi_t,p^* )| \\
&\leq 
(|\Delta \phi_t |_{L^3}|\nabla \phi| + |\Delta \phi |_{L^3}|\nabla \phi_t|) |p^*|_{L^6} + |\Delta \phi |_{L^3}|\nabla \phi_t||p^* |_{L^6} \\
&\leq 
\frac{\alpha}{14} |A u_t|^2 + |\Delta \phi_t ||\Delta \phi_t |_1|\nabla \phi|^2 + |\Delta \phi ||\Delta \phi |_1|\nabla \phi_t|^2.
\end{align*}
To estimate $|\Delta \phi_t |_1$ and $|\Delta \phi |_1$, we need to use \eqref{Delphit_H1} and \eqref{delphi_H2}. Based on these inequalities,
\begin{align*}
&
\frac{\alpha}{14} |A u_t|^2 + |\Delta \phi_t ||\Delta \phi_t |_1|\nabla \phi|^2 + |\Delta \phi ||\Delta \phi |_1|\nabla \phi_t|^2 \\
& \leq
\frac{\alpha}{14} |A u_t|^2 + \frac{\tau \lambda}{12} |\nabla \Delta \phi_t|^2 +\frac{\tau \lambda}{12} | \Delta^2 \phi|^2 + C(|\Delta \phi |^2|\nabla \phi_t|^4 + |\Delta \phi |^2|\nabla \phi_t|^4 ).
\end{align*}
The remaining term in \eqref{est_Au} can be estimated as follows.
\begin{align*}
&|(\Delta \phi_t \nabla \phi + \Delta \phi \nabla \phi_t, { \Delta }u_t)| \leq |(\nabla \cdot (\Delta \phi_t \nabla \phi),\nabla  u_t )| + |(\nabla \cdot (\Delta \phi \nabla \phi_t),\nabla  u_t )| \\
&\leq 
 |\nabla \Delta \phi_t||\nabla \phi|_{L^6(\Omega)}  |\nabla u_t|_{L^3(\Omega)}+| \Delta \phi_t|_{L^6(\Omega)}|\nabla^2 \phi|  |\nabla u_t|_{L^3(\Omega)} 
 \\ & + {  |\nabla \Delta \phi|_{L^6(\Omega)}|\nabla \phi_t|_{L^3(\Omega)}  |\nabla u_t|+| \Delta \phi|_{L^3(\Omega)}}|\nabla^2 \phi_t|_{L^6(\Omega)} |\nabla u_t|
 \\
&\leq \frac{\tau \lambda}{12} |\nabla \Delta \phi_t|^2 +\frac{\tau \lambda}{14} |A u_t|^2 + C(
|\Delta \phi|^4|\nabla u_t|^2 + 
|\Delta \phi|^4|\nabla u_t|^2 
\\ &
+ {  |\nabla \Delta \phi|_{L^6(\Omega)}|\nabla \phi_t|_{L^3(\Omega)}  |\nabla u_t|+| \Delta \phi|_{L^3(\Omega)}}|\nabla^2 \phi_t|_{L^6(\Omega)} |\nabla u_t|
\end{align*}
The inequality $|\nabla u_t |_1 \leq |A u_t|$ holds from \eqref{u_H2} and this was used above. And we used \eqref{Delphit_H1} and \eqref{H2_phi} for the terms regarding $\phi$.\\
Apply Lemma \ref{lemma_phi} to obtain
\begin{align}\label{delphi_H2}
& \frac{\tau \lambda}{12\tilde{C}} |\Delta \phi|_2^2
\leq \frac{\tau \lambda}{12} (| \Delta^2 \phi|^2 + |\int_{\partial \Omega}\Delta \phi d \Gamma|^2)
\\& \leq \frac{\tau \lambda}{12} (| \Delta^2 \phi|^2 + |\int_{\partial \Omega} \Delta \phi d \Gamma|^2)
\\
&\leq \frac{\tau \lambda}{12} (| \Delta^2 \phi|^2 + |\partial_n \frac{\nu'(\phi)}{2} tr(FF^T -I)|_{H^{\frac 12}(\partial \Omega)})
\\
& \leq \frac{\tau \lambda}{12} | \Delta^2 \phi|^2 + C (1 + |\nabla \phi| + |\nabla \phi|^2 + |\Delta \phi|)^2|F|_2^4 )
\end{align}
by using \eqref{bdry_Delphi}. On this inequality, we close the above inequality continuously
\begin{align*}
&\frac{\tau \lambda}{12} |\nabla \Delta \phi_t|^2 +\frac{\tau \lambda}{12} |A u_t|^2 + C(
|\Delta \phi|^4|\nabla u_t|^2 + 
|\Delta \phi|^2|\nabla u_t|| u_t| 
\\ &
+ {  |\nabla \Delta \phi|_{L^6(\Omega)}|\nabla \phi_t|_{L^3(\Omega)}  |\nabla u_t|+| \Delta \phi|_{L^3(\Omega)}}|\nabla^2 \phi_t|_{L^6(\Omega)} |\nabla u_t|\\
&\leq \frac{\tau \lambda}{12} |\nabla \Delta \phi_t|^2 +\frac{\tau \lambda}{10} |A u_t|^2 + \frac{\tau \lambda}{12} | \Delta^2 \phi|^2 + C(
|\Delta \phi|^4|\nabla u_t|^2 + 
|\Delta \phi|^4|\nabla u_t|^2
\\ & + C((1 + |\nabla \phi| + |\nabla \phi|^2 + |\Delta \phi|)^2|F|_2^4)+|\nabla \phi_t|^2|\nabla u_t|^4) 
\\& +C(1 + |\nabla \phi| + |\nabla \phi|^2 + |\Delta \phi|)^2|F|_2^4 + |\Delta \phi|^2|\nabla u_t|^4)
\end{align*}
Above, we used \eqref{phit_H1} and $\langle \Delta \phi_t \rangle = 0$ for generalized poincare inequality to derive $|\nabla \phi_t|_1 \leq |\Delta \phi_t| \leq |\nabla \Delta \phi_t|$.
\\
\par
From $H^2(\Omega) \subset L^\infty(\Omega)$, we get
\begin{align*}
&|(\nu'(\phi)\phi_t(FF^T - I) + \partial_t(FF^T)\nu (\phi), A u_t)| \\
&\leq |\nu'(\phi)|_{L^{\infty}}|\phi_t| (|F|_2^2+C) |A u_t| +2 |F|_{L^\infty(\Omega)} |F_t| |\nu(\phi)|_{L^{\infty}} |A u_t| \\
&\leq \frac{\alpha}{14} |A u_t|^2 + C |F|_2^4 |\phi_t|^2 +C |\phi_t|^2+ C |F|_2^2 |F_t|^2 .
\end{align*}
And,
\begin{align*}
&\left| \left( (\frac{\eta'(\phi)}{\kappa(\phi)} - \frac{\eta(\phi)}{\kappa^2(\phi)} \kappa'(\phi) ) \phi_t (1 - \phi) u - \eta(\phi) \phi_t u+ \eta(\phi)  (1 - \phi) u_t,A u_t \right) \right| 
\\
&\leq \frac{\tau \lambda}{14} | A u_t|^2 +C(|\phi_t|_1^2 (1+ |\phi|_1^2 ) |u|_1^2 + |\phi_t|_1^2 |u|_1^2 +(1+ |\phi|_2^2 )|u_t|^2).
\end{align*}
\\
Combine the above estimates with \eqref{H2_phi} and \eqref{average_const} to apply the Generalized poincare inequality.
\begin{equation}\label{2_33__}
\frac{1}{2} \frac{d}{dt} |\nabla u_t|^2 + \alpha(1-\frac{6}{14}) |A u_t|^2 \leq  \frac{\tau \lambda}{14} |\Delta^2 \phi|^2 + \frac{\tau \lambda}{12} |\nabla \Delta \phi_t|^2 + C (1 + \mathcal{Z})^3.
\end{equation}
To organize \eqref{2_33_} and \eqref{2_33__} until now, we summarize these as
\begin{equation}\label{2_33}
\frac{1}{2} \frac{d}{dt} (|u_t|^2 + |\nabla u_t|^2) + \alpha (1- \frac{4}{10} )| u_t|_1^2  +\alpha(1- \frac{6}{14} )|A u_t|^2\leq \frac{\alpha}{16} |u|_2^2 +  \frac{\tau \lambda}{12} |\Delta^2 \phi|^2 + \frac{\tau \lambda}{12} |\nabla \Delta \phi_t|^2 + C (1 + \mathcal{Z})^3.
\end{equation}
\\
\par
For the next estimate, we take $\partial_t$ to $\eqref{governing_system}_3$ and then test with $F_t$.
\begin{align*}
\int_\Omega \text{tr}(F_t F_t^T) dx -k \int_\Omega \text{tr}((\nu (\phi) \Delta F)_t F_t^T) dx -2k \int_\Omega \text{tr}(\nabla \nu (\phi) \cdot \nabla F)_t F_t^T) dx 
\\+ \int_\Omega \text{tr}((u \cdot \nabla F)_t F_t^T) dx = \int_\Omega \text{tr}((\nabla u F)_t F_t^T) dx.
\end{align*}
\\
From the fact $\int_\Omega \text{tr}((u \cdot \nabla F_t) F_t^T) dx = 0$,
\begin{align*}
\left| \int_\Omega \text{tr}(u_t \cdot \nabla F F_t^T) dx \right| &\leq C |u_t|_1 |F|_2 |F_t| \\
&\leq \frac{\alpha}{8} |u_t|_1^2 + C |F|_2^2 |F_t|^2.
\end{align*}
\\
From Neumann boundary condition on $F$ in \eqref{bdry},
\begin{align*}
& -k\int_\Omega \text{tr}((\nu (\phi) \Delta F)_t F_t^T) dx  =   -k\int_\Omega \text{tr}((\nu' (\phi)\phi_t \Delta F + \nu (\phi) \Delta F_t) F_t^T) dx 
\\& \geq   -\beta k \int_\Omega \text{tr}((\phi_t \Delta F ) F_t^T) dx - \beta k \int_\Omega  tr(\Delta F_t F_t^T)dx 
\\& \geq   -\beta k \int_\Omega \text{tr}((\phi_t \Delta F ) F_t^T) dx {  + \beta k \int_\Omega \sum_{i,j= 1,2} (\nabla F_t^{ij})^2 dx }
\end{align*}
Estimate the other term as
\begin{align*}
\left| \beta k\int_\Omega \text{tr}((\phi_t \Delta F ) F_t^T) dx \right| &\leq \beta k |\phi_t|_{L^3{\Omega}} |\Delta F|| F_t|_{L^6(\Omega)}
\\
&\leq \frac{\beta k}{4} \int_\Omega \sum_{i,j= 1,2} (\nabla F_t^{ij})^2 dx + C (|\nabla \phi_t|^2|\Delta F|^2),
\end{align*}
where the last inequality was driven by \eqref{phit_H1}.
\\
And for the other terms in the test,
\begin{align*}
&\left|2k \int_\Omega \text{tr}(\nabla \nu (\phi) \cdot \nabla F)_t F_t^T) dx \right|
= \left| 2k\int_\Omega \text{tr}(( \nu' (\phi)\nabla \phi \cdot \nabla F)_t F_t^T) dx \right| 
\\&
= \left| 2k\int_\Omega \text{tr}(( \nu'' (\phi) \phi_t \nabla \phi \cdot \nabla F+\nu' (\phi) \nabla \phi_t \cdot \nabla F+\nu' (\phi)  \nabla \phi \cdot \nabla F_t) F_t^T) dx \right| 
\\&
\leq 2 \beta k(| \phi_t |_{L^6}|\nabla\phi  |_{L^6}\sum_{i,j}|\nabla F^{ij}|_{L^6}+ |\nabla \phi_t |_{L^3}\sum_{i,j}| \nabla F^{ij}|_{L^6}) |F_t|
\\&+\beta k|\int_{\Omega}\Delta \phi tr (F_t F_t^T)dx|
\\&
\leq 2 \beta k(|\nabla \phi_t ||\Delta \phi  || F|_2+ |\nabla \phi_t |^{\frac{1}{2}}|\Delta \phi_t |^{\frac{1}{2}}| F|_2) |F_t|+\beta k|\int_{\Omega}\Delta \phi \text{tr} (F_t F_t^T)dx|
\\&\leq  \frac{\beta k}{4} \int_\Omega \sum_{i,j= 1,2} (\nabla F_t^{ij})^2 dx +\frac{\tau \lambda}{8} |\nabla \Delta \phi_t|^2  + C (|\nabla \phi_t|^2|\Delta \phi|^2 +|F|_2^2|F_t|^2 + |\nabla \phi_t|^2 + |F_2|^2|F_t|^2 
\\ & +|\Delta \phi|^4 |F_t|_1^2)
\end{align*}
from \eqref{gradphi_H1} for the variable $|\nabla \phi|_1$. Also we used \eqref{Delphi_H1} to estimate $|\Delta \phi_t|_1$.
\\
And, similarly,
\begin{align*}
\left| \int_\Omega \text{tr}((\nabla u_t F + \nabla u F_t) F_t^T) dx \right| &\leq |\nabla u_t| |F|_2 |F_t| + |u|_3 |F_t|^2 \\
&\leq \frac{\alpha}{8} |u_t|_1^2 + C |F|_2^2 |F_t|^2 + C |u|_3 |F_t|^2.
\end{align*}
\\
Thus we obtain
\begin{equation}\label{est_F_t_}
\begin{aligned}
&\frac{1}{2} \frac{d}{dt} |F_t|^2 + \frac{\beta k}{2} \int_\Omega \sum_{i,j= 1,2} (\nabla F_t^{ij})^2 dx
\leq \frac{\alpha}{4} |u_t|_1^2 + C|u|_3 |F_t|^2+ C \left( 1+Z\right)^3.
\end{aligned}
\end{equation}
\\
\par
As the next estimate, on $\eqref{governing_system}_3$, take $\partial_t$ and test with $- \Delta F_t$.
\begin{align*}
\int_\Omega \text{tr}(\nabla F_{tt} \nabla F_t^T) dx +k \int_\Omega \text{tr}((\nu (\phi) \Delta F)_t \Delta F_t^T) dx +2k \int_\Omega \text{tr}(\nabla \nu (\phi) \cdot \nabla F)_t \Delta F_t^T) dx 
\\- \int_\Omega \text{tr}((u \cdot \nabla F)_t \Delta F_t^T) dx = - \int_\Omega \text{tr}((\nabla u F)_t \Delta F_t^T) dx.
\end{align*}
\\
Also, for the other terms
\begin{align*}
\left| \int_\Omega \text{tr}((u \cdot \nabla F_t) \Delta F_t^T) dx \right| &\leq C |u|_3 |\nabla F_t| |\Delta F_t| \\
&\leq \frac{\alpha k }{10} |\Delta F_t|^2 + C |u|_3^2 |\nabla F_t|^2.
\end{align*}
\begin{align*}
\left| \int_\Omega \text{tr}(u_t \cdot \nabla F \Delta F_t^T) dx \right| &\leq C |u_t|^{\frac 12}|u_t|_1^{\frac 12} |F|_2 |\Delta F_t| \\
&\leq \frac{\alpha k}{10} |u_t|_1^2 +\frac{\alpha}{8} |\Delta F_t|^2 + C |F|_2^4 |u_t|^2.
\end{align*}
\\
With the Neumann boundary condition on $F$ as in \eqref{bdry},
\begin{align*}
& k\int_\Omega \text{tr}((\nu (\phi) \Delta F)_t \Delta F_t^T) dx  =   k\int_\Omega \text{tr}((\nu' (\phi)\phi_t \Delta F + \nu (\phi) \Delta F_t) \Delta F_t^T) dx 
\\& \geq   \alpha k \int_\Omega \text{tr}((\phi_t \Delta F ) \Delta F_t^T) dx + \alpha k |\Delta F_t|^2
\end{align*}
Estimate from the above term as
\begin{align*}
\left| \alpha k \int_\Omega \text{tr}((\phi_t \Delta F ) \Delta F_t^T) dx \right| &\leq \alpha k |\phi_t|_{L^6{\Omega}} |\Delta F|_{L^3(\Omega)}| \Delta F_t|\\
&\leq \alpha k |\phi_t|_1 |\Delta F|^{\frac 12}|\Delta F|_1^{\frac 12}| \Delta F_t|
\\
&\leq \frac{\alpha k}{10} |\Delta F_t|^2 + \frac{\alpha k}{2}|\nabla \Delta F^{i,j}|^2 + C | \nabla \phi_t|^4 |\Delta F|^2,
\end{align*}
where the last inequality was driven by \eqref{DelF_H1} and \eqref{phit_H1}.
\\
And for the other terms in the test,
\begin{align*}
&\left|2k \int_\Omega \text{tr}(\nabla \nu (\phi) \cdot \nabla F)_t \Delta F_t^T) dx \right|
= \left| 2k\int_\Omega \text{tr}(( \nu' (\phi)\nabla \phi \cdot \nabla F)_t \Delta F_t^T) dx \right| 
\\&
= \left| 2k\int_\Omega \text{tr}(( \nu'' (\phi) \phi_t \nabla \phi \cdot \nabla F+\nu' (\phi) \nabla \phi_t \cdot \nabla F+\nu' (\phi)  \nabla \phi \cdot \nabla F_t) \Delta F_t^T) dx \right| 
\\&
\leq 2 \beta k(| \phi_t |_{L^6}|\nabla\phi  |_{L^6}\sum_{i,j}|\nabla F^{ij}|_{L^6}+ |\nabla \phi_t |_{L^3}\sum_{i,j}| \nabla F^{ij}|_{L^6}) |\Delta F_t|
\\&+\beta k|\nabla \phi |_{L^6}|\nabla F_t |_{L^3}|\Delta F_t|
\\&
\leq 2 \beta k(|\nabla \phi_t ||\Delta \phi  || F|_2+ |\nabla \phi_t |^{\frac{1}{2}}|\Delta \phi_t |^{\frac{1}{2}}| F|_2) |\Delta F_t|+\beta k|\nabla \phi |_1|\nabla F_t |^{\frac 12}|\nabla F_t |_1^{\frac 12}|\Delta F_t|
\\&
\leq \frac{\alpha k}{10} |\Delta F_t|^2  + \frac{\beta k}{4} | \nabla \Delta \phi_t|^2  +C (|\nabla \phi_t|^2|\Delta \phi|^2|F|_2^2 + |\nabla \phi_t|^2|F|_2^4 + |\Delta \phi|^4|F_t|_1^2 ).
\end{align*}
from \eqref{gradphi_H1} for the variable $|\nabla \phi|_1$. Also we used \eqref{Delphi_H1} to estimate $|\Delta \phi_t|_1$ and  \eqref{DelF_H1} to estimate $|\Delta F_t|_1$.
\\
Similarly we also have,
\begin{align*}
\left| \int_\Omega \text{tr}((\nabla u_t F + \nabla u F_t) \Delta F_t^T) dx \right| &\leq { |\nabla u_t| |F|_2 |\Delta F_t|} + |u|_3 |F_t||\Delta F_t| \\
&\leq \frac{\alpha k}{10} |\Delta F_t|^2 + C (|\nabla u_t|^2 |F|_2^2 + C |u|_3^2 |F_t|c).
\end{align*}
\\
Thus we obtain
\begin{equation}\label{est_F_t__}
\begin{aligned}
&\frac{1}{2} \frac{d}{dt} |\nabla F_t|^2 + \frac{\alpha k}{2} |\Delta F_t|^2  
\\&
\leq  \frac{\beta k}{4} | \nabla \Delta \phi_t|^2 +  C|u|_3^2 |F_t|^2 +C(1+Z)^3.
\end{aligned}
\end{equation}
Summing up the estimates \eqref{est_F_t_} and \eqref{est_F_t__} enables
\begin{equation}\label{est_F_t}
\begin{aligned}
&\frac{1}{2} \frac{d}{dt} |F_t|^2 + \frac{d}{dt} |\nabla F_t|^2 + \frac{\beta k}{2} \int_\Omega \sum_{i,j= 1,2} (\nabla F_t^{ij})^2 dx
 +  \frac{\beta k}{2} |\Delta F_t|^2  
\\&
\leq \frac{\alpha}{10} |u_t|_1^2 
 + \frac{\tau \lambda}{12} | \nabla \Delta \phi_t|^2 +  C|u|_3^2 |F_t|^2 +C(1+Z)^3.
\end{aligned}
\end{equation}
\\
{Next,} we apply $\partial_t$ to $\eqref{governing_system}_4$, test it with $-\Delta \phi_t$ and then use the expression $\eqref{governing_system}_5$ for $\mu$ to obtain,
\begin{align*}
&\frac{1}{2} \frac{d}{dt} |\nabla \phi_t|^2 
-\frac{1}{2} \tau (
(\nu''(\phi) \nabla \phi_t + \nu'''(\phi)\phi_t \nabla \phi)\text{tr} (FF^T - I), \nabla \Delta \phi_t
)- (u_t \cdot \nabla \phi + u \cdot \nabla \phi_t, \Delta \phi_t) 
\\
&= \tau \lambda (\Delta^2\phi_t, \Delta \phi_t) + \frac{\nu'(\phi)}{2} \tau (\partial_t (\nabla \text{tr}(FF^T)), \nabla \Delta \phi_t) \\
&\quad - \tau \lambda \gamma (f''(\phi) \nabla \phi_t + f'''(\phi) \phi_t \nabla \phi, \nabla \Delta \phi_t).
\end{align*}
Above, we estimate the first term in the right hand side as
\begin{equation}
    \begin{aligned}
        &(\Delta^2\phi_t, \Delta \phi_t) = - |\nabla  \Delta \phi_t|^2 + \int_{\partial \Omega} \partial_n \Delta \phi_t \Delta \phi_t d\Gamma
    \end{aligned}
\end{equation}
and from $\eqref{governing_system}_5$
\begin{equation}
    \begin{aligned}
        &\int_{\partial \Omega} \partial_n \Delta \phi_t \Delta \phi_t d\Gamma
         = \int_{\partial \Omega} \partial_n (\Delta \phi_t )^2 d\Gamma 
        \\& = \frac{1}{\lambda}\int_{\partial \Omega} \partial_n (-\mu_t + \lambda \gamma f''(\phi)\phi_t - \nu''(\phi)\phi_t tr(FF^T-I)+2 \nu'(\phi) trF_tF^T)^2 d\Gamma\\&=0
    \end{aligned}
\end{equation}
from the Nuemann boundary condition on $\phi$, $F$ and $\mu$ in \eqref{bdry}.
\\
For the other partial term estimate,
\begin{equation}\label{2_58_}
    \begin{aligned}
        &-\frac{1}{2} \tau (
(\nu''(\phi) \nabla \phi_t + \nu'''(\phi)\phi_t \nabla \phi)\text{tr} (FF^T - I), \nabla   \Delta \phi_t
) 
\\&
\geq -\frac{\tau C}{2} |\nabla \phi_t \Delta \phi_t|^2+\frac{\tau \beta}{2} |\Delta \psi_t|^2
+  \frac{\tau \alpha}{2}(\phi_t \nabla \phi\text{tr} (FF^T - I), \nabla \Delta \phi_t).
    \end{aligned}
\end{equation}
On \eqref{2_58_}, the last terms are bounded as
\begin{equation}
    \begin{aligned}
        &|(\phi_t \nabla \phi\text{tr} (FF^T - I), \nabla \Delta \phi_t) |
        \leq |\phi_t|_1|\nabla \phi|_1|\text{tr} (FF^T - I)|_1|\nabla \Delta \phi_t|
        \\& \leq \frac{\tau \alpha}{24} |\nabla \Delta \phi_t|^2 + C(|\phi_t|_1^2 |\phi|_2^2 |F|_1^2|F|_2^2).
    \end{aligned}
\end{equation}
Using \eqref{average} along with the generalized Poincar\'e inequality, we obtain the following estimates on the terms appearing in the equation above.

\begin{align*}
|(u_t \cdot \nabla \phi + u \cdot \nabla \phi_t, \Delta \phi_t)| &\leq C |u_t| |\nabla \phi|_1 |\nabla \Delta \phi_t| + C |u|_1 |\nabla \phi_t| |\nabla \Delta \phi_t| \\
&\leq \frac{\tau \lambda}{12} |\nabla \Delta \phi_t|^2 + C (|u_t|^2 |\phi|_2^2 + |u|_1^2 |\nabla \phi_t|^2).
\end{align*}

The last inequality comes from Gagliardo-Nirenberg inequality(See \cite{bm18}).
Therefore, use this inequality to get
\begin{align*}
|(\partial_t (\nabla \text{tr}(FF^T)), \nabla \Delta \phi_t)| &\leq (|F_t|_1 |\nabla F|_1 + |F|_{L^\infty(\Omega)} |\nabla F_t|) |\nabla \Delta \phi_t| \\
&\leq \frac{\tau \lambda}{6} |\nabla \Delta \phi_t|^2 +  C {   |F_t|_1^2}|F|_2^2 .
\end{align*}
\\
\\
Next we use the fact that $f$ is a polynomial of degree 4 to observe,
\begin{align*}
&|(f''(\phi) \nabla \phi_t + f'''(\phi) \phi_t \nabla \phi, \nabla \Delta \phi_t)| \\
&\lesssim (|\phi|_{L^\infty(\Omega)}^2 + |\phi|_{L^\infty(\Omega)} + C )|\nabla \phi_t| |\nabla \Delta \phi_t| + (|\phi|_{L^\infty(\Omega)}+C) |\phi_t||\nabla \phi|_{L^3(\Omega)} |\nabla \Delta \phi_t| \\
&\leq \frac{\tau \lambda}{12} |\nabla \Delta \phi_t|^2 + C ((|\phi|_2^4 + |\phi|_2^2 + 1)|\nabla \phi_t|^2 +(|\phi|_2^2 + 1)|\phi|_2^2|\phi_t|^2).
\end{align*}
Hence, combining the above estimates and using \eqref{H2_phi} , generalized poincare theorem with \eqref{average_const} and Young's inequality, we obtain
\begin{equation}\label{2_34}
\begin{aligned}
&\frac{1}{2} \frac{d}{dt} |\nabla \phi_t|^2 + \tau \lambda(1-\frac{3}{12}) |\nabla \Delta \phi_t|^2 
\\
&\leq C(1+Z)^3
\end{aligned}
\end{equation}
\\
We next apply $\nabla$ to $\eqref{governing_system}_3$ and then test with $-\nabla \Delta F$. That is we consider,
\begin{equation}\label{Del_F_test}
\begin{aligned}
(\Delta F_t ,\Delta F) + &k (\nabla (\nu (\phi) \Delta F), \nabla \Delta F)& \\
&+ 2k (\nabla (\nabla \nu(\phi) \cdot \nabla F), \nabla \Delta F)- (\nabla (u \cdot \nabla F), \nabla \Delta F) 
\\
& = -(\nabla (\nabla u F), \nabla \Delta F) .
\end{aligned}
\end{equation}
Here, we can expand the terms with integration by part as
\begin{equation}\label{kalpha_1}
\begin{aligned}
 &k (\nabla (\nu (\phi) \nabla \Delta F^{ij}), \nabla \Delta F^{ij}) 
\\& =  k (\nu'(\phi) \nabla \phi \Delta F^{ij}, \nabla \Delta F^{ij}) + k(\nu (\phi) \nabla \Delta F^{ij}, \nabla \Delta F^{ij})
\end{aligned}
\end{equation}
for $1 \leq i, j \leq d$. \\
Above, we employed Lemma \ref{lem_neg_sobolev} the interpolation inequality and its constant $C_d$ for bound of $H^{-1/2}(\partial \Omega)$-norm. Also we used \eqref{H-1norm_Del} and Gagliardo-Nirenberg inequality (find text to refer) to bound the $H^{-1}(\Omega)$-norm.
\\
For the 2nd term in the right hand side of \eqref{kalpha_1}, from the bound of the $\nu$,
\begin{equation}\label{}
\begin{aligned}
& k ( \nu (\phi) \nabla \Delta F^{ij}, \Delta F^{ij}) \geq k \alpha ( \nabla \Delta  F^{ij}, \nabla \Delta F^{ij})
\end{aligned}
\end{equation}
Also for the 1st term in the right hand side of \eqref{kalpha_1},
\begin{equation}\label{}
\begin{aligned}
&k |( \nu' (\phi)\nabla \phi  \Delta F^{i,j}, \nabla \Delta F^{i,j})| \\
&\leq
k|\nu'(\phi)|_{L^{\infty}}|\nabla \phi|_{L^{6}(\Omega)} |\Delta F^{i,j}|_{L^{3}(\Omega)}|\nabla \Delta F^{i,j}| \\
&\leq \alpha k|\Delta \phi| |\Delta F^{i,j}|^{\frac 12}| \Delta F^{i,j}|_1^{\frac 12}|\nabla \Delta F^{i,j}| \\
&\leq \frac{\alpha k}{6}|\nabla \Delta F^{i,j}|^2 + |\Delta \phi|^2| \Delta F^{i,j}|
\end{aligned}
\end{equation}
for $1 \leq i, j \leq d$. We used \eqref{gradphi_H1} and \eqref{DelF_H1}.
\\
For another term in the \eqref{Del_F_test}, using \eqref{H-1norm_Del} yields
\begin{align}\label{2_50_}
&|2k (\nabla (\nabla \nu(\phi) \cdot \nabla F^{i,j}), \nabla \Delta F^{i,j}) |\leq 
C|(\nabla (\nu' (\phi) \nabla \phi \cdot \nabla F^{i,j}), \nabla \Delta F^{i,j}) |
\\&\leq C|( (\nu'' (\phi) \nabla \phi \cdot \nabla F^{i,j}), \nabla \Delta F^{i,j}) | + C|( (\nu' (\phi) \Delta \phi \cdot \nabla F^{i,j}), \nabla \Delta F^{i,j}) |
\\&\leq C|\nabla \phi |_1| \nabla F^{i,j}|_1|\nabla \Delta F^{i,j} | + C|\Delta \phi |^{\frac 12}|\Delta \phi |_1^{\frac 12}|\nabla F^{i,j}|_1| \nabla \Delta F^{i,j} |
\\&\leq \frac{\alpha k}{6} \sum_{i,j}|\nabla \Delta F^{ij}|^2 + \dfrac{\tau \lambda}{14}|\Delta^2 \phi|+|\phi|_2^2|F|_2^2 + |\phi|_2^2|F|_2^4
\end{align}
by using \eqref{DelF_H1}.
\\

Thanks to the divergence free property of $u$, $\text{tr}( u \cdot \nabla \Delta F \Delta F^T) dx=0$ and thus
\begin{align*}
\left| \int_\Omega \text{tr}(\nabla (u \cdot \nabla F) \nabla \Delta F^T) dx \right| &= |\nabla u || \nabla F| |\nabla \Delta F| \\
&\leq C  \frac{\alpha k}{6} \sum_{i,j}|\nabla \Delta F^{ij}|^2 +|u|_1^2|F|_1^2
\end{align*}
And, 
\begin{align*}
\left| \int_\Omega \text{tr}(\nabla (\nabla u F) \nabla \Delta F^T) dx \right| &\leq C (|\nabla^2 u ||F|_2 |\nabla \Delta F|+ |\nabla u |_1|\nabla F|_1 |\nabla \Delta F|) \\
&\leq \frac{\alpha k}{6} \sum_{i,j}|\nabla \Delta F^{ij}|^2 +|u|_2^2|F|_2^2
\end{align*}

We used \eqref{H2_phi} for the second inequality above and 
Since $\int_\Omega \text{tr}(\Delta F_t \Delta F^T) dx = \frac{1}{2} \frac{d}{dt} |\Delta F|^2$, we thus have,
\begin{equation}\label{2_37}
\begin{aligned}
\frac{1}{2} \frac{d}{dt} |\Delta F|^2 +\frac{\alpha k}{2} \sum_{i,j}|\nabla \Delta F^{i,j}|^2  \leq &C |u|_3 |F|_2^2+ + C(1+\mathcal{Z})^3.
\end{aligned}
\end{equation}
\\
\par
{  In similar way, for point-wise component, apply $\nabla$ to $\eqref{governing_system}_3$ component-wise and then test with $\nabla F^{ij}$ for $1 \leq i, j \leq d$.} We can write this down as
\begin{equation}\label{}
\begin{aligned}
(\nabla F_t^{i,j} ,\nabla F^{i,j}) - &k (\nabla (\nu (\phi) \Delta F^{i,j}),  \nabla F^{i,j})& \\
&- 2k (\nabla (\nabla \nu(\phi) \cdot \nabla F^{i,j}),  \nabla F^{i,j})+ (\nabla (u \cdot \nabla F^{i,j}), \nabla F^{i,j}) 
\\
& = (\nabla (\nabla u F)^{i,j}, \nabla F^{i,j}) .
\end{aligned}
\end{equation}
Estimates for each term are as follows.
\begin{equation*}
    \begin{aligned}
        &-k (\nabla (\nu (\phi) \Delta F^{i,j}),  \nabla F^{i,j}) = k |\nu (\phi) \Delta F^{i,j})|^2
        \\&
        \geq \alpha k |\Delta F^{i,j}|^2
    \end{aligned}
\end{equation*}
To estimate $|\nabla \phi|_1$, the inequality \eqref{gradphi_H1} was used. And,
\begin{equation*}
    \begin{aligned}
        &- 2k (\nabla (\nabla \nu(\phi) \cdot \nabla F^{i,j}),  \nabla F^{i,j}) 
        \\& = - 2k ( (\nabla^2 \nu(\phi)\cdot \nabla F^{i,j}),  \nabla F^{i,j}) - 2k (\nu'(\phi) \sum_{r}(\partial_r \phi \partial_r \nabla F^{i,j}),  \nabla F^{i,j})
        \\&
        \leq
        C |\nabla^2 \nu(\phi)||\nabla F^{i,j}|_{L^3}  |\nabla F^{i,j}|_{L^6} 
        + 
        C \sum_{r}|\nabla \phi |_{L^3}|\nabla^2 F^{i,j}||  \nabla F^{i,j}|_{L^6}
        \\&
        \leq
        C( |\Delta \phi|^4 + |\Delta \phi|^2 +|F|_2^4) +C(|\Delta \phi|^2 + |F|_2^4)
    \end{aligned}
\end{equation*}
Here, we used \eqref{gradphi_H1}. Also by using this inequality, we applied the fact that $|\nabla^2 \nu(\phi) |\leq \beta (|\nabla \phi \cdot \nabla \phi| + |\Delta \phi|  ) \leq C (|\Delta \phi |^2 + |\Delta \phi| )$ for the last inequality. 
\begin{equation*}
    \begin{aligned}
        & | (\nabla (u \cdot \nabla F)^{i,j}, \nabla F^{i,j}) |
        =   |( (u \cdot \nabla F)^{i,j}, \Delta F^{i,j})|
        \\&
        =   |( (u \cdot \nabla F^{i,j}), \Delta F^{i,j})|
        \\&
        \leq C( |u |_1^2 + | F^{i,j}|_2^4)
    \end{aligned}
\end{equation*}
For the right hand side term in the test,
\begin{equation*}
    \begin{aligned}
        & (\nabla (\nabla u F)^{i,j}, \nabla F^{i,j})
        = (\nabla \sum_{k}( \partial_k u^i F^{kj}), \nabla F^{i,j})
        \\&
        =   |( \sum_{k} \partial_k u^i F^{kj}, \Delta F^{i,j})|
        \\&
        \leq {  \frac{\alpha}{10}|u |_2^2} +C( |u |_1^2 + | F^{i,j}|_2^2 \sum_{k}| F^{k,j}|_1^2)
    \end{aligned}
\end{equation*}
By adding all the component together, we get
\begin{equation}\label{2_37'}
\begin{aligned}
\frac{1}{2} \frac{d}{dt} |\nabla F|^2 + \alpha k | \Delta F|^2 \leq {  \frac{\alpha}{10}|u|_2^2} &+ C(1+Z)^2.
\end{aligned}
\end{equation}

{  Similarly to the previous steps, test $\eqref{governing_system}_3$ with $ F$.} We can write this down as
\begin{equation}\label{}
\begin{aligned}
( F_t , F) - &k ( \nu (\phi) \Delta F,   F) - 2k ( \nabla \nu(\phi) \cdot \nabla F,   F)+ ( u \cdot \nabla F,  F) 
\\
& = ( \nabla u F,  F) .
\end{aligned}
\end{equation}
This follows the Estimates for each term as below.
\begin{equation*}
    \begin{aligned}
        & -k ( \nu (\phi) \Delta F,   F)
        \geq 
        -k \beta ( \Delta F,   F)
        \\&
        = \beta k \sum_{i,j}|\nabla F^{i,j}|^2
    \end{aligned}
\end{equation*}
and,
\begin{equation*}
    \begin{aligned}
        & |2k ( \nabla \nu(\phi) \cdot \nabla F,   F)|
        \leq 
        \beta k |\nabla \phi|\sum_{i,j}|\nabla F^{i,j}|_1|F|_1
        \\&
        \leq C( |\nabla \phi|^2 +| F |_2^4)
    \end{aligned}
\end{equation*}
Also,
\begin{equation*}
    \begin{aligned}
        & |( u \cdot \nabla F,  F)| + |( \nabla u F,  F)|
        \leq 
          |u|\sum_{i,j}|\nabla F^{i,j}||F|_{L^{\infty}} + |\nabla u||F|_{L^{3}}|F|_{L^{6}}
        \\&
        \leq C( |u|^2 +| F |_2^4 + |u|_1^2 +| F |_1^4)
    \end{aligned}
\end{equation*}
Therefore, we get
\begin{equation}\label{2_37''}
\begin{aligned}
\frac{1}{2} \frac{d}{dt} | F|^2 +{\beta k} \sum_{i,j}|\nabla F^{i,j}|^2  \leq &C(1+Z)^2.
\end{aligned}
\end{equation}
Combining the equation \eqref{2_37}, \eqref{2_37'} and \eqref{2_37''} allows
\begin{equation}\label{2_38}
\begin{aligned}
    &\frac{1}{2} \frac{d}{dt} (|\Delta F|^2+\sum_{i,j}|\nabla F^{i,j}|^2+| F|^2) +\frac{\alpha k}{4} \sum_{i,j}|\nabla \Delta F^{i,j}|^2+ \frac{\alpha k}{4} |\Delta F|^2+\frac{\beta k}{4} \sum_{i,j}|\nabla F^{i,j}|^2 
\\&\leq \frac{\alpha}{10} |u_t|_1^2 + C |u|_3 |F_t|^2 + C (1 + \mathcal{Z})^3.
\end{aligned}
\end{equation}
\par
Now we combine \eqref{2_31}, \eqref{2_32}, \eqref{2_33}, \eqref{est_F_t}, \eqref{2_34} and \eqref{2_38} to obtain,
\begin{equation}\label{2_39}
\begin{aligned}
&\frac{1}{2} \frac{d}{dt} \left( |\nabla u|^2 + |\nabla u_t|^2 + |\Delta \phi|^2 + |\nabla \phi_t|^2 + |\Delta F|^2+\sum_{i,j}|\nabla F^{i,j}|^2+| F|^2+\sum_{i,j}|\nabla F_t^{i,j}|^2+ |F_t|^2 \right) 
\\
&+ \frac{\alpha}{2} |u|_2^2 + \frac{\alpha}{2} | u_t|_2^2 + \frac{\tau \lambda}{2} |\Delta^2 \phi|^2 + \frac{\tau \lambda}{2} |\nabla \Delta \phi_t|^2
\\&
{  +\frac{\alpha k}{4} \sum_{i,j}|\nabla \Delta F^{i,j}|^2+ \frac{\alpha k}{4} |\Delta F|^2+\frac{\beta k}{4} \sum_{i,j}|\nabla F^{i,j}|^2 
+ \frac{\beta k}{2} \int_\Omega \sum_{i,j= 1,2} (\nabla F_t^{ij})^2 dx
 +  \frac{\beta k}{4} |\Delta F_t|^2 }
\\
&\leq C |u|_3{  (|F|_2^2 + |F_t|^2)} + C(1 + Z)^3 
\end{aligned}
\end{equation}
\\
We used \eqref{u_H2} and \eqref{Au_decomposition} - \eqref{p*} here for obtaining $|u|_2^2$ term on the left hand side above.
For the formulas, with $\mathcal{Z}$ as defined in (2.29), we have
\begin{equation}\label{2_40}
\frac{1}{2} \frac{d}{dt} \mathcal{Z} + \mathcal{M} \leq C_0 |u|_3 \mathcal{Z} + C (1 + \mathcal{Z})^3, \tag{2.40}
\end{equation}
where $C_0 > 0$, and
\begin{equation}
\begin{aligned}
\mathcal{M} &:= \frac{\alpha}{2} |u|_2^2 + \frac{\alpha}{2} | u_t|_2^2 + \frac{\tau \lambda}{2} |\Delta^2 \phi|^2 + \frac{\tau \lambda}{2} |\nabla \Delta \phi_t|^2
\\&
{  +\frac{\alpha k}{4} \sum_{i,j}|\nabla \Delta F^{i,j}|^2+ \frac{\alpha k}{4} |\Delta F|^2+\frac{\beta k}{4} \sum_{i,j}|\nabla F^{i,j}|^2 
+ \frac{\beta k}{2} \int_\Omega \sum_{i,j= 1,2} (\nabla F_t^{ij})^2 dx
 +  \frac{\beta k}{4} |\Delta F_t|^2 }.
\end{aligned}
\end{equation}
\\
Now, we will close the result of the a priori estimate \eqref{2_40} with terms of $Z$, there remains making upper bound for the term $|u|_3$. This will be done by using Lemma \ref{lemma_u} on equation \eqref{governing_system}$_1$ as follows.
\begin{align*}
&|u|_3 + \left| \left( \frac{p}{\eta(\phi)} \right) \right|_1 \leq C \left[ 1 + (1 + |\phi|_2^2) \left( |\phi|_2^2 + |\phi|_2^{1/2} |\phi|_3^{1/2} \right) \right] 
( |u_t|_1 + |u \cdot \nabla u|_1 
\\
&\quad + |\Delta \phi \nabla \phi|_1 +|\nabla \| \nabla \phi\|^2|_1+ |\nabla \cdot (\lambda_e(\phi)(FF^T - I))|_1 
 + \left| \eta(\phi) \frac{(1 - \phi) u}{\kappa(\phi)} \right|_1 ).
\end{align*}
While applying the Lemma, we used \eqref{2_17} to get above.
\\
Each of the terms can be estimated as follows.
\begin{align*}
|u \cdot \nabla u|_1 
&\leq 
| u|_{L^6}|\nabla u|_{L^3} + |\nabla u|_{L^3}|\nabla u|_{L^6}+ |u|_{L^6}|\Delta u|_{L^3}
\leq |u|_1 | \nabla u|^{\frac{1}{2}}|| u|_2^{\frac{1}{2}}+ | \nabla u|^{\frac 12}| u|_2^{\frac 32}+ | u|_1| u|_2^{\frac{1}{2}} | u|_3^{\frac{1}{2}}.
\end{align*}
And, 
\begin{equation}
    \begin{aligned}
        &|\text{div}(\nu(\phi)(FF^T - I))|_1 \leq
        |\nu'(\phi)\nabla \phi (FF^T - I)|_1 + |\nu(\phi) \nabla (FF^T - I)|_1
        \\
        &\leq \beta (|\nabla \phi(FF^T - I)|+ |\nabla^2 \phi(FF^T - I)|+ |\nabla \phi \sum_{i,j}\nabla (F^{ij}F^{ij} )|) + \beta| \nabla FF^T|_1
        \\&
        \leq 
        C(|\nabla \phi|_{L^3}|FF^T -I|_{L^6} + |\nabla^2 \phi||FF^T-I|_{L^\infty} 
        + |\nabla \phi|_{L^\infty}|\sum_{i,j}\nabla (F^{ij}F^{ij} )| +|F|_2^2)
        \\&
        \leq 
        C(|\phi|_2|FF^T -I|_1 + | \phi|_2|FF^T-I|_2 
        + |\phi|_3|F|_2^2 +|F|_2^2)
        \\&
        \leq C(|\phi|_2(|F|_2^2+1) + | \phi|_2(|F|_2^2 +1) 
        + |\phi|_3|F|_2^2 +|F|_2^2).
    \end{aligned}
\end{equation}
Also using Galiardo-nirenberg inequality $|\phi|_3 \leq C |\phi|_2^{\frac 12} |\phi|_4^{\frac 12}$ enables
\begin{equation}
    \begin{aligned}
        |\Delta \phi \nabla \phi|_1 \leq C (|\Delta \phi|_{L^3} +|\nabla \Delta \phi|_{L^3} ) (|\nabla \phi|_{L^6} + |\nabla^2 \phi|_{L^6} )\leq C |\phi|_3^{\frac 12} |\phi|_4^{\frac 12}|\phi|_3
        \leq C |\phi|_2^{\frac 34} |\phi|_4^{\frac 54}
    \end{aligned}
\end{equation}
and, by using Galiardo-nirenberg again,
\begin{equation}
    \begin{aligned}
        &| \nabla \| \nabla \phi \|^2|_1 = |\nabla \phi \cdot \nabla^2 \phi|_1
        \\& 
        \leq
        (|\nabla \phi|_{L^6}+|\Delta \phi|_{L^6})(|\nabla^2 \phi|_{L^3}+|\nabla \cdot \nabla^2 \phi|_{L^3})
        \\&\leq
        C| \phi|_3|\phi|_4
        \\&\leq
        C| \phi|_2^{\frac 12}|\phi|_4^{\frac 32}.
    \end{aligned}
\end{equation}
Lastly,
\begin{equation}
    \begin{aligned}
        &\left| \eta(\phi) \frac{(1 - \phi) u}{\kappa(\phi)} \right|_1 
        \\&\leq
        \left|( \frac{\eta(\phi)}{\kappa(\phi)} + \nabla \frac{\eta(\phi)}{\kappa(\phi)})(1 - \phi) u \right| +  \left| \frac{\eta(\phi)}{\kappa(\phi)}\nabla ((1 - \phi) u )\right| + C |(1 - \phi) u|
        \\&\leq
        \left| \frac{\eta(\phi)}{\kappa(\phi)}(1 - \phi) u + \nabla \frac{\eta(\phi)}{\kappa(\phi)}(1 - \phi) u \right|+  \left| \frac{\eta(\phi)}{\kappa(\phi)}\nabla ((1 - \phi) u )\right|  + C |(1 - \phi) u|
        \\&\leq
        C|(1 - \phi)|_{L^3}| u |_{L^6}+ C |\nabla \phi|_{L^6} |1 - \phi|_{L^6} |u|_{L^6} + C(|\nabla \phi|_{L^3}| u|_{L^6}+|1- \phi|_{L^{\infty}}|\nabla u|)+ C |(1 - \phi) u|
        \\&
        \leq
        C(1+|\phi|_1)|u|_1 + C|\phi|_2(1+|\phi|_1)|u|_1 + C(1+|\phi|_2|u|_1) .
    \end{aligned}
\end{equation}
As a result on estimating $|u|_3$, after combining all terms above and using Gagliardo-Nirenberg inequality $|\phi|_3^{\frac 12} \leq |\phi|_2^{\frac 14}|\phi|_4^{\frac 14}$, we get
\begin{equation}\label{u3_est}
\begin{aligned}
|u|_3 &\leq C \Big[ 1 + (1 + |\phi|_2^2) (|\phi|_2^2 + |\phi|_2^{3/4} |\phi|_4^{1/4})\Big] \Big[|u_t|_1 +|u|_1 | \nabla u|^{\frac{1}{2}}|| u|_2^{\frac{1}{2}} { + | \nabla u|^{\frac 12}| u|_2^{\frac 32}}+ | u|_1| u|_2^{\frac{1}{2}} | u|_3^{\frac{1}{2}}
\\&+|\phi|_2(|F|_1^2+1) + | \phi|_2(|F|_2^2 +1) 
 + |\phi|_3|F|_2^2 +|F|_2^2 + { |\phi|_2^{\frac 34} |\phi|_4^{\frac 54} + | \phi|_2^{\frac 12}|\phi|_4^{\frac 32} }
 \\&+ (1+|\phi|_1)|u|_1 + |\phi|_2(1+|\phi|_1)|u|_1 + (1+|\phi|_2|u|_1)\Big]
\end{aligned}
\end{equation}
And in this inequality, we can separate $|u|_3$ in the right hand side as
\begin{equation}
\begin{aligned}
|u|_3 &\leq 
\frac{1}{2}| u|_3+C
| u|_1^2| u|_2 \Big[ 1 + (1 + |\phi|_2^4) (|\phi|_2^4 + |\phi|_2^{3/2} |\phi|_4^{1/2})\Big] 
\\&
+C \Big[ 1 + (1 + |\phi|_2^2) (|\phi|_2^2 + |\phi|_2^{3/4} |\phi|_4^{1/4})\Big] (|u_t|_1 +|u|_1 | \nabla u|^{\frac{1}{2}}|| u|_2^{\frac{1}{2}} { + | \nabla u|^{\frac 12}| u|_2^{\frac 32}}
\\&+|\phi|_2(|F|_1^2+1) + | \phi|_2(|F|_2^2 +1) 
 + |\phi|_3|F|_2^2 +|F|_2^2 + { |\phi|_2^{\frac 34} |\phi|_4^{\frac 54} + | \phi|_2^{\frac 12}|\phi|_4^{\frac 32} }
 \\&+ (1+|\phi|_1)|u|_1 + |\phi|_2(1+|\phi|_1)|u|_1 + (1+|\phi|_2|u|_1))
\end{aligned}
\end{equation}
Absorb the $|u|_3$ term of right hand side into the left hand side, and we will organize the right hand side of \eqref{2_39} with $\mathcal{Z}$ as
\begin{equation}\label{2_61}
\begin{aligned}
C_0 |u|_3 \mathcal{Z}^3 &\leq  \frac{\alpha}{8} |u|_2^2 + \frac{\alpha}{8} |u_t|_1^2 + \frac{\tau \lambda}{8 \widetilde{C}} |\phi|_4^2 + C(1 + \mathcal{Z})^{20}.
\end{aligned}
\end{equation}
With this, we continue the a priori estimate from the last inequality result \eqref{2_40}.
For the higher degree $\phi$ term on the right hand side, use Lemma \ref{lemma_phi} and \ref{bdry_Delphi} and then we can see
\begin{equation}\label{2_62}
\begin{aligned}
\frac{\tau \lambda}{8 \widetilde{C}} |\phi|_4^2 \leq \frac{\tau \lambda}{8} |\Delta^2 \phi|^2 + C (1 + |\nabla \phi| + |\nabla \phi|^2 + |\Delta \phi|)|F|_2^2.
\end{aligned}
\end{equation}
We now combine right above two inequalities \eqref{2_61} and \eqref{2_62} to apply them in \eqref{2_40}. Then the final form of the a priori estimate in this section is written as
\begin{equation}\label{2_45}
\frac{d}{dt} \mathcal{Z} + \mathcal{M} \leq C (1 + \mathcal{Z})^{20}.
\end{equation}
\\
\par
Next our aim is to close this inequality \eqref{2_45} with initial values for the solutions $u$, $\phi$ and $F$. For this, there are two bounds to be made for the left hand side. First, we focus on the first term of the left hand side by dropping the positive $\mathcal{M}$ in the inequality. Divide by $(1+\mathcal{Z})^{20}$ for both sides of \eqref{2_45} and integrate as
\begin{equation}\label{2_pde_integration}
\begin{aligned}
\int_{\mathcal{Z}(0)}^{\mathcal{Z}(t)} \frac{1}{(1+\mathcal{Z})^{20}}d\mathcal{Z} \leq \int_0^{t} C dt
\end{aligned}
\end{equation}
Then, with the assumption $t \leq \frac{1}{19C(1 + \mathcal{Z}(0))^{19}}$, the integration values are
\begin{equation}\label{2_65}
\begin{aligned}
(1 +  \mathcal{Z}(t))^{19} \leq \frac{1}{(1 + \mathcal{Z}(0))^{-19}-19Ct}
\end{aligned}
\end{equation}
Moreover, suppose that $t \leq \frac{2^{19}-1}{2^{19}} \cdot \frac{1}{19C(1 + \mathcal{Z}(0))^{19}}$ and this implies
\begin{equation}
\begin{aligned}
    &1 +  \mathcal{Z}(t) \leq \left( \frac{1}{(1 + \mathcal{Z}(0))^{-19}-19Ct} \right)^{1/19}
    \\&  \leq \left( \frac{1}{(1 + \mathcal{Z}(0))^{-19}-\frac{2^{19}-1}{2^{19}(1 + \mathcal{Z}(0))^{19}}} \right)^{1/19}
    \\& \leq 2 (1 + \mathcal{Z}(0))
\end{aligned}
\end{equation}
As a result, if $t \leq T_0$ where $T_0 := \frac{2^{19}-1}{2^{19} \cdot 19C(1 + \mathcal{Z}(0))^{19}}$, then
\begin{equation}\label{K1_0}
1 +  \mathcal{Z}(t) \leq 2 (1 + \mathcal{Z}(0)).
\end{equation}
\\
\par
Using \eqref{K1_0}, now, continue the estimate on the second term of the left hand side in \eqref{2_45}. Integrate whole inequality for $[0, T_0]$ and then we get
\begin{equation}\label{K2_0}
\begin{aligned}
\mathcal{Z}(t) + \int_0^{T_0} \mathcal{M}(t) dt &\leq 
 \mathcal{Z}(0)+ C ( 2(1 + \mathcal{Z}(0)))^{20}.
\end{aligned}
\end{equation}
Since Z(t) is positive, we can bound $\int_0^{T_0} \mathcal{M}(t) dt$ term with given initial value of $\mathcal{Z}$.
\\
\par
As we can see in two results in \eqref{K1_0} and \eqref{K2_0}, we can make them closed with the initial conditions as the upper bounds is composed of the terms of them.
With given initial conditions at $t = 0$ in the Theorem \ref{Thm_local_wp}, we now bound time-derivative terms in the $\mathcal{Z}(0)$. Take $\nabla$ on $\eqref{governing_system}_1$ and test with $\nabla u_t$ and substitute $t = 0$ as below.
\begin{equation}\label{u_t0}
\begin{aligned}
| u_t(0)|_1 &\leq |u_0 \cdot \nabla u_0|_1 + |\nabla p_0|_1 + |\nabla \eta(\phi_0) \cdot \nabla u_0|_1 
\\&+ C| \Delta u_0|_1 +C| \Delta \nabla \phi_0 |_1 + C| \nabla (\nabla \phi_0 \cdot \nabla \phi_0)|_1
\\&+ C |\nabla \nu(\phi_0) (F_0 F_0^T - I)|_1  + C | \nabla \cdot (F_0 F_0^T)|_1  + C | \frac{\eta(\phi)}{\kappa(\phi)}(1 - \phi_0) u_0 |_1
\\& \leq |u_0|_2|u_0|_3 + |u_0|_3 + |\nabla \eta(\phi_0) \cdot \nabla u_0|_1 
\\& + |u_0|_3 + |\phi_0|_4 + |\phi_0|_4^2 
\\&+ C |\nabla \nu(\phi_0) (F_0 F_0^T - I)|_1  + C | \nabla \cdot (F_0 F_0^T)|_1  + C | \frac{\eta(\phi)}{\kappa(\phi)}(1 - \phi_0) u_0 |_1. 
\end{aligned}
\end{equation}
We used \eqref{Au_decomposition} to estimate $|\nabla p_0|_1$. To estimate other terms in above, the following inequalities are needed.
\begin{equation}\label{u_t01}
    \begin{aligned}
        |\nabla \eta(\phi_0) \cdot \nabla u_0|_1 & \leq |\nabla \eta(\phi_0) |_{L^6}|\nabla u_0|_{L^3} + |\eta'(\phi_0) \nabla \phi_0|| \nabla u_0|_{L^{\infty}} + |\nabla \eta(\phi_0)|_{L^{\infty}}|\nabla^2 u_0|
        \\& \leq 
        |\phi_0|_2|u_0|_2 + C|\phi_0|_1|u_0|_3 + C|\phi_0|_3|u_0|_2,
    \end{aligned}
\end{equation}
and
\begin{equation}\label{u_t02}
    \begin{aligned}
        |\nabla \nu(\phi_0) (F_0 F_0^T - I)|_1  \leq &
        |\nabla \nu(\phi_0)||F_0 F_0^T - I|_{L^{\infty}} + (|\nu''(\phi_0) |_2|\nabla \phi_0|_1|\nabla \phi_0|_1 + |\nabla^2 \phi_0|_1 )|F_0 F_0^T - I|_1 
        \\&+ |\nu'(\phi_0) \nabla \phi_0 |_1\sum_{i,j}|\nabla F_0^{i,j} (F_0^{i,j})^T |_1
        \\& \leq
        C(
        |\phi_0|_1(|F_0|_2^2 +1) + (|\phi_0|_2|\phi_0|_2 + |\phi_0|_3 )(|F_0|_2^2 +1 ) 
        + |\phi_0 |_2 |F_0|_2^2
        ).
    \end{aligned}
\end{equation}
Also,
\begin{equation}\label{u_t03}
    \begin{aligned}
| \frac{\eta(\phi)}{\kappa(\phi)}(1 - \phi_0) u_0 |_1
&\leq | \frac{\eta(\phi)}{\kappa(\phi)}(1 - \phi_0) u_0 |
+ | \frac{\eta'(\phi_0)\nabla \phi_0 \kappa(\phi_0)- \eta(\phi_0)\kappa'(\phi_0) \nabla \phi_0}{\kappa^2(\phi)}||1 - \phi_0|_{L^{\infty}}| u_0 |_{L^{\infty}}
\\& + | \frac{\eta(\phi)}{\kappa(\phi)}|_{L^{\infty}}| \nabla \phi_0 |_1|u_0 |_1
+ | \frac{\eta(\phi)}{\kappa(\phi)}|_{L^{\infty}}|1 - \phi_0 |_1|\nabla u_0 |_1
\\&
\leq C(
(1 + |\phi_0|_1) |u_0|_1 + |\phi_0 |_1(1 + |\phi_0|_2)| u_0 |_2
 + | \phi_0 |_2|u_0 |_1
+ (1 + |\phi_0 |_1)| u_0 |_2
)
    \end{aligned}
\end{equation}
We now apply \eqref{u_t01} - \eqref{u_t03} to continue the estimate of \eqref{u_t0} as
\begin{equation}\label{u_t04}
\begin{aligned}
| u_t(0)|_1 &\leq  |u_0|_2|u_0|_3 + |u_0|_3 + |\phi_0|_2|u_0|_2 + C|\phi_0|_1|u_0|_3 + C|\phi_0|_3|u_0|_2
\\& + |u_0|_3 + |\phi_0|_4 + |\phi_0|_4^2 
\\&+ C (|\phi_0|_1(|F_0|_2^2 +1) + (|\phi_0|_2|\phi_0|_2 + |\phi_0|_3 )(|F_0|_2^2 +1 ) 
        + |\phi_0 |_2 |F_0|_2^2) 
\\& + C | F_0 |_2^2  
\\& + C(
(1 + |\phi_0|_1) |u_0|_1 + |\phi_0 |_1(1 + |\phi_0|_2)| u_0 |_2
 + | \phi_0 |_2|u_0 |_1
+ (1 + |\phi_0 |_1)| u_0 |_2
). 
\end{aligned}
\end{equation}
By applying \eqref{poincare_u}, this gives bounds for $|u_t(0)|$ and $|u_t(0)|_1$ at once with given conditions for the terms $u_0$, $\phi_0$ and $F_0$ in Theorem \ref{Thm_local_wp}.\\
\par
To bound $|\phi_t(0)|_1$, 
\begin{align*}
|\phi_t(0)|_1 &\leq |u_0 \cdot \nabla \phi_0|_1 + \tau(|\lambda \Delta^2 \phi_0|_1 + |\lambda \gamma \Delta f'(\phi)|_1 
+ |\frac{\nu (\phi)}{2} \text{tr}(FF^T-I)|_1
\\&\leq
C (|u_0|_1|\phi_0|_2 + |\phi_0|_5
+|\phi_0|_2^2+|\phi_0|_3
+|\phi_0|_2|F_0|_2^2
)
\end{align*}
Then we can bound $|\phi_t(0)|$ with given conditions for the terms $u_0$, $\phi_0$ and $F_0$ in Theorem \ref{Thm_local_wp}.\\
\\
Similarly, test $\eqref{governing_system}_3$ with $F_t$ and take $t = 0$ as
\begin{align*}
|F_t(0)|^2 &\leq |k(\nu(\phi)\Delta F_0, F_t(0))|+ |2k(\nabla \nu(\phi) \cdot \nabla F_0, F_t(0))| +|(u_0 \cdot \nabla F_0, F_t(0))| + |(\nabla u_0 F_0),F_t(0)| \\
&\leq \beta k |\Delta F_0||F_t(0)|+ 2k|\nabla \nu(\phi)|_1 | \nabla F_0|_1| F_t(0)| +|u_0 |_2| \nabla F_0|| F_t(0)| + |\nabla u_0|| F_0|_2|F_t(0)| \\
&\leq 
\frac{1}{2}|F_t(0)|^2 + C ( |F_0|_2^2+ |\phi|_1^2 |  F_0|_2^2 +|u_0 |_2^2|  F_0|_1^2 + | u_0|_1^2| F_0|_2^2 ).
\end{align*}
Move the $|F_t(0)|$ term to the left hand side in above, then $|F_t(0)|$  is bounded with given conditions for the terms $u_0$, $\phi_0$ and $F_0$ in Theorem \ref{Thm_local_wp}.\\
And, take $H^1(\Omega)$ norm on the governing system $\eqref{governing_system}_3$ and we can get the same bound results coming directly from Theorem \ref{Thm_local_wp}.
\begin{equation}\label{Ft0_est}
    \begin{aligned}
        |F_t(0)|_1 & \leq  k |\nu (\phi) \Delta F_0 |_1+2k |\nabla \nu (\phi) \cdot \nabla F_0 |_1 + |u \cdot \nabla F_0 |_1+ |\nabla u F_0|_1
        \\& \leq C( 
        |\phi |_3| F_0 |_3 + |\phi |_3| F_0 |_2 + |u |_2| F_0 |_2+ | u|_2| F_0|_2
        )
    \end{aligned}
\end{equation}
\par
With above bounds for $|u_t(0)|_1$, $|\phi_t(0)|_1$, $|F_t(0)|$ and $|F_t(0)|_1$, in inequalities \eqref{K1_0} and \eqref{K2_0}, the left hand sides are bounded with the constants which depend only on the initial data as set in Theorem \ref{Thm_local_wp}. Let us now summarize these results with the constants.
\\
Therefore, there exists coefficient $K_1 > 0$ s.t.
\begin{equation}\label{K_1}
|\nabla u|^2 + |\nabla u_t|^2 + |\Delta \phi|^2 + |\nabla \phi_t|^2 + |\Delta F|^2+\sum_{i,j}|\nabla F^{i,j}|^2+| F|^2+\sum_{i,j}|\nabla F_t^{i,j}|^2+ |F_t|^2 \leq K_1, \quad \forall t \in [0, T_0]
\end{equation}
where $K_1$ is constant depending only on given initial data in Theorem \ref{Thm_local_wp} and coefficients of Sobolev interpolation inequalities we used. 
Integrated functional in \eqref{K2_0} over $[0, T_0]$ with bound on $\mathcal{Z}(0)$ brings that there is constant $K_2 > 0$ s.t.
\begin{equation}\label{K_2}
\begin{aligned}
\int_0^{T_0} &\left( \right. \frac{\alpha}{2} |u|_2^2 + \frac{\alpha}{2} | u_t|_2^2 + \frac{\tau \lambda}{2} |\Delta^2 \phi|^2 + \frac{\tau \lambda}{2} |\nabla \Delta \phi_t|^2
\\&
+\frac{\alpha k}{4} \sum_{i,j}|\nabla \Delta F^{i,j}|^2+ \frac{\alpha k}{4} |\Delta F|^2+\frac{\beta k}{4} \sum_{i,j}|\nabla F^{i,j}|^2 
+ \frac{\beta k}{2} \int_\Omega \sum_{i,j= 1,2} (\nabla F_t^{ij})^2 dx
 +  \frac{\beta k}{4} |\Delta F_t|^2 \left. \right) dt 
 \\&
 \leq K_2.
 \end{aligned}
\end{equation}
where $K_2$ is depending only on the factors the $K_1$ depends.
\\
\par
Moreover, since the left hand side of \eqref{2_61} is bounded by $\mathcal{Z}$ and $\mathcal{M}$, applying \eqref{2_62} and \eqref{K_1} - \eqref{K_2} gives the constant $K_3 > 0$ which satisfies the following inequality by depending only on the the factors the $K_1$, $K_2$ depend.
\begin{equation}\label{K_3}
|u|_{L^2(0,T_0;H^3(\Omega))} \leq K_3.
\end{equation}

Let us finalize the two inequalities \eqref{K_1} - \eqref{K_3} as one inequality with a constant $K_4>0$ satisfying for any $t$ in $[0, T_0]$,
\begin{equation}\label{K_4}
\begin{aligned}
&|\nabla u|^2 + |\nabla u_t|^2 + |\Delta \phi|^2 + |\nabla \phi_t|^2 + |\Delta F|^2+\sum_{i,j}|\nabla F^{i,j}|^2+| F|^2+\sum_{i,j}|\nabla F_t^{i,j}|^2+ |F_t|^2 
\\&
    +\int_0^{T_0} \left( \right.
     \frac{\alpha}{2} |u|_3^2 + \frac{\alpha}{2} | u_t|_2^2 + \frac{\tau \lambda}{2} |\Delta^2 \phi|^2 + \frac{\tau \lambda}{2} |\nabla \Delta \phi_t|^2
\\&
 +\frac{\alpha k}{4} \sum_{i,j}|\nabla \Delta F^{i,j}|^2+ \frac{\alpha k}{4} |\Delta F|^2+\frac{\beta k}{4} \sum_{i,j}|\nabla F^{i,j}|^2 
+ \frac{\beta k}{2} \int_\Omega \sum_{i,j= 1,2} (\nabla F_t^{ij})^2 dx
 +  \frac{\beta k}{4} |\Delta F_t|^2 
   \left. \right) dx 
\\&\leq K_4.
\end{aligned}
\end{equation}
The constant $K_4$ depends only on the factors the $K_1$ depends.\\
Remember that the constants $K_i$ $(i = 1, 2, 3, 4)$ are independent on the dimension of the finite subspace in the Galerkin scheme. Let us develop theses finite subspaces by passing the dimension to limit in the next section.
\\
\par

\section{Galerkin scheme}\label{section_Galerkin}

From the foundation of definitions in chapter \ref{Galerkin_def} for the Galerkin Scheme, the existence of the strong solution of the \eqref{governing_system} will be proved. 
\\
\par
\textbf{Proof of existence of Theorem \ref{Thm_local_wp}.}
For $v \in V_n^1$, $\psi \in V_n^2$ and $\Xi \in V_n^3$, we construct the approximating equations of $u_n \in V^1_n$, $\phi_n \in V^2_n$ and $F_n \in V^3_n$ as
\begin{equation}\label{weakform}
\begin{array}{l}
\left\langle \frac{\partial u_n}{\partial t}, v \right\rangle + (u_n \cdot \nabla u_n, v) + (\eta(\phi_n) \nabla u_n, \nabla v) = (\lambda_e(1 - \phi_n)(F_n F_n^T - I), \nabla v) \\[5pt]
\quad + \lambda(\nabla \phi_n \otimes \nabla \phi_n, \nabla v) + \left(\eta(\phi_n) \frac{(1 - \phi_n)u_n}{\kappa(\phi_n)}, v\right), \\[5pt]
\left\langle \frac{\partial \phi_n}{\partial t}, \psi \right\rangle + (u_n \cdot \nabla \phi_n, \psi) + \tau(\nabla \mu_n, \nabla \psi) = 0, \\[5pt]
\left\langle \frac{\partial F_n}{\partial t}, \Xi \right\rangle +k( \nu(\phi_n) \Delta F_n -2 \nu(\phi_n) \cdot \nabla F_n, \Xi)+ (u_n \cdot \nabla F_n, \Xi) = (\nabla u_n F_n, \Xi).
\end{array}
\end{equation}
For the $\mu_n$, we approximate as writing on
\begin{equation*}
\mu_n := -\lambda \Delta \phi_n + \lambda \gamma f'(\phi_n) + \frac{\nu'(\phi_n)}{2} \text{tr}(F_n F_n^T - I). 
\end{equation*}

We prove the existence of $u_n$, $\phi_n$ and $F_n$ as we use the standard method, which solve \eqref{weakform} on time $0 < t_n \leq T$ for some chosen $T$.
\\
\par
{ Consider the following elements for each $n$ of these finite subspaces,
\begin{equation}\label{galerkin}
\begin{array}{l}
u_n = \sum_{i=1}^n a_{in}(t) \omega_i, \\[5pt]
\phi_n = \sum_{i=1}^n b_{in}(t) e_i, \\[5pt]
F_n = \sum_{i=1}^n c_{in}(t)M_i.
\end{array}
\end{equation}
where initial data $(u_n(0), \phi_n(0), F_n(0))$ is given as $lim_{n \rightarrow \infty}(u_n(0), \phi_n(0), F_n(0)) = (u_0, \phi_0, F_0) $ and $a$, $b$, $c$ are scalar function on the time domain [0,T]. When we insert $u_n, \phi_n, F_n$ in the \eqref{weakform}, the terms except $\partial_t$ in each equation are locally lipschitz because they are projected functions by $\mathbf{P}_n^1, \mathbf{P}_n^2$ and $\mathbf{P}_n^3$. Therefore, by Cauchy-Lipschitz theorem for ODE system, these are the approximating solution of the \eqref{weakform} in local time interval. \\
\\
Meanwhile, from the a priori estimate results \eqref{K_1}- \eqref{K_4}, together with the $L^2(\Omega)$-bound for $F$ obtained from the energy dissipation law of \eqref{Total_Energy}
\begin{equation}\label{bounds}
    \begin{split}
u_n &\text{ is bounded in } L^{\infty}(0, T_0; V) \cap L^2(0, T_0; H^3(\Omega)^d) \text{ independently of } n,\\
\phi_n &\text{ is bounded in } L^{\infty}(0, T_0; H^2(\Omega)) \cap L^2(0, T_0; H^4(\Omega)) \text{ independently of } n, \\
F_n &\text{ is bounded in } L^{\infty}(0, T_0; H^2(\Omega)^{d \times d})\cap L^2(0, T_0; H^3(\Omega)^{d \times d})\text{ independently of } n,
\end{split}
\end{equation}
and,
\begin{equation*}
\begin{aligned}
(u_n)_t &\text{ is bounded in } L^{\infty}(0, T_0; V) \cap L^2(0, T_0; D(A)) \text{ independently of } n, \\
(\phi_n)_t &\text{ is bounded in } L^{\infty}(0, T_0; H^1(\Omega)) \cap L^2(0, T_0; H^3(\Omega)^{d \times d}) \text{ independently of } n, \\
(F_n)_t &\text{ is bounded in } L^{\infty}(0, T_0; H^1(\Omega)^{d \times d})\cap L^2(0, T_0; H^2(\Omega)) \text{ independently of } n.
\end{aligned}
\end{equation*}
{ For the bounded sets in reflexive spaces $L^2(0,T_0; *)$ above, each closed ball has weak-converging subsequence. Also by Banach-Alaoglu theorem, the above uniform bounds in $L^{\infty} (0,T_0; *)$ implies weak-star sense convergences as follows.
$(u_n, \phi_n, F_n)$ in \eqref{bounds} converges to $(u, \phi, F)$, for some subsequence $n \to \infty$:}
\begin{equation}\label{aprioriresult2}
\begin{aligned}
u_n &\rightharpoonup u \text{ weak-* in } L^{\infty}(0, T_0; V) \text{ and weakly in } L^2(0, T_0; H^3(\Omega)^d ), \\
(u_n)_t &\rightharpoonup u_t \text{ weak-* in } L^{\infty}(0, T_0; V) \text{ and weakly in } L^2(0, T_0; D(A)), \\
u_n &\to u \text{ a.e. in }\Omega \times (0, T_0) \text{ and in }L^2([0,T];D(A)), 
\\
\phi_n &\rightharpoonup \phi \text{ weak-* in } L^{\infty}(0, T_0; H^2(\Omega)) \text{ and weakly in } L^2(0, T_0; H^4(\Omega)), \\
(\phi_n)_t &\rightharpoonup \phi_t \text{ weak-* in } L^{\infty}(0, T_0; H^1(\Omega)) \text{ and weakly in } L^2(0, T_0; H^3(\Omega)), \\
\phi_n &\to \phi \text{ a.e. in } \Omega \times (0, T_0) \text{ and in } L^2(0, T_0; H^{3 + \varepsilon}(\Omega)), \quad \forall \varepsilon \in [0, 1) , 
\\
F_n &\rightharpoonup F \text{ weak-* in } L^{\infty}(0, T_0; H^2(\Omega)^{d \times d}) \text{ and weakly in } L^2(0, T_0; H^3(\Omega)), \\
(F_n)_t &\rightharpoonup F_t \text{ weak-* in } L^{\infty}(0, T_0; H^1(\Omega)^{d \times d}) \text{ and weakly in } L^2(0, T_0; H^2(\Omega)), \\
F_n &\to F \text{ and a.e. in } \Omega \times (0, T_0) \text{ and in } L^2(0, T_0; H^{2 + \varepsilon}(\Omega)^{d \times d}), \quad \forall \varepsilon \in [0, 1)  .
\end{aligned}
\end{equation}
The Aubin-Lions compactness theorem was used here and the above space description $H^{2+\varepsilon}(\Omega)$ and $H^{3+\varepsilon}(\Omega)$ are denoted as interpolation space with $0 < \varepsilon < 1$(See \cite{t97}). To be specific, we describe the application of this theorem as follows.
\\\\
{ 
{Aubin-Lions compactness theorem} : $X_0 \subset X \subset X_1$ are Banach spaces. Suppose that $X_0$ is compactly embedded in $X$ and that $X$ is continuously embedded in $X_1$. For $1 \leq p,q \leq \infty,$ let
$$W = \{ u \in L^p ([0,T]; X_0) | u_t \in L^q([0,T];X_1)\}.$$
Then, the below statements follow.
\\
($i$) If $p < \infty $, then the embedding of $W$ into $L^p([0,T];X)$ is compact.\\
($ii$) If $p = \infty$ and $q > 1$, then the embedding of $W$ into $C([0,T];X)$ is compact.
\\
\par
Let us start from the assumptions we have as in \eqref{aprioriresult2}. From the bounds that $u_n \in L^2 (0,T_0;H^3(\Omega)^d \cap V)$ and $(u_n)_t  \in L^2 (0,T_0;V)$, by (i), the embedding of $W= \{u \in L^2([0, T];H^3(\Omega))\cap V| u_t \in L^2([0, T];D(A))\}$ with $p = 2, q = 2$ into $L^2 (0,T_0;D(A))$ is compact. Then $W$ is compact subset of $L^2([0,T];D(A))$. Hence, up to subsequence, there is $u_n \rightarrow u$ in $L^2 (0,T_0;D(A))$. This implies that there is a subsequence of $(u_n)$ such that a.e. $u_n \rightarrow u$ in $\Omega \times [0, T_0]$.
\\
\\
For the other lines to derive the \eqref{aprioriresult2} convergences, we can employ this theorem in the same way.
}
\\
\par
On \eqref{aprioriresult2} with a reference, using the Theorem 2.3 in reference \cite{lm72}, to obtain continuity in time results for $u, \phi, F$, we obtain the following:
\\Given two Hilbert spaces $X$ and $Y$, the space $W_{2,2} := \{v \in L^2(0,T;X), v_t \in L^2(0,T;Y)\}$ is continuously embedded in $C([0,T];[X,Y]_{\frac{1}{2}})$ with the definition $[X,Y]_{\frac{1}{2}}$ as the interpolation space of order $\frac{1}{2}$ of $X$ and $Y$. This yields
\begin{equation}\label{strong}
u \in C([0,T_0];D(A)), \quad \phi \in C([0,T_0]; H^3(\Omega)), \quad F \in C([0,T_0]; H^2(\Omega)^{d \times d}).
\end{equation}
Note that the norm of $D(A)$ is equivalent to that of $\mathbf{H}^2(\Omega)$ and $u$ takes values in $\mathbf{H}^2(\Omega)$.
\\
\par
Now, we will show the convergence of solution in the scheme \eqref{weakform} as $n$ goes to limit. Let us consider $\chi_i \in C^1(([0,T]),\mathbf{R})$ where $i = 1,2,3$ and $\chi_i(T) = 0 $. For the above system \eqref{weakform}, we multiply by $\chi_i = \chi_i(t) $, integrate over $[0,T]$ and obtain
\begin{equation}\label{eqn:sys.with.n.conv2}
	\left\{
    \begin{array}{ll}
	\int _0^T (  {u_n}(t), v \chi_1 '(t) ) dt + \int _0^T ( u_n(t) \cdot \nabla u_n  + \nabla p_n - \nabla \cdot (\eta({\phi_n }) \nabla u_n) , v \chi_1 (t) )   =(u_n(0), v\chi_1 (0))&
    \\    + \int_0^T ( - \lambda \nabla \cdot (\nabla {\phi }_n \otimes \nabla {\phi }_n ) + \nabla \cdot ({{\lambda}}_{e}  (1 - {\phi }_n) (F_n {{F_n}^T} - I ) )- \eta({\phi }_N)  \frac{(1- {\phi }_n) u_n}{\kappa ({\phi }_n)}, v \chi_1(t)) dt,&
    \\  (\nabla  \cdot  u_n , v) = 0,& 
    \\ \int _0^T  \left( {F_n}(t), \Xi \chi_2 ' (t)\right) dt + \int_0^T k\left(\nu(\phi_n)\Delta F_n - 2 \nabla \nu(\phi_n) \cdot F_n, \Xi \chi_2 (t)  \right) dt + \int_0^T \left(u_n \cdot \nabla F_n, \Xi \chi_2 (t)  \right) dt &
    \\ =
     (F_n(0), \Xi \chi_2 (0) ) + \int_0^T \left( \nabla u_n F_n, \Xi \chi_2(t) \right) dt,&
    \\ \int _0^T  \left(  {\phi_n}(t), \psi \chi_3 ' (t) \right) dt+ \int_0^T \left( u_n \cdot \nabla {\phi}_n, \psi \chi_3  (t) \right) dt  = 
   (\phi_n(0), \psi \chi_3 (0)) + \int_0^T ( \tau  \Delta {\mu}_n, \psi \chi_3(t)) dt ,&
    \\({\mu}_n, \psi) = ( - \lambda \Delta {\phi}_n + \lambda \gamma f'({\phi}_n) + \dfrac{\nu' (\phi_n)}{2} tr(F_n {F_n}^T -I), \psi ).&
    \end{array}
    \right.
\end{equation}
In above system, the first terms came from the integration by part for the time variable:
\begin{equation}\label{eqn:sys.with.n.conv3}
	    \begin{array}{ll}
\int _0^T \left\langle \frac{\partial {u_n}}{\partial t}, v \right\rangle  \chi_1(t)  dt  =  (u_n(T) v\chi_1(T)- u_n(0) v\chi_1(0))  - \int _0^T ({u_n} ,v )\chi_1'(t) dt
&\\
\qquad \qquad \qquad \qquad =  - u_n(0) v\chi_1(0)  - \int _0^T ( {u_n}(t), v \chi_1 '(t) ) dt . &\\
\end{array}
    \end{equation}
And the other first terms on each line are derived in the same way. As we pass $n\to\infty$, the linear terms in \eqref{eqn:sys.with.n.conv2} can be treated easily due to the weak convergence results stated in \eqref{aprioriresult2}. Thus we obtain

\begin{equation}\label{eqn:conveach1}
    \begin{split}
	 \int _0^T (  {u_n}(t), v \chi_1 '(t) ) dt \rightarrow \int _0^T (  {u}(t), v \chi_1 '(t) ) dt
    \\
    \int_0^T \left( {F_n}(t), \Xi \chi_2 ' (t)\right) dt \rightarrow \int_0^T \left( {F}(t), \Xi \chi_2 ' (t)\right) dt
    \\
    \int_0^T \left(  {\phi_n}(t), \psi \chi_3 ' (t) \right) dt \rightarrow \int_0^T \left(  {\phi}(t), \psi \chi_3 ' (t) \right) dt
    \end{split}
\end{equation}

weak convergent terms for $({u}_n, {\phi}_n, {F}_n)$ in $ (D(A) \cap H^3(\Omega)^d)\times H^4(\Omega) \times  (H^2(\Omega))^{d \times d}$ as $n \rightarrow \infty$. 
\\
\\
Due to \eqref{aprioriresult2}, we can show that
\begin{equation}\label{eqn:conveach2}
    \begin{split}
	(f_{1n} :=) \quad & - \nabla p_n + \nabla \cdot (\eta({\phi_n }) \nabla u_n)  - \lambda \nabla \cdot (\nabla {\phi }_n \otimes \nabla {\phi }_n ) 
    \\  & \qquad  +  \nabla \cdot ({{\lambda}}_{e}  (1 - {\phi }_n) (F_n {{F_n}^T} - I ) )- \eta({\phi }_n)  \frac{(1- {\phi }_n) u_n}{\kappa ({\phi }_n)}
    \\ &\qquad \rightarrow \quad - \nabla p + \nabla \cdot (\eta({\phi }) \nabla u)  - \lambda \nabla \cdot (\nabla {\phi } \otimes \nabla {\phi } ) 
    \\  &\qquad  \qquad  +  \nabla \cdot ({{\lambda}}_{e}  (1 - {\phi }) (F {{F}^T} - I ) )- \eta({\phi })  \frac{(1- {\phi }) u}{\kappa ({\phi })} \quad (=: f_1)
    \\ (f_{2n} :=) \quad &  \nabla  \cdot  u_n \rightarrow \nabla  \cdot  u\quad (=: f_2),
    \\ (f_{3n} :=) \quad  &  k(\nu(\phi_n) \Delta F_n - 2 \nabla \nu (\phi_n) \cdot \nabla F_n) +\nabla u_n F_n \rightarrow  k(\nu(\phi) \Delta F - 2 \nabla \nu (\phi) \cdot \nabla F) + \nabla u F \quad (=: f_3),
    \\ (f_{4n} :=) \quad  & - \lambda \Delta {\phi}_n + \lambda \gamma f'({\phi}_n) - \dfrac{{\lambda}_e}{2} tr(F_n {F_n}^T -I) 
    \\ \qquad  &\rightarrow - \lambda \Delta {\phi} + \lambda \gamma f'({\phi}) - \dfrac{{\lambda}_e}{2} tr(F {F}^T -I) \quad (=: f_4).
    \end{split}
\end{equation}
And,
\begin{equation}\label{eqn:conveach3}
    \begin{split}
    \int _0^T ( u_n(t) \cdot \nabla u_n  , v \chi_1 (t) ) dt  \rightarrow \int _0^T ( u(t) \cdot \nabla u  , v \chi_1 (t) ) dt 
    \\ \int_0^T \left(u_n \cdot \nabla F_n, \Xi \chi_2 (t)  \right) dt   \rightarrow \int_0^T \left(u \cdot \nabla F, \Xi \chi_2 (t)  \right) dt
    \\ \int_0^T \left( u_n \cdot \nabla {\phi}_n, \psi \chi_3  (t) \right) dt \rightarrow \int_0^T \left( u \cdot \nabla {\phi}, \psi \chi_3  (t) \right) dt.
    \end{split}
\end{equation}
Thus we arrive at the conclusion that the following equations are satisfied by $(u, \phi, F)$,
\begin{equation}\label{eqn:sysconv1}
	\left\{
    \begin{array}{ll}
    \int _0^T (  {u}(t), v \chi_1 '(t) ) dt + \int _0^T ( u(t) \cdot \nabla u  - \nabla \cdot (\eta({\phi }) \nabla u) , v \chi_1 (t) ) dt  =&
    \\[5pt] \qquad  (u_0, v\chi_1 (0)) + \int_0^T ( - \lambda \nabla \cdot (\nabla {\phi } \otimes \nabla {\phi } ) + \nabla \cdot ({{\lambda}}_{e}  (1 - {\phi }) (F {{F}^T} - I ) )&
    \\[5pt]  \qquad \qquad- \eta({\phi })  \frac{(1- {\phi }) u}{\kappa ({\phi })}, v \chi_1(t)) dt,&
    \\[5pt]  (\nabla  \cdot  u , v) = 0,& 
    \\[5pt] \int_0^T \left( {F}(t), \Xi \chi_2 ' (t)\right) dt + \int_0^T \left( k(\nu(\phi) \Delta F - 2 \nabla \nu (\phi) \cdot \nabla F), \Xi \chi_2  (t)\right) dt +\int_0^T \left(u \cdot \nabla F, \Xi \chi_2 (t)  \right) dt &
    \\= (F_0, \Xi \chi_2 (0) ) + \int_0^T \left( \nabla u F, \Xi \chi_2(t) \right) dt,&
    \\[5pt] \int_0^T \left(  {\phi}(t), \psi \chi_3 ' (t) \right) dt + \int_0^T \left( u \cdot \nabla {\phi}, \psi \chi_3  (t) \right) dt  = (\phi_0, \psi \chi_3 (0)) + \int_0^T ( \tau  \Delta {\mu}, \psi \chi_3(t)) dt ,&
    \\[5pt]({\mu}, \psi) = ( - \lambda \Delta {\phi} + \lambda \gamma f'({\phi}) + \dfrac{\nu'(\phi)}{2} tr(F {F}^T -I), \psi ).&
    \end{array}
    \right.
\end{equation}
for every $v \in (D(A) \cap H^3(\Omega)^d), \psi \in H^4(\Omega), \Xi \in (H^3(\Omega))^{d \times d}$ and $\chi_i\in C^1([0,T])$ for $i=1,2,3$. \\
\par
Now, for the weak solution ${u, \phi, F}$ proved in above, we can interpret \eqref{eqn:sysconv1} by  the following argument. From choosing the functions $\chi_1, \chi_2, \chi_3 \in $ $C_c^{\infty}(]0,1[)$, we obtain the following statement in the distribution sense on (0, T) directly from \eqref{eqn:sysconv1}.
\begin{equation}\label{weakhold}
    \begin{split}
&\text{To~find~the~functions}\\
& \qquad t \rightarrow (u(t), \phi (t),  F (t))\\
&  \text{~from~} [0, T] \text{~into~}  (D(A) \cap H^3(\Omega)^d) \times H^4(\Omega) \times  (H^3(\Omega))^{d \times d} \text{~satisfying}
    \end{split}
\end{equation}

\begin{equation*} 
	\left\{
		\begin{array}{lcl}
			  \left( \left( \frac{d}{dt} ({u}(t)) + u(t) \cdot \nabla u \right)  - \nabla \cdot (\eta({\phi }) \nabla u) , v \right) =& &
			  \\ ~~  \left( - \lambda \nabla \cdot (\nabla {\phi } \otimes \nabla {\phi } ) + \nabla \cdot ({{\lambda}}_{e}  (1 - {\phi }) (F {{F}^T} - I ) ) - \eta({\phi })  \frac{(1- {\phi }) u}{\kappa ({\phi })}, v\right),& &\\
			  (\nabla \cdot u , v) = 0, \\
			  \left( \frac{d}{dt} ({F}(t)), \Xi \right) +\left( k(\nu(\phi) \Delta F - 2 \nabla \nu (\phi) \cdot \nabla F),\Xi \right) + \left(u \cdot \nabla F, \Xi  \right) = \left( \nabla u F, \Xi  \right), && \\
			  \left( \frac{d}{dt} ({\phi}(t)) + u \cdot \nabla {\phi}, \psi \right) = ( \tau  \Delta {\mu}, \psi), && \\
			  ({\mu}, \psi) = ( - \lambda \Delta {\phi} + \lambda \gamma f'({\phi}) + \dfrac{\nu'(\phi)}{2} tr(F {F}^T -I), \psi ). &&\\
		\end{array}
		\right.
\end{equation*}
for every $v \in D(A) \cap H^3(\Omega)^d, \psi \in H^4(\Omega)$ and  $\Xi \in (H^3(\Omega))^{d \times d} $.
\\
\\
Moreover, for the weak solution $(u, F, \phi)$, we need to derive

$$(u(0), F(0), \phi(0)) = (u_0, F_0, \phi_0).$$

Since \eqref{weakhold} means that $(\dfrac{du}{dt}, \dfrac{dF}{dt}, \dfrac{d \phi}{dt})$ is weakly continuous from $(0,T)$ into $V \times (H^2(\Omega))^{d \times d} \times L^2(\Omega)$
, so $(u(0), F(0), \phi (0))$ can be well-defined from regularity on trace. For \eqref{weakhold} equation system, multiply $\chi_1, \chi_2, \chi_3 $, integrate on (0,T) and apply integration by parts. By subtracting this whole system from \eqref{eqn:sys.with.n.conv2} system equation, we get
\begin{equation}\label{eqn:initial}
	\left\{
    \begin{array}{ll}
    (u_0 - u(0), v)\chi_1 (0) = 0, \qquad \forall \chi_1 \in (H^3(\Omega))^d&
    \\(F_0 - F(0), \Xi) \chi_2 (0)  = 0,  \qquad  \forall \chi_2 \in  (H^2(\Omega))^{d\times d}&
    \\ ( \phi_0 - \phi(0), \psi )\chi_3 (0) = 0, \qquad \forall \chi_3 \in H^4(\Omega)&
    \end{array}
    \right.
\end{equation}
To be specific for the calculation to derive above \eqref{eqn:initial} from \eqref{weakhold}, we use similar computation as \eqref{eqn:sys.with.n.conv3}: 
\begin{equation}\label{}
    \begin{array}{ll}
\int_0^T (\dfrac{d}{dt} u(t), v \chi_1(t) )dt  = - (u(0), v \chi_1(0))  - \int_0^T (u(t), v \chi_1'(t) ) dt
    \end{array}
\end{equation}
 by integration by parts for the term in \eqref{eqn:sysconv1}. Then we obtain that this subtracted result terms as $-(u(0),v\chi_1 (0)) + (u_0, v\chi_1(0))$. With same calculation to get the second equation and third equation related to \eqref{eqn:initial}, we derive the following.

\begin{equation}\label{eqn:weakhold3}
    \begin{split}{}
    (u(0), F(0), \phi(0)) = (u_0, F_0, \phi_0).&
    \end{split}
\end{equation}
By choosing $\chi_1, \chi_2, \chi_3$ s.t. $\chi_1(0), \chi_2(0), \chi_3(0)$ are nonzero, we showed \eqref{eqn:weakhold3}.} Therefore, initial values are well-defined with given condition.
\\
\par
This completes the proof for the existence part of Theorem \ref{Thm_local_wp}.\qedsymbol
\\
\\
\textbf{Proof of uniqueness of Theorem \ref{Thm_local_wp}.} We describe the beginning step for proof on the uniqueness of strong solutions to the equations \eqref{governing_system}-\eqref{bdry} in Theorem \ref{Thm_local_wp}.
\\
Suppose that there are two solution sets $(u, \phi, F)$ and $(v, \psi, G)$ to \eqref{governing_system}-\eqref{bdry} when they are in the boundedness condition as Theorem \ref{Thm_local_wp}. Define $(\tilde{u}, \tilde{\phi}, \tilde{F}) := (u, \phi, F) - (v, \psi, G)$. For two governing systems solved by these solution sets, we subtract corresponding equations respectively and test them with $\tilde{u}$, $\tilde{\phi}$ and $\tilde{F}$.
\begin{equation}\label{Thm_wp_uniq}
\left\{ 
\begin{aligned}
&\frac{1}{2} \dfrac{d}{dt} |\tilde{u}|^2 + (\eta(\phi)\nabla u - \eta(\psi) \nabla v, \nabla \tilde{u})
        \\& \qquad = (-u \cdot \nabla u + v \cdot \nabla v, \tilde{u}) - \lambda (\nabla \cdot(\nabla \phi \otimes \nabla \phi - \nabla \psi \otimes \nabla \psi), \tilde{u})
        \\&\qquad \qquad + (\nabla \cdot (\nu(\phi)(FF^T -I)-\nu(\psi)(GG^T -I)), \tilde{u})
        +(-\frac{\eta(\phi)}{\kappa(\phi)}(1-\phi)u+\frac{\eta(\psi)}{\kappa(\psi)}(1-\psi)v, \tilde{u}) 
        \\&
\nabla \cdot \tilde{u} = 0, 
\\&
\frac{1}{2} \dfrac{d}{dt} |\tilde{F}|^2 +(- k \nu (\phi) \Delta F + k \nu (\psi) \Delta G, \tilde{F} )+2k( \nabla \nu(\phi) \cdot \nabla F - \nabla \nu(\psi) \cdot \nabla G, \tilde{F})
        + ( u \cdot \nabla F -v \cdot \nabla G, \tilde{F})
        \\& \qquad = (\nabla u F - \nabla v G, \tilde{F})
\\&
\frac{1}{2} \dfrac{d}{dt} |\tilde{\phi}|^2 = (-(u \cdot \nabla \phi -v \cdot \nabla \psi), \tilde{\phi} )- \tau \lambda (\Delta^2 \tilde{\phi} , \tilde{\phi} )
        +\tau \lambda \gamma (\Delta f'(\phi) 
        - \Delta f'(\psi), \tilde{\phi}) 
        \\& \qquad \qquad + \frac{\tau}{2}(\Delta (\nu'(\phi)tr(FF^T-I))-\Delta (\nu'(\psi)tr(GG^T-I)), \tilde{\phi}) .
\end{aligned}
\right.
\end{equation}
Here, $\epsilon$ is defined as role of $\mu$, i.e. $\epsilon = -\lambda \Delta \psi + \lambda \gamma f'(\psi) + \dfrac{{ \nu'(\psi)}}{2} \text{tr}(GG^T - I)$. Now, we will use minimal coefficient $\dot{m}$ = min$\{ {\alpha}, {\alpha k}, {\tau \lambda} \}$  in our computation which is the set of the dissipation term coefficients in governing system.
\\
\par
Test $\eqref{Thm_wp_uniq}_1$ with $\tilde{u}$. We obtain
\begin{equation}\label{5_1}
    \begin{aligned}
        & \frac{1}{2} \dfrac{d}{dt} |\tilde{u}|^2 + (\eta(\phi)\nabla u - \eta(\psi) \nabla v, \nabla \tilde{u})
        \\&= (-u \cdot \nabla u + v \cdot \nabla v, \tilde{u}) - \lambda (\nabla \cdot(\nabla \phi \otimes \nabla \phi - \nabla \psi \otimes \nabla \psi), \tilde{u})
        \\&+ (\nabla \cdot (\nu(\phi)(FF^T -I)-\nu(\psi)(GG^T -I)), \tilde{u})
        +(-\frac{\eta(\phi)}{\kappa(\phi)}(1-\phi)u+\frac{\eta(\psi)}{\kappa(\psi)}(1-\psi)v, \tilde{u}).
    \end{aligned}
\end{equation}
One term in above inequality is written as
\begin{equation}
    \begin{aligned}
        &(\eta(\phi)\nabla u - \eta(\psi) \nabla v, \nabla \tilde{u}) = (\eta(\phi)\nabla u -\eta(\phi)\nabla v, \nabla \tilde{u}) + (\eta(\phi)\nabla v  - \eta(\psi) \nabla v, \nabla \tilde{u})
        \\&
        \geq \alpha |\nabla \tilde{u}|^2 + (\eta(\phi)\nabla v  - \eta(\psi) \nabla v, \nabla \tilde{u})
    \end{aligned}
\end{equation}
Let us define partial terms for the test.
\begin{equation}
    \begin{aligned}
        G_1 : &= \nabla \cdot (\eta(\phi)\nabla v  - \eta(\psi) \nabla v)- \lambda (\nabla \cdot(\nabla \phi \otimes \nabla \phi - \nabla \psi \otimes \nabla \psi))
        \\&+ \nabla \cdot (\nu(\phi)(FF^T -I)-\nu(\psi)(GG^T -I))
        -\frac{\eta(\phi)}{\kappa(\phi)}(1-\phi)u+\frac{\eta(\psi)}{\kappa(\psi)}(1-\psi)v
    \end{aligned}
\end{equation}
We test $G_1$ by the $\tilde{u}$. To do this, let us use formula $\nabla \cdot (\nabla \phi \otimes \nabla \phi) = \Delta \phi \nabla \phi +\nabla(\frac{1}{2}|\nabla \phi|^2)$ to bound several terms of above equation. 
By estimating as follows, we bound the term with $(|\tilde{u}|^2 + |\tilde{\phi}|^2)$ and higher-order terms which will be absorbed properly later.
\begin{equation}
    \begin{aligned}
        (G_1 &, \tilde{u}) := |\nabla \cdot (\eta(\phi)\nabla v  - \eta(\psi) \nabla v)|+| - \lambda (\nabla \cdot(\nabla \phi \otimes \nabla \phi - \nabla \psi \otimes \nabla \psi), \tilde{u})
        \\&+ (\nabla \cdot (\nu(\phi)(FF^T -I)-\nu(\psi)(GG^T -I)), \tilde{u})
        +(-\frac{\eta(\phi)}{\kappa(\phi)}(1-\phi)u+\frac{\eta(\psi)}{\kappa(\psi)}(1-\psi)v, \tilde{u})  | 
        \\& \leq |(\eta(\phi)\nabla v  - \eta(\psi) \nabla v, \nabla \tilde{u}|+
         \lambda |(\Delta \tilde{\phi} \nabla \phi + \Delta \psi \nabla \tilde{\phi} +\nabla(\frac{1}{2}(\nabla \phi + \nabla \psi) \cdot (\nabla \phi - \nabla \psi)) , \tilde{u})|
        \\&+ |(\nu'(\phi)\nabla \phi \cdot (FF^T -I)-\nu'(\psi)\nabla \psi \cdot (GG^T -I), \tilde{u})| + \beta |(\nabla \cdot (FF^T - GG^T), \tilde{u})| 
        \\&
        +\frac{\beta}{\alpha}|(\tilde{u}, \tilde{u})|
        + \frac{\beta}{\alpha}{ |(\tilde{\phi}u + \psi \tilde{u}, \tilde{u})|}
        \\&\leq |\eta(\phi) - \eta(\psi)|| \nabla v |_2| \nabla \tilde{u})|+
         \lambda |(\Delta \tilde{\phi} \nabla \phi + \Delta \psi \nabla \tilde{\phi} + \frac{1}{2}\nabla (\nabla \phi + \nabla \psi) \cdot \nabla \tilde{\phi} + (\nabla \phi + \nabla \psi) \cdot \nabla^2 \tilde{\phi}, \tilde{u})|
        \\&
        +{  |((\nu'(\phi)-\nu'(\psi))\nabla \phi \cdot (FF^T -I),\tilde{u})|}+ \beta|(\nabla \tilde{\phi} \cdot (FF^T -I), \tilde{u})|+\beta |(\nabla \psi \cdot (FF^T -GG^T), \tilde{u})| 
        \\
        & + 2\beta|(\nabla \cdot (\tilde{F} {F}^T), \tilde{u})|+ 2\beta|(\nabla \cdot ({G}\tilde{F}^T), \tilde{u})|
        +\frac{\beta}{\alpha}|\tilde{u}|^2+ \frac{\beta}{\alpha}{ |(\tilde{\phi}u + \psi \tilde{u}, \tilde{u})|}
        \\&\leq |\tilde{\phi}|| \nabla v |_2| \nabla \tilde{u}|
        +
        \lambda
        (
        |\Delta \tilde{\phi } ||\nabla \phi|_{L^{\infty}(\Omega)}| \tilde{u}| + |\Delta \psi|_{L^{\infty}(\Omega)} |\nabla \tilde{\phi}|| \tilde{u}| + |\nabla^2 (\phi +\psi)|_{L^{\infty}(\Omega)}|\Delta \tilde{\phi}| | \tilde{u}|
        \\&+{  C|\tilde{\phi}||\nabla \phi|_2 | FF^T -I|_2|\tilde{u}|}
        )
        +
        \beta|\Delta \tilde{\phi} ||(FF^T -I)|_{L^{\infty}(\Omega)}| \tilde{u}|
        \\&+\beta |\nabla \psi|_{L^{\infty}(\Omega)}|\tilde{F}||F+G|_{L^{\infty}}| \tilde{u}| + 2\beta|\nabla \tilde{F}|| {F}|_{L^{\infty}}| \tilde{u}|+2 \beta |\tilde{F}||\nabla F|_1|\tilde{u}|_1 + 2 \beta |G|_2 |\nabla \tilde{F}||\tilde{u}|
        \\&+ 2\beta|\nabla {G} |_{L^{6}(\Omega)}|\tilde{F}|_{L^{3}}| \tilde{u}|
        +\frac{\beta}{\alpha}|\tilde{u}|^2
        +\frac{\beta}{\alpha} (|\tilde{\phi}||u|_{L^{\infty}}+ |{\psi}|_{L^{\infty}}|\tilde{u}|)|\tilde{u}|
        \\&\leq 
        C
        (
        |\Delta \tilde{\phi } ||\phi|_3| \tilde{u}| + |\psi|_3 |\Delta \tilde{\phi}|| \tilde{u}| +(|\phi|_4+|\psi|_4)|\Delta \tilde{\phi}| | \tilde{u}|
        +{  |\tilde{\phi}||\nabla \phi|_2 (|F|_2^2+1)|\tilde{u}|}
        \\&
        +|\Delta \tilde{\phi} || \tilde{u}|
        +|\Delta \tilde{\phi} ||{F}|_2^2| \tilde{u}| 
        \\&
        +|\psi |_3|\nabla \tilde{F}|(|F|_2 + |G|_2)| \tilde{u}|+ |\nabla \tilde{F}||F|_2| \tilde{u}| +|G|_2|| \tilde{F}|^{\frac 12}| \tilde{F}|_1^{\frac 12} \tilde{u}|
        \\&+|\tilde{u}|^2
        +(|\tilde{\phi}||u|_2+ |\psi|_2|\tilde{u}|)|\tilde{u}|
        )
        \\&\leq |\tilde{\phi}|^2|  v |_3^2 + \frac{\dot{m}}{6} |\nabla \tilde{u}|^2
        +
        \frac{\dot{m}}{9}(|\Delta \tilde{\phi}|^2 + |\nabla \tilde{F}|) 
        \\&+ 
        C
        (
        |\phi|_3^2|\tilde{u}|^2 +|\psi|_3^2|\tilde{u}|^2  
        + (|\phi|_4^2 +|\psi|_4^2)|\tilde{u}|^2 
        +
        {  |\tilde{\phi}|^2 + |\phi|_3^2 (|F|_2^4+1)|\tilde{u}|^2}
        \\& +|\tilde{u}|^2 + |{F}|^4|\tilde{u}|^2
        +|\psi|_3(|{F}|_2^2+|{G}|_2^2)|\tilde{u}|^2
        +
        |{F}|_2^2|\tilde{u}|^2 + |G|_2^2 |\tilde{u}|^2
        +|\tilde{F}|^2 +|\tilde{u}|^2
        \\& + |\tilde{\phi}|^2|u|_2^2 + |\tilde{u}|^2+ |\psi|_2|\tilde{u}|^2)
        \\& \leq
        \frac{\dot{m}}{6} |\nabla \tilde{u}|^2
        + \frac{\dot{m}}{9}(|\Delta \tilde{\phi}|^2 + |\nabla \tilde{F}|) + M_1(u(t), \phi(t), F(t))(|\tilde{u}|^2 + |\tilde{\phi}|^2).
    \end{aligned}
\end{equation}
where $M_1(u(t), \phi(t), F(t))$ is bounded part for the terms of $(u, \phi, F)$ and $(v, \psi, G)$ as we proved in a priori estimate which we can check in Theorem \eqref{Thm_local_wp}. Here, we used \eqref{H2_phi}. And we used that $\nu$ is lipsitz continuous.\\
\par
Then, we can simplify the \eqref{5_1} from Poincare inequality \eqref{poincare_u} as follows.
\begin{equation}\label{5_1'}
    \begin{aligned}
        &\frac{1}{2} \dfrac{d}{dt} |\tilde{u}|^2 + \dot{m} |\nabla \tilde{u}|^2 
        \leq 
        -(\tilde{u} \cdot \nabla u , \tilde{u}) + (G_1, \tilde{u}) 
        \\&\leq |\tilde{u}|_{L^6}|\nabla u|_{L^3}| \tilde{u}| + (G_1, \tilde{u})
        \\&\leq C|\nabla \tilde{u}||\nabla u|^{\frac 12}|\nabla u|_1^{\frac 12}| \tilde{u}| + (G_1, \tilde{u}) 
        \\&\leq C|\nabla u||\nabla u|_1|\tilde{u}|^2 + \frac{\dot{m}}{9}|\nabla \tilde{u}|^2 + (G_1, \tilde{u}) 
        \\&
         \leq C|\nabla u||\nabla u|_1|\tilde{u}|^2 + {  \frac{\dot{m}}{9}}|\nabla \tilde{u}|^2 
        \\& + \frac{\dot{m}}{6} |\nabla \tilde{u}|^2
        + \frac{\dot{m}}{9}(|\Delta \tilde{\phi}|^2 + |\nabla \tilde{F}|) + M_1(u(t), \phi(t), F(t))(|\tilde{u}|^2 + |\tilde{\phi}|^2)
    \end{aligned}
\end{equation}
\\
\\
Next, test the difference of $\eqref{governing_system}_4$ on $(u, \phi, F)$ and $\eqref{governing_system}_4$ on $(v, \psi, G)$ with $\tilde{\phi}$. Firstly, we write the subtraction as
\begin{equation}\label{5_2_0}
    \begin{aligned}
        & \dfrac{d}{dt} \tilde{\phi} = -(u \cdot \nabla \phi -v \cdot \nabla \psi) + \tau \Delta \tilde{\mu}- \mu_3I_h(\tilde{\phi})
        \\&= -(u \cdot \nabla \phi -v \cdot \nabla \psi) - \tau \lambda \Delta^2 \tilde{\phi} 
        +\tau \lambda \gamma (\Delta f'(\phi) 
        - \Delta f'(\psi)) \\&+ \frac{\tau}{2}(\Delta (\nu'(\phi)tr(FF^T-I))-\Delta (\nu'(\psi)tr(GG^T-I)))
    \end{aligned}
\end{equation}
On above equation, we test the difference of the governing equation on $(u, \phi, F)$ and governing equation on $(v, \psi, G)$ and used $\eqref{governing_system}_5$ to expand the equation. Now, test the \eqref{5_2_0} with $\tilde{\phi}$ to obtain
\begin{equation}\label{5_2}
    \begin{aligned}
        &\frac{1}{2} \dfrac{d}{dt} |\tilde{\phi}|^2 + \tau \lambda |\Delta \tilde{\phi}|^2 = -\tau \lambda \int_{\partial \Omega}\partial_n \Delta \tilde{\phi} \tilde{\phi} d\Gamma -(u \cdot \nabla \phi -v \cdot \nabla \psi, \tilde{\phi})
        +\tau \lambda \gamma (\Delta f'(\phi) 
        - \Delta f'(\psi), \tilde{\phi}) \\&+ \frac{\tau}{2}(\Delta (\nu'(\phi)tr(FF^T-I))-\Delta (\nu'(\psi)tr(GG^T-I)), \tilde{\phi}) .
    \end{aligned}
\end{equation}

On this, let us define
\begin{equation}
    \begin{aligned}
        (G_2, \tilde{\phi}) : &= -\tau \lambda \int_{\partial \Omega}\partial_n \Delta \tilde{\phi} \tilde{\phi} d\Gamma + (\tau \lambda \gamma (\Delta f'(\phi) - \Delta f'(\psi)) , \tilde{\phi})
        \\&+ (\frac{\tau}{2}(\Delta (\nu'(\phi)tr(FF^T-I)), \tilde{\phi})-(\Delta (\nu'(\psi)tr(GG^T-I)), \tilde{\phi})
    \end{aligned}
\end{equation}
For the integration term on the boundary of the domain, we estimate as
\begin{equation}\label{5_2_bdry}
    \begin{aligned}
        & \int_{\partial \Omega}\partial_n (\lambda  \Delta \tilde{\phi}) \tilde{\phi} d\Gamma =
          \int_{\partial \Omega}\partial_n (-\epsilon + \lambda \gamma f'(\psi) + \frac{\nu' (\psi)}{2} tr(GG^T-I)) \tilde{\phi} d\Gamma
        \\& = 0
    \end{aligned}
\end{equation}
from Neumann boundary condition on $\epsilon$, $\psi$ and $G$ in \eqref{bdry}.\\
Continue the estimate as
\begin{equation}\label{5_2_F2}
    \begin{aligned}
        &(G_2 , \tilde{\phi}) 
        \\&= \tau \lambda \gamma (\Delta f'(\phi) - \Delta f'(\psi), \tilde{\phi}) 
        + \frac{\tau}{2}(\Delta (\nu'(\phi)tr(FF^T-I))-\Delta (\nu'(\psi)tr(GG^T-I)), \tilde{\phi})
        \\&\leq \tau \lambda \gamma ( f'''(\phi)\|\nabla \phi\|^2 -  f'''(\psi)\|\nabla \psi\|^2 +f''(\phi) \Delta \phi - f''(\psi) \Delta \psi, \tilde{\phi}) 
        \\&+ \frac{\tau}{2} \Biggl\{  \Biggr.( \nu'''(\phi)|\nabla \phi|^2 tr(FF^T-I)-\nu'''(\psi) |\nabla \psi|^2tr(GG^T-I), \tilde{\phi})
        \\&+ ( \nu''(\phi)\Delta \phi tr(FF^T-I)-\nu''(\psi) \Delta \psi tr(GG^T-I), \tilde{\phi})
        \\&+(\nu''(\phi)\nabla\phi 2\sum_{i,j}(\nabla F^{ij} F^{ij}) - \nu''(\psi)\nabla\psi 2\sum_{i,j}(\nabla G^{ij} G^{ij}), \tilde{\phi} )
        +(\nu'(\phi)\sum_{ij}\Delta ((F^{ij})^2)- \nu'(\psi)\sum_{ij}\Delta ((G^{ij})^2), \tilde{\phi})\Biggl.  \Biggr\}
        \\&\leq 
        C((\phi \|\nabla \phi \|^2 - \phi \|\nabla \psi \|^2 +  \phi \|\nabla \psi \|^2 - \psi \| \nabla \psi\|^2, \tilde{\phi})  
        + (\|\nabla \phi \|^2 - \| \nabla \psi\|^2, \tilde{\phi})) \\&+ C((\phi^2 \Delta{\phi} -\phi^2 \Delta{\psi} + \phi^2 \Delta{\psi} -\psi^2 \Delta{\psi}, \tilde{\phi}) + (\phi \Delta{\phi}-\phi \Delta{\psi}+\phi \Delta{\psi} -\psi \Delta{\psi}, \tilde{\phi}) + (\Delta{\phi}-\Delta{\psi}, \tilde{\phi}))
        \\&+\frac{ \tau}{2}|((\nu'''(\phi)-\nu'''(\psi))|\nabla \phi|^2tr(FF^T-I),\tilde{\phi})|+ \frac{\beta \tau}{2}( \|\nabla \phi\|^2 tr(FF^T-I)-\|\nabla \psi\|^2 tr(FF^T-I)
        \\&+\|\nabla \psi\|^2 tr(FF^T-I)+ \|\nabla \psi\|^2tr(GG^T-I), \tilde{\phi}) 
        \\&+\frac{ \tau}{2}|((\nu''(\phi)-\nu''(\psi))\Delta \phi tr(FF^T-I),\tilde{\phi})|+ \frac{\beta \tau}{2}( \Delta \phi tr(FF^T-I)- \Delta \psi tr(FF^T-I)
        \\&+\Delta \psi tr(FF^T-I)+ \Delta \psi tr(GG^T-I), \tilde{\phi}) 
        \\&+2 ((\nu''(\phi)-\nu''(\psi))\nabla \phi \sum_{i,j}(\nabla F^{i,j}F^{i,j}), \tilde{\phi})+2 \beta (\nabla\phi \sum_{i,j}(\nabla F^{ij} F^{ij}) - \nabla\psi \sum_{i,j}(\nabla G^{ij} G^{ij}), \tilde{\phi} )
        \\&+2((\nu'(\phi)-\nu'(\psi))\sum_{i,j}\Delta ((F^{i,j})^2), \tilde{\phi})+\beta (2\sum_{i,j}(|\nabla F^{i,j}|^2 - |\nabla G^{i,j}|^2 + F^{i,j} \Delta F^{i,j} - G^{i,j} \Delta G^{i,j}), \tilde{\phi})
        \\&\leq  
        C (|\phi|_2|\nabla \tilde{\phi}|_1(|\nabla \phi|_1 + |\nabla \psi|_1 )|\tilde{\phi}| + |\tilde{\phi}||\nabla \psi|_2^2|\tilde{\phi}|
        +|\nabla \tilde{\phi}|_1(|\nabla \phi|_1 + |\nabla \psi|_1 )|\tilde{\phi}|) 
        \\&+
        C(|\phi|_2^2|\Delta \tilde{\phi}||\tilde{\phi}|+ |\tilde{\phi}|(|\phi|_2 + |\psi|_2)|\Delta {\psi}|_2|\tilde{\phi}|+|\phi|_2|\Delta \tilde{\phi}||\tilde{\phi}|+|\tilde{\phi}||\Delta {\psi}|_2|\tilde{\phi}|+|\Delta \tilde{\phi}||\tilde{\phi}|)
        \\
        &+ C |\tilde{\phi}||\nabla \phi|_2^2|tr(FF^T-I)|_2 |\tilde{\phi}| 
        + {  \frac{\beta \tau}{2} }
        |\nabla \tilde{\phi}|(|\nabla \phi|_2 + |\nabla \psi|_2)|tr(FF^T)|_2|\tilde{\phi}|
        \\& +  \frac{\beta \tau}{2} |\nabla {\psi}|_2^2|(|\tilde{F}||F|_2+|G|_2|\tilde{F}|)|\tilde{\phi}|
        \\& +  \frac{\beta \tau}{2} (|\tilde{\phi}||\Delta \phi|_2(|F|_2|F|_2 +1)|\tilde{\phi}| + |\Delta \tilde{\phi}|(|F|_2|F|_2 +1)|\tilde{\phi}| +|\Delta \psi |_2 (|\tilde{F}||G|_2+ |F|_2|\tilde{F}|)|\tilde{\phi}|)
        \\
        & + C |\nu''(\phi)-\nu''(\psi)| |\nabla \phi |_2 \sum_{i,j}|\nabla F^{i,j}|_1|F^{i,j}|_2|\tilde{\phi}|_1
        +2 \beta \sum_{i,j}(
        \nabla \tilde{\phi} \nabla F^{i,j} F^{i,j} + \nabla \psi (\nabla \tilde{F}^{i,j} F^{i,j} + \nabla G^{i,j} \tilde{F}^{i,j})
        , \tilde{\phi}) 
        \\&+ C |\nu'(\phi)-\nu'(\psi)||F|_2^2 |\tilde{\phi}|_1+
        2\beta \sum_{i,j}|\nabla \tilde{F}^{i,j}|(|\nabla F^{i,j}|_{L^6}+|\nabla G^{i,j}|_{L^6})|\tilde{\phi}|_{L^3} + 2\beta  \sum_{i,j}(\tilde{F}^{i, j}\Delta F^{i, j}+ {G}^{i, j}\Delta \tilde{F}^{i, j}, \tilde{\phi})
        \\
        &
        \leq C (|\phi|_2|\nabla \tilde{\phi}|_1(| \phi|_2 + | \psi|_2 )|\tilde{\phi}| + |\tilde{\phi}|| \psi|_3^2|\tilde{\phi}|
         + |\Delta \tilde{\phi}|(|\phi|_2 + |\psi|_2 +1)|\tilde{\phi}|)
        \\&
        +
        C(|\phi|_2^2|\Delta \tilde{\phi}||\tilde{\phi}|+ |\tilde{\phi}|(|\phi|_2 + |\psi|_2)| {\psi}|_4|\tilde{\phi}|+|\phi|_2|\Delta \tilde{\phi}||\tilde{\phi}|+|\tilde{\phi}|| {\psi}|_4|\tilde{\phi}|+|\Delta \tilde{\phi}||\tilde{\phi}|)
        \\
        &+ C |\tilde{\phi}|^2| \phi|_4^2(|F|_2^2+1)+
        |\nabla \tilde{\phi}|(|\nabla \phi|_2 + |\nabla \psi|_2)|\tilde{\phi}|
        + |\nabla \tilde{\phi}|(| \phi|_3
        +| \psi|_3)|F|_2^2|\tilde{\phi}|
        \\&+  \frac{\beta \tau}{2} |\nabla {\psi}|_2^2|(|\tilde{F}||F|_2+|G|_2|\tilde{F}|)|\tilde{\phi}|)
         + |\Delta \tilde{\phi}|(|F|_2^2 +1)|\tilde{\phi}| +|\Delta \psi |_2 (|\tilde{F}||G|_2+ |F|_2|\tilde{F}|)|\tilde{\phi}|)
        \\
        &+{  |\tilde{\phi}||\phi|_3|F|_2^2(|\tilde{\phi}|+ |\Delta \tilde{\phi}|)}
        +2 \beta(|
        \nabla \tilde{\phi} |\nabla F|_1| F|_1 + |\nabla \psi|_2 ( \sum_{i,j}|\nabla \tilde{F}^{i,j}|| F |_2+ |\nabla G|_1|\tilde{F}|_1)
        |\tilde{\phi}|)
        \\&+{  |\tilde{\phi}||F|_2^2(|\tilde{\phi}|+ |\Delta \tilde{\phi}|)}
        +2\beta \sum_{i,j}|\nabla \tilde{F}^{i,j}|(| F|_2+| G|_2)|\tilde{\phi}|_1^{\frac 12}|\tilde{\phi}|^{\frac 12} + 2\beta |\tilde{F}||F|_2|\tilde{\phi}|_2+ \sum_{i,j}(G^{i, j}\Delta \tilde{F}^{i, j}, \tilde{\phi}).
    \end{aligned}
\end{equation}
We used that $\nu'$ and $\nu''$ are Lipschitz continuous and that $ |\tilde{\phi}|_2 \leq C|\Delta \tilde{\phi} |+C|\int_{\Omega} \phi dx|$.
\\
For the last term in the above inequality,
\begin{equation}\label{5_2_Del}
    \begin{aligned}
        & \sum_{i,j}(G^{i,j} \Delta \tilde{F}^{i,j}, \tilde{\phi})
        \\& = 
        -\sum_{i,j}(\nabla G^{i,j} \cdot \nabla \tilde{F}^{i,j}, \tilde{\phi})
        -\sum_{i,j}(\nabla \tilde{F}^{i,j}, G^{i,j} \nabla \tilde{\phi})
        + \sum_{i,j}\int_{\partial \Omega} \mathbf{n} \cdot (G^{i,j} \nabla \tilde{F}^{i,j} \tilde{\phi}) d \Gamma
        \\&
        \leq
        \sum_{i,j}|\nabla G^{i,j} |_{L^6}|\nabla \tilde{F}^{i,j}||\tilde{\phi}|_{L^3}
        +\sum_{i,j}| \nabla \tilde{F}^{i,j}|_{H^{-1}}| G^{i,j}\nabla \tilde{\phi}|_1+ \sum_{i,j}(|\tilde{F}^{i,j} |_1 
        \\&
        \leq
        \sum_{i,j}|\nabla G^{i,j} |_{L^6}|\nabla \tilde{F}^{i,j}||\tilde{\phi}|_{L^3}
        +\sum_{i,j}| \tilde{F}^{i,j}|(| G^{i,j}|_2|\nabla \tilde{\phi}| +| \nabla G^{i,j}|_{L^6}|\nabla \tilde{\phi}|_{L^3} + | G^{i,j}|_2|\Delta \tilde{\phi}|)
        \\& \leq \sum_{i,j}| G^{i,j}|_2| \nabla \tilde{F}^{i,j}||\tilde{\phi}|^{\frac 12}|\tilde{\phi}|_1^{\frac 12}
         + \sum_{i,j}|\tilde{F}^{i,j}| (| G^{i,j}|_2|\Delta \tilde{\phi}| +|  G^{i,j}|_2|\Delta \tilde{\phi}| + | G^{i,j}|_2|\Delta \tilde{\phi}|)
        \\&\leq \frac{\dot{m}}{9}(|\Delta \tilde{\phi}|^2 + \sum_{i,j}|\nabla \tilde{F}^{ij}|) 
        \\& +
        C
        (
        |G|_2^4|\tilde{\phi}|^2 +|\tilde{F}|^2|G|_2^2 
        )
    \end{aligned}
\end{equation}
On above, we employed inequalities \eqref{lem_neg_sobolev} and \eqref{def_neg_sobolev}. Also we used Gagliardo-nirenberg inequality here.
\\
\par
Let us combine \eqref{5_2_bdry} - \eqref{5_2_Del}.
\begin{equation}
    \begin{aligned}
        &(G_2, \phi)  
        \\& \leq C (|\phi|_2|\nabla \tilde{\phi}|_1(|\nabla \phi|_1 + |\nabla \psi|_1 )|\tilde{\phi}| + |\tilde{\phi}||\nabla \psi|_2^2|\tilde{\phi}|
        +|\nabla \tilde{\phi}|_1(|\nabla \phi|_1 + |\nabla \psi|_1 )|\tilde{\phi}|) 
        \\&+
        C(|\phi|_2^2|\Delta \tilde{\phi}||\tilde{\phi}|+ |\tilde{\phi}|(|\phi|_2 + |\psi|_2)|\Delta {\psi}|_2|\tilde{\phi}|+|\phi|_2|\Delta \tilde{\phi}||\tilde{\phi}|+|\tilde{\phi}||\Delta {\psi}|_2|\tilde{\phi}|+|\Delta \tilde{\phi}||\tilde{\phi}|)
        \\
        &+ C |\tilde{\phi}|^2| \phi|_4^2(|F|_2^2+1)+
        |\nabla \tilde{\phi}|(|\nabla \phi|_2 + |\nabla \psi|_2)|\tilde{\phi}|
        + |\nabla \tilde{\phi}|(| \phi|_3
        +| \psi|_3)|F|_2^2|\tilde{\phi}|
        \\&+  \frac{\beta \tau}{2} |\nabla {\psi}|_2^2|(|\tilde{F}||F|_2+|G|_2|\tilde{F}|)|\tilde{\phi}|)
         + |\Delta \tilde{\phi}|(|F|_2^2 +1)|\tilde{\phi}| +|\Delta \psi |_2 (|\tilde{F}||G|_2+ |F|_2|\tilde{F}|)|\tilde{\phi}|)
        \\
        &+{  |\tilde{\phi}||\phi|_3|F|_2^2(|\tilde{\phi}|+ |\Delta \tilde{\phi}|)}
        +2 \beta(|
        \nabla \tilde{\phi} |\nabla F|_1| F|_1 + |\nabla \psi|_2 ( \sum_{i,j}|\nabla \tilde{F}^{i,j}|| F |_2+ |\nabla G|_1|\tilde{F}|_1)
        |\tilde{\phi}|)
        \\&+{  |\tilde{\phi}||F|_2^2(|\tilde{\phi}|+ |\Delta \tilde{\phi}|)}
        +2\beta \sum_{i,j}|\nabla \tilde{F}^{i,j}|(| F|_2+| G|_2)|\tilde{\phi}|_1^{\frac 12}|\tilde{\phi}|^{\frac 12} + 2\beta |\tilde{F}||F|_2|\tilde{\phi}|_2+ \sum_{i,j}(G^{i, j}\Delta \tilde{F}^{i, j}, \tilde{\phi})
        \\&+ \frac{\dot{m}}{9}(|\Delta \tilde{\phi}|^2 + \sum_{i,j}|\nabla \tilde{F}^{ij}|)\\& +
        C
        (
        |G|_2^4|\tilde{\phi}|^2 +|\tilde{F}|^2|G|_2^2 
        )
        \\&\leq \frac{\dot{m}}{9}(|\Delta \tilde{\phi}|^2 + \sum_{i,j}|\nabla \tilde{F}^{ij}|) 
        \\& + C (|\phi|_2^2(| \phi|_2^2 + | \psi|_2^2 )|\tilde{\phi}|^2 + |\tilde{\phi}|^2| \psi|_3^2
        +C (|\phi|_2^2 + |\psi|_2^2 +1)|\tilde{\phi}|^2) 
        \\&+
        C(|\phi|_2^4|\tilde{\phi}|^2+ |\tilde{\phi}|^2(|\phi|_2 + |\psi|_2)| {\psi}|_4+|\phi|_2^2|\tilde{\phi}|^2+|\tilde{\phi}|^2| {\psi}|_4+
        C|\tilde{\phi}|^2)
        \\
        &+ C |\tilde{\phi}|^2| \phi|_4^2(|F|_2^2+1)+
        (| \phi|_3^2 + C| \psi|_3^2)|\tilde{\phi}|^2
        + C(| \phi|_3^2
        +| \psi|_3^2)|F|_2^4|\tilde{\phi}|^2
        \\&+ C (|\tilde{F}|^2 + | {\psi}|_3^4|(|F|_2^2+|G|_2^2)|\tilde{\phi}|^2)
         +C(|F|_2^4 +1)|\tilde{\phi}|^2 +C(|\tilde{F}|^2 + | \psi |_3^2 (|G|_2^2+ |F|_2^2)|\tilde{\phi}|^2)
        \\&+C \left( \right.
        (|\tilde{\phi}|^2|F|_2^2
        +|\tilde{\phi}|^2|F|_2^4)
        +|F|_2^2| F|_1^2 
        + 
        | \psi|_3^2| F |_2^2|\tilde{\phi}|^2+ | \psi|_3| G|_2|\tilde{F}||\tilde{\phi}|
        \\&
        +{  (|\tilde{\phi}|^2|F|_2^4+|\tilde{\phi}|^2|F|_2^4)}
        +((| F|_2+| G|_2)|\tilde{\phi}|^2+(| F|_2^4+| G|_2^4)|\tilde{\phi}|^2) + (|\tilde{F}|^2|F|_2^2+|\tilde{F}|^2|F|_2^2+|\tilde{\phi}|^2)
        \left. \right)
        \\& +
        C
        (
        |G|_2^4|\tilde{\phi}|^2 +|\tilde{F}|^2|G|_2^2 
        )
        \\&
        \leq  
         \frac{\dot{m}}{9}(|\Delta \tilde{\phi}|^2 + \sum_{i,j}|\nabla \tilde{F}^{ij}|)  + M_2(u(t), \phi(t), F(t))( |\tilde{\phi}|^2 +|\tilde{F}|^2)
    \end{aligned}
\end{equation}
by using \eqref{H2_phi} to estimate $|\nabla \tilde{\phi}|$. $M_2(u(t), \phi(t), F(t))$ is bounded part for the terms of $(u, \phi, F)$ and $(v, \psi, G)$ as we proved in a priori estimate for Theorem \eqref{Thm_local_wp} as we can see in this Theorem.
We used \eqref{H2_phi} here as well.\\
Then, we can simplify the \eqref{5_2} as
\begin{equation}\label{5_2'}
    \begin{aligned}
        &\frac{1}{2} \dfrac{d}{dt} |\tilde{\phi}|^2 + \dot{m} |\Delta  \tilde{\phi}|^2 
        \\&\leq -(\tilde{u} \cdot \nabla \phi , \tilde{\phi}) + (G_2, \tilde{\phi}) 
        \\&
        \leq |\tilde{u}|_1|\nabla \phi|^{\frac{1}{2}}|\nabla {\phi}|_1^{\frac{1}{2}}|\tilde{\phi}|+ (G_2, \tilde{\phi}) 
        \\&
        \leq 
         C| \nabla \phi|| \nabla \phi|_1|\tilde{\phi}|^2+ \frac{\dot{m}}{9}|\nabla \tilde{u}|^2 +(G_2, \tilde{\phi}) 
        \\&\leq C| \nabla \phi|| \nabla \phi|_1|\tilde{\phi}|^2 + {  \frac{\dot{m}}{9}}|\nabla \tilde{u}|^2 
        + \frac{\dot{m}}{9}(|\Delta \tilde{\phi}|^2 + |\nabla \tilde{F}|) + M_2(t)(|\tilde{\phi}|^2 +|\tilde{F}|^2)  
        .
    \end{aligned}
\end{equation}
Also, \eqref{H2_phi} was employed to control of $H^2(\Omega)$ norm of the $\tilde{\phi}$.
\\
\\
\\
As a next step, test the difference of $\eqref{governing_system}_3$ on $(u, \phi, F)$ and $\eqref{governing_system}_3$ on $(v, \psi, G)$ with $\tilde{F}$. We can write the subtraction as
\begin{equation}\label{5_3_0}
    \begin{aligned}
        &\frac{1}{2} \dfrac{d}{dt} \tilde{F} - k \nu (\phi) \Delta F + k \nu (\psi) \Delta G +2 k\nabla \nu(\phi) \cdot \nabla F - 2 k\nabla \nu(\psi) \cdot \nabla G
        + u \cdot \nabla F -v \cdot \nabla G
        \\&= \nabla u F - \nabla v G
    \end{aligned}
\end{equation}
By taking the test function $\tilde{F}$, 
\begin{equation}
    \begin{aligned}
    &- k (\nu (\phi) \Delta \tilde{F} +(\nu (\phi) -\nu (\psi) )\Delta G, \tilde{F}) 
    \\& \geq \beta k \sum_{i,j}(\nabla \tilde{F}^{i,j},\nabla \tilde{F}^{i,j}) -k((\nu (\phi) -\nu (\psi) )\Delta G, \tilde{F}) 
    \end{aligned}
\end{equation}
Therefore, we write the test as
\begin{equation}\label{5_3}
    \begin{aligned}
        &\frac{1}{2} \dfrac{d}{dt} |\tilde{F}|^2 + \alpha k \sum_{i,j}|\nabla \tilde{F}^{ij}|^2 
        \\&\leq k((\nu (\phi) -\nu (\psi) )\Delta G, \tilde{F}) 
        +2k (\nabla \nu(\phi) \cdot \nabla \tilde{F}+( \nu'(\phi) \nabla \phi- \nu'(\phi)\nabla \psi + \nu'(\phi)\nabla \psi  -  \nu'(\psi) \nabla \psi )\cdot \nabla G, \tilde{F})
        \\&-(u \cdot \nabla F -v \cdot \nabla G, \tilde{F})
        +(\nabla u F - \nabla v G, \tilde{F})
        \\&\leq Ck ( \|\phi - \psi\|\Delta G,  \tilde{F})
        +2k(\nu'(\phi)\nabla \phi \cdot \nabla F, \tilde{F})
        \\&
        + 2 \beta k ((\nabla \phi - \nabla \psi)\cdot \nabla G, \tilde{F})
        + 2 C k  (\|\phi - \psi\|\nabla \psi \cdot \nabla G, \tilde{F})
        \\&
        -(\tilde{u} \cdot \nabla F , \tilde{F})
        + (\nabla \tilde{u} F + \nabla v \tilde{F}, \tilde{F}).
    \end{aligned}
\end{equation}
Here, {  we used the neumann boundary condition} for F as in \eqref{bdry} for integration by part and divergence free condition on $v$. Also, we used Lipschitz continuity of $\nu$ and $\nu'$.
\\
Let us define
\begin{equation}
    \begin{aligned}
        (F_3,\tilde{F}) : &= Ck ( \|\phi - \psi\|\Delta G,  \tilde{F})
        +2k(\nu'(\phi)\nabla \phi \cdot \nabla F, \tilde{F})
        + 2 \beta k ((\nabla \phi - \nabla \psi)\cdot \nabla G, \tilde{F})
        \\&
        + 2 C k  (\|\phi - \psi\|\nabla \psi \cdot \nabla G, \tilde{F})
         + (\nabla \tilde{u} F +\nabla v \tilde{F}, \tilde{F})
    \end{aligned}
\end{equation}
We will estimate the terms of $G_3$ by testing with $\tilde{F}$. 
By estimating, we bound the term with $(|\tilde{u}|^2+|\tilde{\phi}|^2+\tilde{F}|^2 )$ and higher-order  terms which we will absorb to the other side later.
\begin{equation}
    \begin{aligned}
        &(F_3 , \tilde{F}) 
        \\&\leq 
         Ck ( \|\phi - \psi\|\Delta G,  \tilde{F})
        +2k(\nu'(\phi)\nabla \phi \cdot \nabla F, \tilde{F})
        + 2 \beta k ((\nabla \phi - \nabla \psi)\cdot \nabla G, \tilde{F})
        \\&
        + 2 C k  (\|\phi - \psi\|\nabla \psi \cdot \nabla G, \tilde{F})
        + ( |\nabla \tilde{u}|| F|_{L^{\infty}} + |\nabla v |_{L^{\infty}}|\tilde{F}|) |\tilde{F}|
        \\&
        \leq C|\tilde{\phi}|_2|\Delta G||\tilde{F}|
        + 2k \beta |\nabla \phi|_1\sum_{i,j}|\nabla F^{i,j}|_1|\tilde{F}|+ 2 \beta k |\tilde{\phi}|_1\sum_{i,j}|\nabla G^{i,j}|_1|\tilde{F}|
        \\
        &+
        2 C k |\tilde{\phi}|_2|\nabla \psi|_2|\nabla G|| \tilde{F}|
        +
        ( |\nabla \tilde{u}|| F|_{L^{\infty}} + |\nabla v |_{L^{\infty}}|\tilde{F}|) |\tilde{F}|
        \\&
        \leq
        C(|\Delta \tilde{\phi}|+|\tilde{\phi}|)|\Delta G||\tilde{F}|
        + 2 \beta k |{\phi}|_2 | F|_2|\tilde{F}|
        + 2 \beta k (|\Delta \tilde{\phi}|+|\tilde{\phi}|) | G|_2|\tilde{F}|
        \\&
        +2Ck(|\Delta \tilde{\phi}| + |\tilde{\phi}|)|\psi|_3|G|_1 |\tilde{F}|+\frac{\dot{m}}{6}|\nabla \tilde{u}| +C|F|_2^2|\tilde{F}|^2 + |v|_3^2 |\tilde{F}|^2
        \\&
        \leq
          C(| G|_2^2|\tilde{F}|^2+|\tilde{\phi}|^2 + | G|_2^2 |\tilde{F}|^2)
        +\frac{\dot{m}}{27}|\Delta \tilde{\phi}|^2 
        +C(|{\phi}|_2^2 + | F|_2|\tilde{F}|)
        \\&
        + C(| G|_2^2|\tilde{F}|^2+|\tilde{\phi}|^2 + | G|_2^2 |\tilde{F}|^2)
        +\frac{\dot{m}}{27}|\Delta \tilde{\phi}|^2
        \\&
        +2Ck(|\psi|_3^2|G|_1^2 |\tilde{F}|^2 + |\tilde{\phi}|^2 + |\psi|_3^2|G|_1^2 |\tilde{F}|^2)
        +\frac{\dot{m}}{27}|\Delta \tilde{\phi}|^2
        +\frac{\dot{m}}{6}|\nabla \tilde{u}| +C|F|_2^2|\tilde{F}|^2 + |v|_3^2 |\tilde{F}|^2
        \\&
        \leq  
        \frac{\dot{m}}{6}|\nabla \tilde{u}|^2 +\frac{\dot{m}}{9}|\Delta \tilde{\phi}|^2  + M_3(u(t), \phi(t), F(t))(|\tilde{u}|^2 + |\tilde{\phi}|^2 +|\tilde{F}|^2)
    \end{aligned}
\end{equation}
where $M_3(u(t), \phi(t), F(t))$ is bounded part for the terms of $(u, \phi, F)$ and $(v, \psi, G)$ as we proved in a priori estimate for Theorem \eqref{Thm_local_wp} result regularity. We used Poincare inequality for $\tilde{u}$ from \eqref{poincare_u} and also used \eqref{H2_phi} to estimate $|\tilde{\phi}|_2$.
Then, we can simplify the \eqref{5_3} as follows with Poincare inequality for $\tilde{F}$.
\begin{equation}\label{5_3'}
    \begin{aligned}
        &\frac{1}{2} \dfrac{d}{dt} |\tilde{F}|^2 + \dot{m} |\nabla  \tilde{F}|^2 
        \\&\leq 
        -(\tilde{u} \cdot \nabla F , \tilde{F}) + (G_3, \tilde{F})
        \\&\leq |\tilde{u}|_1|\nabla F|^{\frac{1}{2}}|\nabla F |_1^{\frac{1}{2}}|\tilde{F}| + (G_3, \tilde{F})
        \\&
        \leq 
        \frac{9c^2}{\dot{4m}}(1+\frac{1}{\lambda_1}) |\nabla F||\nabla F|_1|\tilde{F}|^2 
        + \frac{\dot{m}}{9}|\nabla \tilde{u}|^2 
        + (G_3, \tilde{F}) 
        \\&\leq \frac{9c^2}{\dot{4m}}(1+\frac{1}{\lambda_1})|\nabla F||\nabla F|_1|\tilde{F}|^2 + {  \frac{\dot{m}}{9}}|\nabla \tilde{u}|^2 
        \\&
        + \frac{\dot{m}}{6}(|\nabla \tilde{u}|^2) +\frac{\dot{m}}{9}(|\Delta \tilde{\phi}|^2 + |\nabla \tilde{F}|) + M_3(t)(|\tilde{u}|^2 + |\tilde{\phi}|^2 +|\tilde{F}|^2) 
        .
    \end{aligned}
\end{equation}
\\
\par
Combining the equations \eqref{5_1'}, \eqref{5_2'} and \eqref{5_3'} yields
\begin{equation}\label{long_att_pdeineq}
    \begin{aligned}
        & \dfrac{d}{dt} (|\tilde{u}|^2 +|\tilde{\phi}|^2 + |\tilde{F}|^2) 
        \\&+
        - 2(M_1(t) + M_2(t) + M_3(t) + C(| u|_1| u|_2+ |\phi|_1|  \phi|_2 + | F|_1| F|_2) ))(|\tilde{u}|^2 +|\tilde{\phi}|^2 + |\tilde{F}|^2)
        \\&
        \leq 0
    \end{aligned}
\end{equation}
On above inequality \eqref{long_att_pdeineq}, because $2(M_1(t) + M_2(t) + M_3(t) + C(| u|_1| u|_2+ |\phi|_1|  \phi|_2 + | F|_1| F|_2)$ is integrable on time interval $[0, T_0]$ from the normed space conditions in Theorem \eqref{Thm_local_wp}, $(|\tilde{u}|^2 +|\tilde{\phi}|^2 + |\tilde{F}|^2)$ is zero by using Gronwall's Inequality since $|\tilde{u_0}|^2 =|\tilde{\phi}_0|^2 = |\tilde{F}_0|^2 = 0$. \qedsymbol
\\
\par
\section{Numerical simulation on the system}
\subsection{Model developments and Benchmarks}
To proceed the system-stabilizing work on the governing system in \cite{ktt22}, we should have diffusion term for every single solution variable of $u$, $F$ and $\phi$ assigned their role as "spill over" role as in paper \cite{aot14}. 
To stabilize the governing system in \cite{ktt22}, we introduce diffusion terms for the solution variables 
$u, F$ and $\phi$, following the stabilizing role discussed in \cite{aot14}.

In this reference, the diffusion term helps control higher-order derivatives. Therefore, we construct a new model that preserves the energy-dissipation structure and local well-posedness. Before the steps to obtain this new model including second degree of differential as diffusion, we introduced more generalized viscoelasticity parameter $\nu(\phi)$ from $\lambda_e$ in \cite{ktt22}.
\\
\begin{figure}[htbp]
  \centering
    \begin{subfigure}[t]{0.495\textwidth}
      \centering
      \includegraphics[width=\linewidth,trim=10 20 10 20]{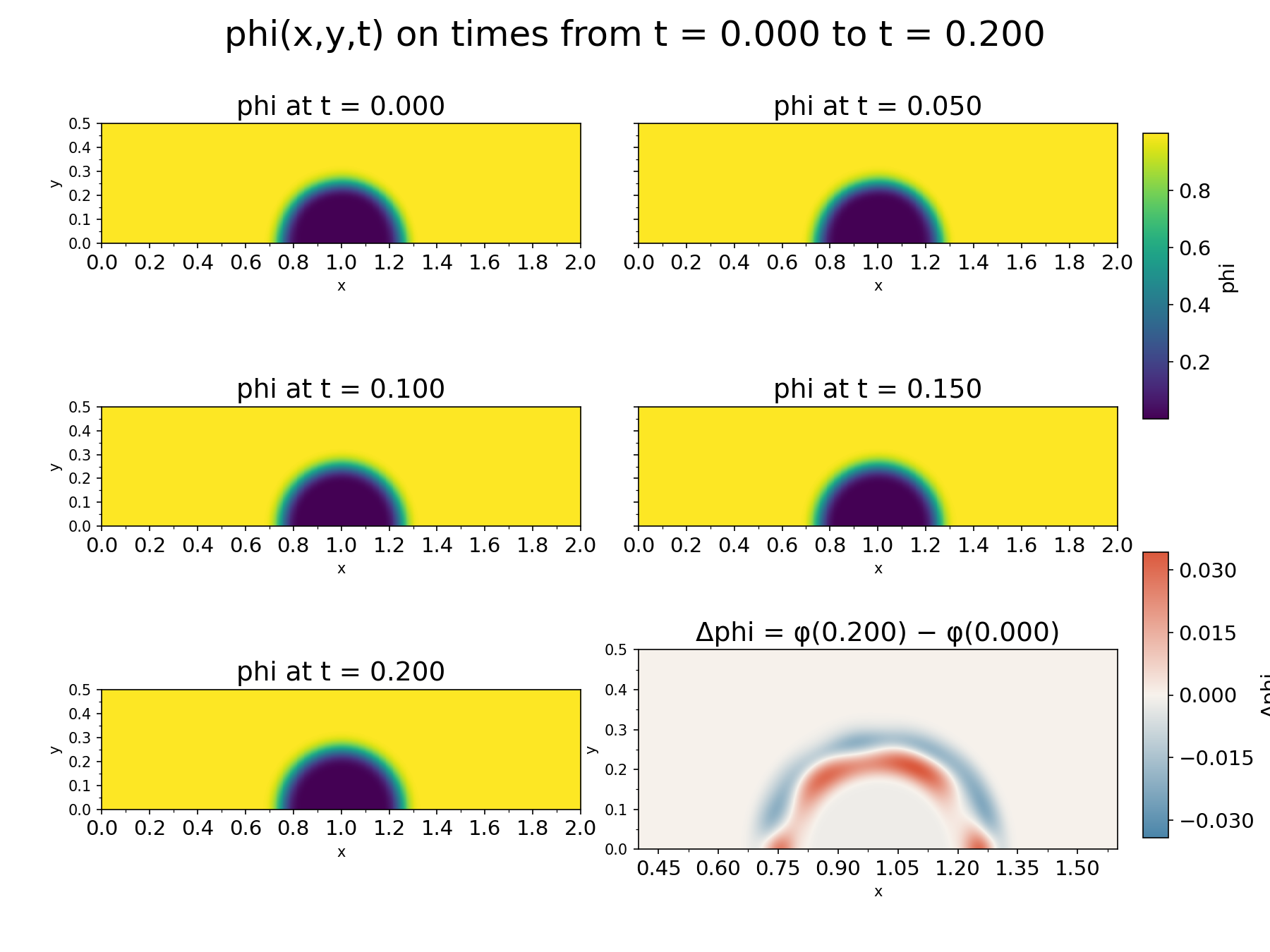}
      \caption{}
    \end{subfigure}
    \begin{subfigure}[t]{0.495\textwidth}
      \centering
      \includegraphics[width=\linewidth,trim=10 20 10 20]{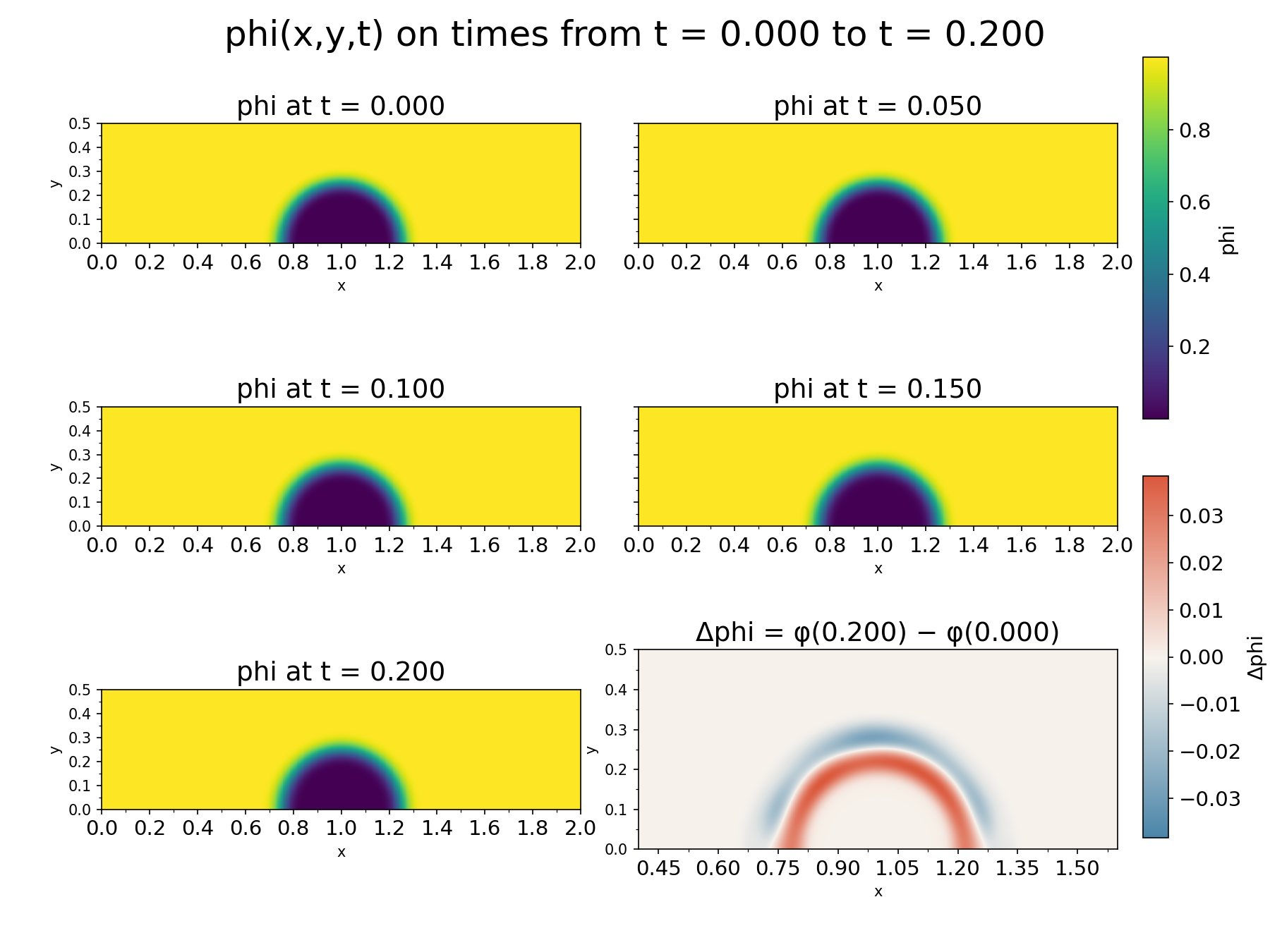}
      \caption{}
    \end{subfigure}
    \caption{On time domain [0, .2], phase field variable $\phi$ (=0 on clot and =1 on blood) on (a)original thrombus system with $\lambda_e=1$ on thrombus and $\lambda_e=0$ on blood in \cite{ktt22}, (b)modified new model with generalized viscoelasticity $\nu(\phi)= (1-\phi) + 10^{-5}\phi$. (a) displays mixed and ambiguous phase field shape change, not describing interplay of dynamics on shock region. However, (b) shows more stable formation for the interface change for the $\phi$ as better performance to catch the quickly changing gradient phenomenon on the shock. }
    \label{fig_compare_clot}
\end{figure}
\\
\\
To be specific, the following system is the original model of this paper's governing system.(\cite{ktt22})
\\
\par
$<$Original model$>$
\begin{equation}\label{ktt_ge}
\left\{ 
\begin{array}{rcl}
\rho(\frac{du}{dt} + u \cdot \nabla u) + \nabla p - \nabla \cdot (\eta(\phi) \nabla u) &=& -\lambda \nabla \cdot (\nabla \phi \otimes \nabla \phi) \\[5pt]
+ \nabla \cdot \left(\lambda_e(1 - \phi)(FF^T - I)\right) &- &\dfrac{\eta(\phi)(1-\phi)u}{\kappa(\phi)}, \\[5pt]
\nabla \cdot u &=& 0, \\[5pt]
\frac{dF}{dt} +u \cdot \nabla F &=& \nabla u F, \\[5pt]
\frac{d \phi}{dt}  + u \cdot \nabla \phi &=& \tau \Delta \mu, \\[5pt]
\mu &=& -\lambda \Delta \phi + \lambda \gamma f'(\phi) - \dfrac{\lambda_e}{2} \text{tr}(FF^T - I),
\end{array}
\right.
\end{equation}
The $\lambda_e$ parameter is assigned positive value only on thrombus and is zero on blood, i.e. as positive constant viscoelasticity for thrombus region of phase field variable $\phi = 0$. This is not physically plausible. Therefore, by introducing the viscoelasticity variable $\nu(\phi)$ distributing positive constants on the pure phase fields of blood and thrombus, we build new model as follows. This yields the energy-dissipation law \eqref{2_27}, which is physically more reasonable since the elastic energy associated with $\nu(\phi)$ is no longer zero in the blood phase. This generalization is written as
\\
\par
$<$Modified model$>$
\begin{equation}\label{nu_ge}
\left\{ 
\begin{array}{rcl}
\rho(\frac{du}{dt} + u \cdot \nabla u) + \nabla p - \nabla \cdot (\eta(\phi) \nabla u) &=& -\lambda \nabla \cdot (\nabla \phi \otimes \nabla \phi) \\[5pt]
+ \nabla \cdot \left(\pmb{\nu(\phi)}(FF^T - I)\right) &- &\dfrac{\eta(\phi)(1-\phi)u}{\kappa(\phi)}, \\[5pt]
\nabla \cdot u &=& 0, \\[5pt]
\frac{dF}{dt} +u \cdot \nabla F &=& \nabla u F, \\[5pt]
\frac{d \phi}{dt}  + u \cdot \nabla \phi &=& \tau \Delta \mu, \\[5pt]
\mu &=& -\lambda \Delta \phi + \lambda \gamma f'(\phi) + \dfrac{\pmb{\nu'(\phi)}}{2} \text{tr}(FF^T - I),
\end{array}
\right.
\end{equation}
where $\nu$ satisfies \eqref{visco_diff}.
Thus, substituting $\lambda_e (1 - \phi)$ into $\nu(\phi)$ in the governing equation \eqref{nu_ge} results in the previous model \eqref{ktt_ge} exactly. By building this modified model, we can successfully simulate the mixed region of blood and thrombus as the Figure \ref{fig_compare_clot}. In the figure, we can affirm that the interface region for the thrombus and blood can be more sensitively expressed in the modified model due to introducing $\nu(\phi)$ variable. This is because, in the original model, the viscoelasticity was zero on the pure blood region which does not physically make sense. To be specific, zero viscoelasticity on blood resulted in zero elastic energy on the pure blood region from the following total energy formula 
\begin{equation}\label{total_e}
    E(x,t) := \int_{\Omega} |u|^2 + \lambda |\nabla \phi|^2 + 2 \lambda \gamma f(\phi)  + { \nu(\phi)} tr(FF^T -I ) dx
\end{equation}
as in \eqref{2_27}. Elastic energy means the potential energy from the deformation variable $F$ to be stretched as time goes by and this leads the movement of the $F$ until it is equal to $I$ at the static equilibrium status. Therefore, setting the viscoelasticity as zero on pure blood field is not sound assumption physically to describe the total energy and evolutionary movement of dynamics of the original thrombus system \eqref{ktt_ge}. As a result of this discussion, we was able to see the development on the better simulation profile on $\phi$ on the interface shock area which is challenging for almost numerical simulations on phase field governing system.
\\
\par
Next, we developed the model continuously as \eqref{governing_system} from the modified one \eqref{nu_ge}, adding the diffusion term $\pmb{k (\nu (\phi) \Delta F +2 \nabla \nu (\phi) \cdot \nabla F )}$. As we already discussed in the previous chapters, this new model satisfied energy dissipation as \eqref{2_27} and local well-posedness as Theorem \eqref{Thm_local_wp} with added neumann boundary conditions as \eqref{bdry}.
\\
\\
\begin{figure}[H]
  \centering
    \begin{subfigure}[t]{0.495\textwidth}
      \centering
      \includegraphics[width=\linewidth,trim=35 20 10 20]{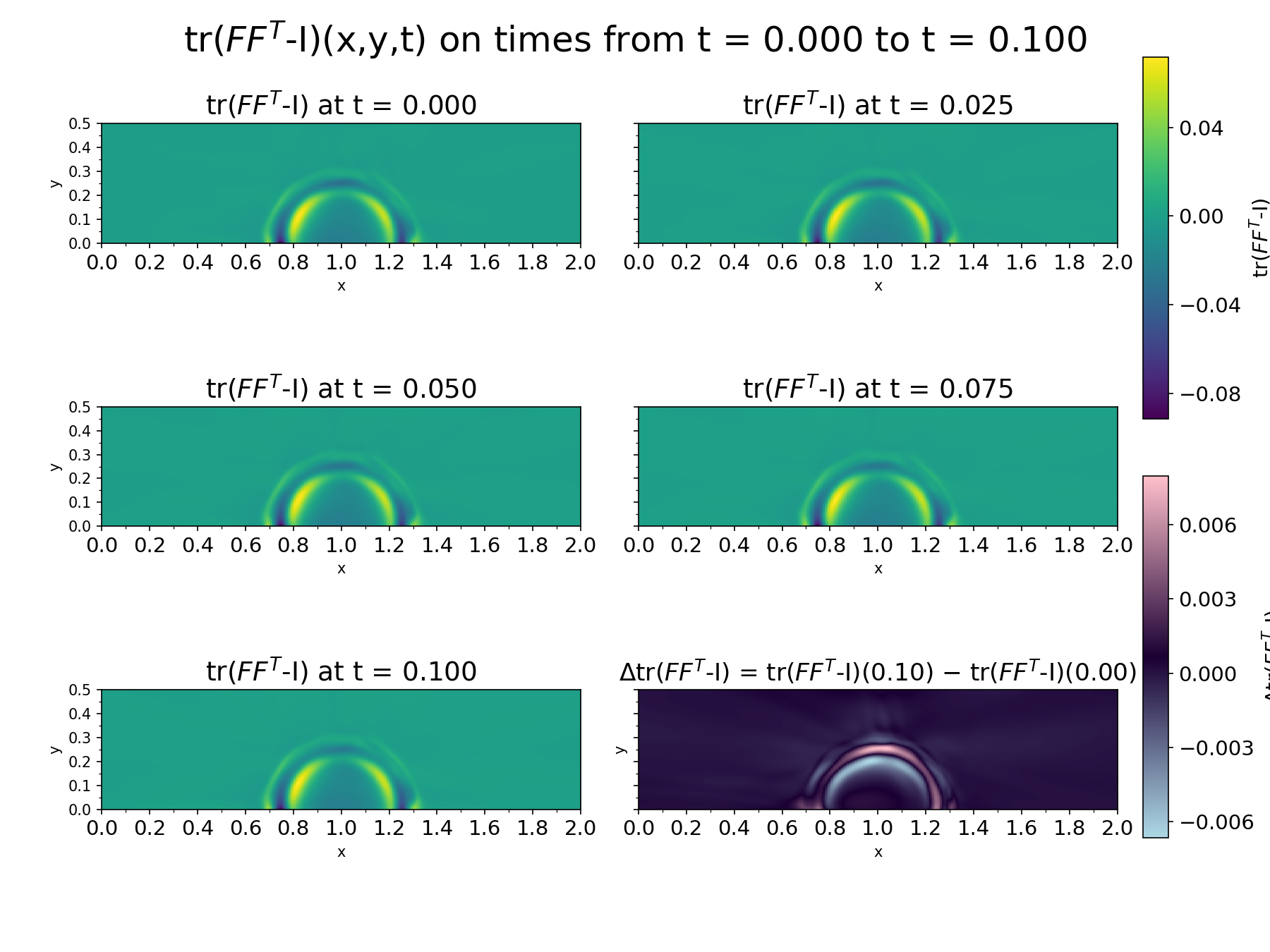}
      \caption{}
    \end{subfigure}
    \begin{subfigure}[t]{0.495\textwidth}
      \centering
      \includegraphics[width=\linewidth,trim=13 20 20 20]{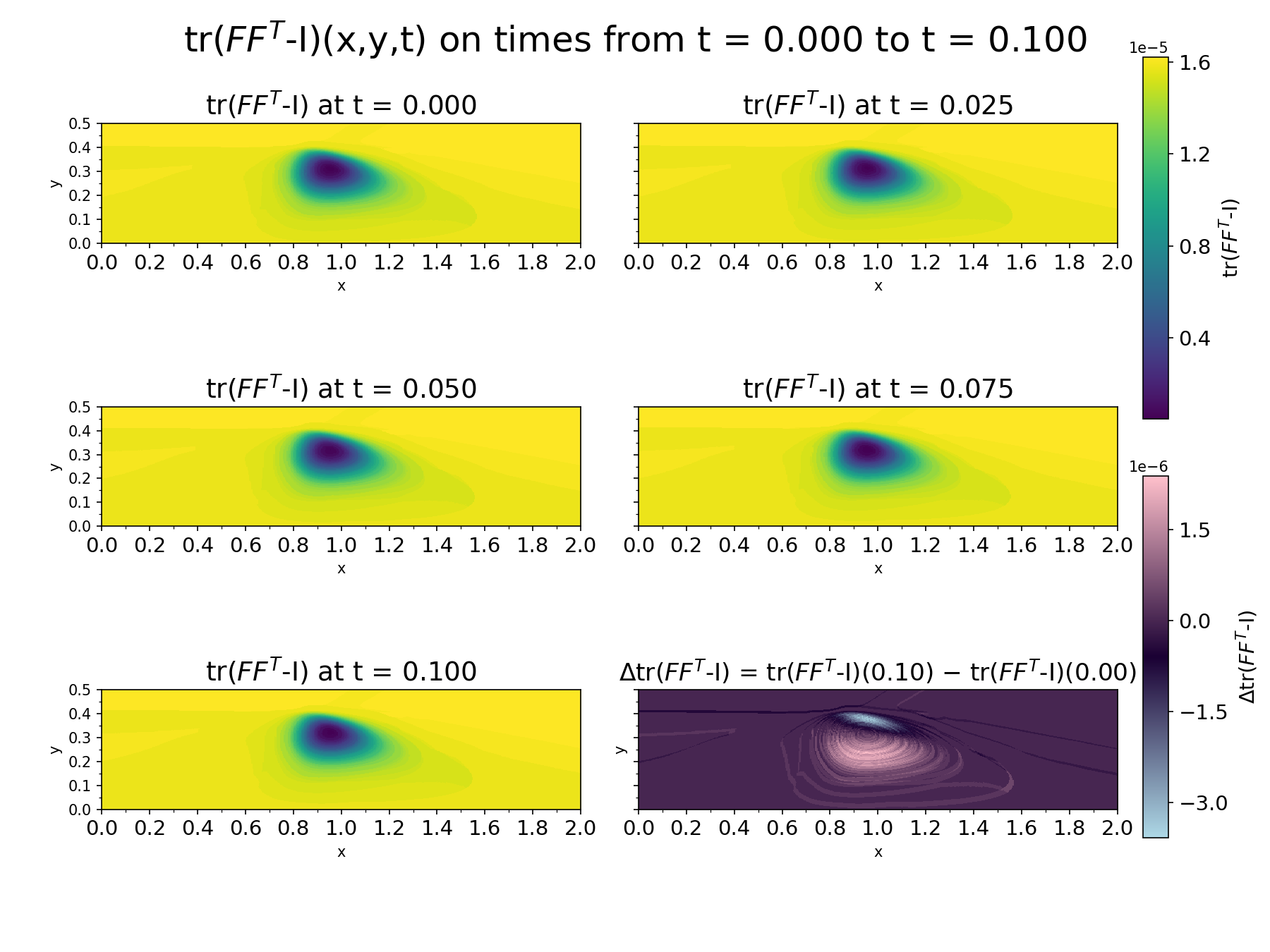}
      \caption{}
    \end{subfigure}
    \caption{Change of $tr(FF-I)$ (a)\eqref{nu_ge} and (b)\eqref{governing_system}.}
    \label{fig:F_change}
\end{figure}
Comparing Figures \ref{fig:F_change}, \ref{fig:phi_change} we see that diffusion slows the variation of $F$ away from the initial state $F = I$, more than it does to $\phi$, specifically, the variation in F decreases from around $10^{-3}$
 to around $10^{-6}$ level whereas that of $\phi$ stays near $10^{-3}$.
We interpret this as the diffusion by absorbing the change of $F$ on the interface area via equation $\eqref{new_system}_3$ by mixing up the values of $F$. As shown in the Figure \ref{fig:F_change}, another reason for the reduced variation in $F$ is the presence of the Neumann boundary condition.
\begin{figure}[H]
  \centering
    \begin{subfigure}[t]{0.495\textwidth}
      \centering
      \includegraphics[width=\linewidth,trim=35 20 10 20]{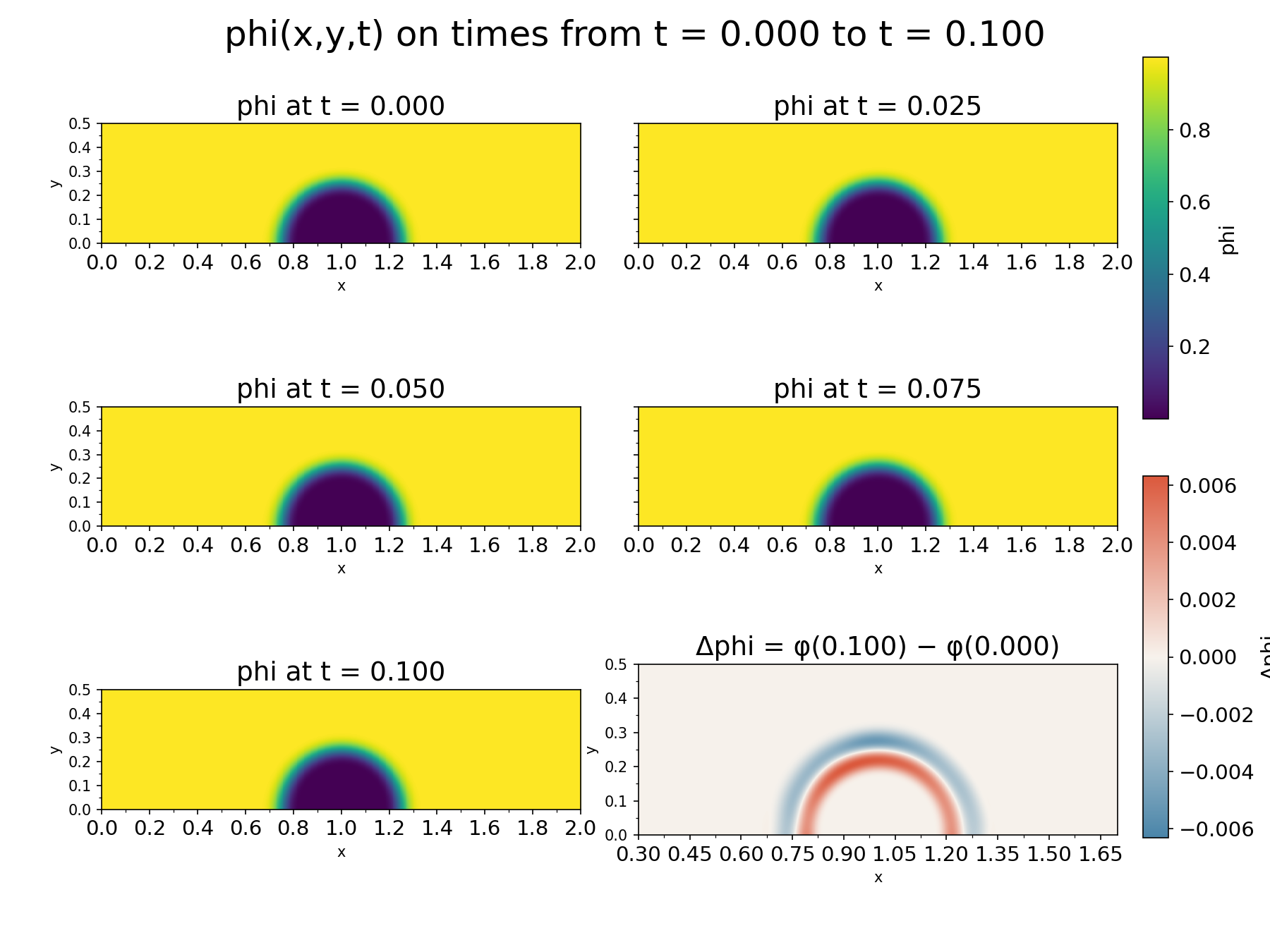}
      \caption{}
    \end{subfigure}
    \begin{subfigure}[t]{0.495\textwidth}
      \centering
      \includegraphics[width=\linewidth,trim=13 20 20 20]{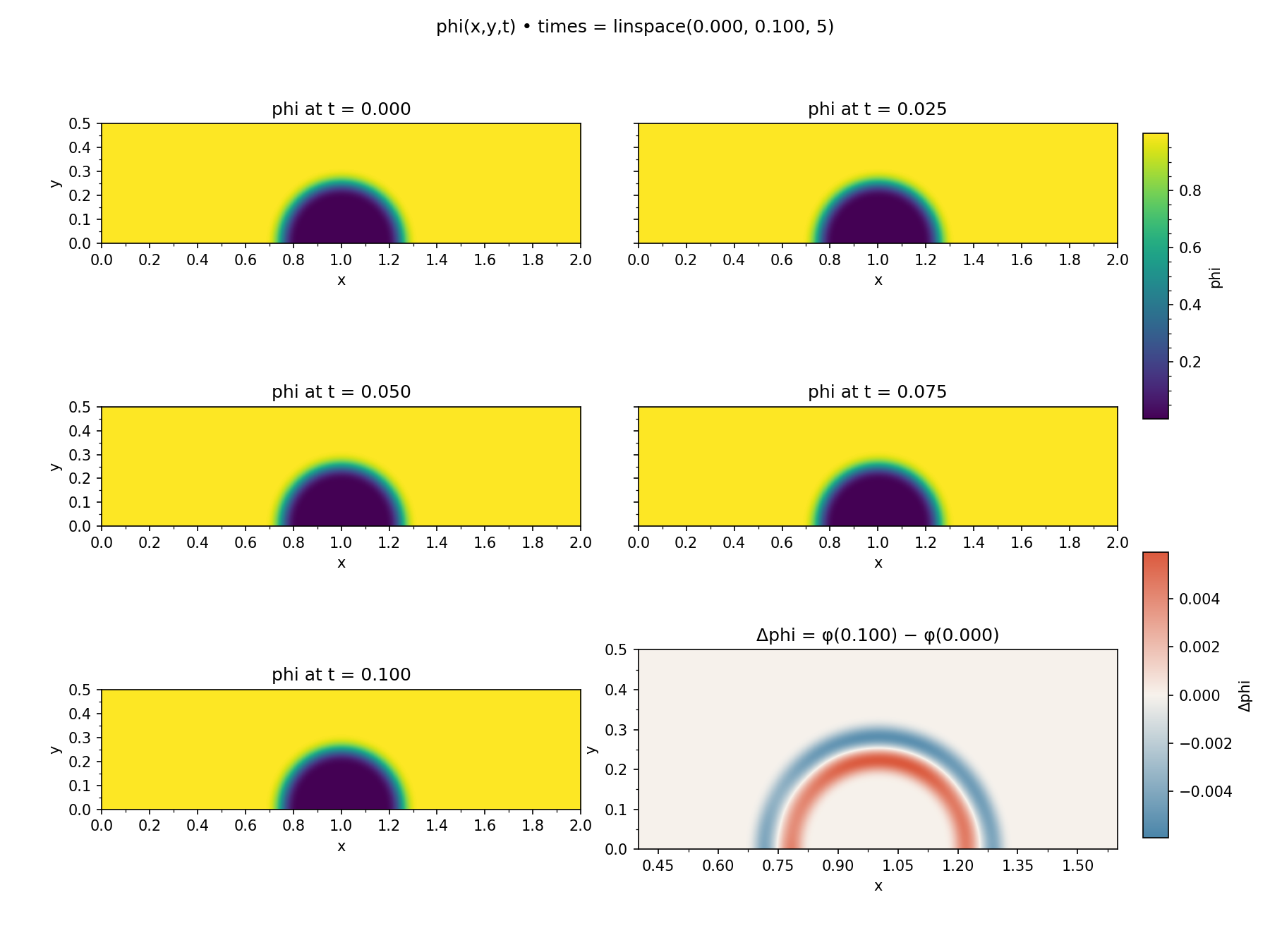}
      \caption{}
    \end{subfigure}
    \caption{Change of $\phi$ (a)\eqref{nu_ge} (b)\eqref{governing_system}}.
    \label{fig:phi_change}
\end{figure}
We next study the behavior of \eqref{governing_system} under different settings. We take the viscoscity to be $\eta = (\phi)\eta_b + (1 - \phi)\eta_t$ and permeability to be $\kappa = (\phi)\kappa_b + (1 - \phi)\kappa_t$, wherer the subscripts $b$, $t$ refer to blood thrombus, respectively.
\\
\subsection{Neural network architecture and technical training setting} For the neural network architecture, we used a network with 10 layers and 128 neurons per layer. The input variables are
 $(x, y, t)$, and the outputs are $(u_1, u_2, \phi, F^{11}, F^{12}, F^{21}, F^{22}, p)$, corresponding to the variables in the governing system \eqref{governing_system}. For one training window, we sampled 50,000 collocation points in the space-time domain for the PDE residual, 1,000 collocation points for the initial condition, and 800 collocation points on the boundary. The sampling procedure consists of Latin hypercube sampling in space and uniform sampling in time.
We used the Adam optimizer with the hyperbolic tangent activation function.
\\
For the learning-rate schedule, we used 250 epochs at $10^{-3}$
, 200 epochs at $10^{-4}$, and 100 epochs each at $10^{-5}$, $10^{-6}$, and $10^{-7}$. For each epoch, a mini-batch size of $500$ was used for the $50,000$ domain collocation points, resulting in $100$ iterations per epoch.
\\
\par
In this paper, all numerical simulations were carried out using the window-sweeping method(\cite{sweeping23}). This method is based on the time-marching method(See \cite{ac21}), in which PINN training is performed on small subintervals obtained by splitting the time domain.
\\
In \cite{ac21}, the time-marching method is introduced in two forms. First, it splits the whole time domain into unit intervals and trains the PINN on progressively increasing time intervals. Second, on these equally partitioned unit intervals, it uses transfer learning to pass the learned weights from one time interval to the next.
\\
Based on these ideas, the window-sweeping method uses overlapping training windows for consecutive time intervals, as shown in Figure \ref{sweeping}. After training the first initiating time domain via initial condition loss term of given initial condition and other loss terms, the training window then moves to the next unit time interval, as shown in Figure \ref{sweeping}. It consists of a blue block representing the previously trained time interval and a pink block representing the newly trained interval. Let us denote this one block unit train time interval as $\Delta_{unit} t $ and it has default value $0.05$ without specific description in this paper.
\\
\\
On the blue block, which overlaps with the previous time interval, transfer learning enables us to inherit the outputs
$$w_{\text{pred}}^ {\theta^-}(x_i,t_i) \in U^- :=\{ u^{\theta^-}_1,  u^{\theta^-}_2,  \phi^{\theta^-},  f^{\theta^-}_{ij},  p^{\theta^-} (i,j = 1, 2)\}$$ from the previous training stage by minimizing the following loss term:
\begin{equation}\label{sweeping_bc_loss}
    Loss_{\text{com}} = \sum_U^- \sum_i(w_{\text{pred}}^ {\theta^-}(x_i,t_i) - w_{\text{pred}}^{\theta^+}(x_i,t_i))^2 / \text{(number of collocation points)}
\end{equation}
Here,
$$w_{\text{pred}}^ {\theta^+}(x_i,t_i)\in U^+ :=\{ u^{\theta^+}_1,  u^{\theta^+}_2,  \phi^{\theta^+},  f^{\theta^+}_{ij},  p^{\theta^+} (i,j = 1, 2)\}$$ denotes the current training outputs corresponding to $w_{\text{pred}}^ {\theta^-}$ at the collocation points $(x_i,t_i)$. We refer to this as backward compatibility. 
\\
When training starts from a nonzero time, this term replaces the initial-condition loss in the total loss. In addition, over the whole block consisting of the blue and pink intervals, we minimize the squared $L^2$ loss for the standard PDE residual:
\begin{equation}
    Loss_{\text{PDE}} = \sum_{i, U^+}(\text{Residual}(w_{\text{pred}}^{\theta^+}(x_i,t_i)))^2 / \text{(number of collocation points)}.
\end{equation}
Here, Residual denotes the sum of the residuals of all equations in the governing system \eqref{governing_system}. Therefore, the total PINN loss is defined by combining all of these terms. This approach allows us to split the time interval into sufficiently small subintervals, which improves accuracy by tracking the dynamics at an appropriate temporal scale.
\\
\par
This procedure is enabled by the transfer-learning mechanism in \eqref{sweeping_bc_loss}. The network weights obtained from the previous time interval are loaded before training on the current time interval begins. As a result, the outputs on randomly sampled collocation points in the backward-compatibility interval remain consistent with those obtained on the previous interval.\\
In this paper, unit train time interval $\Delta_{unit} t$ is set as $0.05$ or $0.1$ with the transfer learning time interval $0.025$ as backward compatibility time interval.
\\
\begin{figure}[htbp]
  \centering
  \begin{subfigure}[t]{0.63\textwidth}
      \centering
      \includegraphics[width=\linewidth,trim=10 5 10 0]{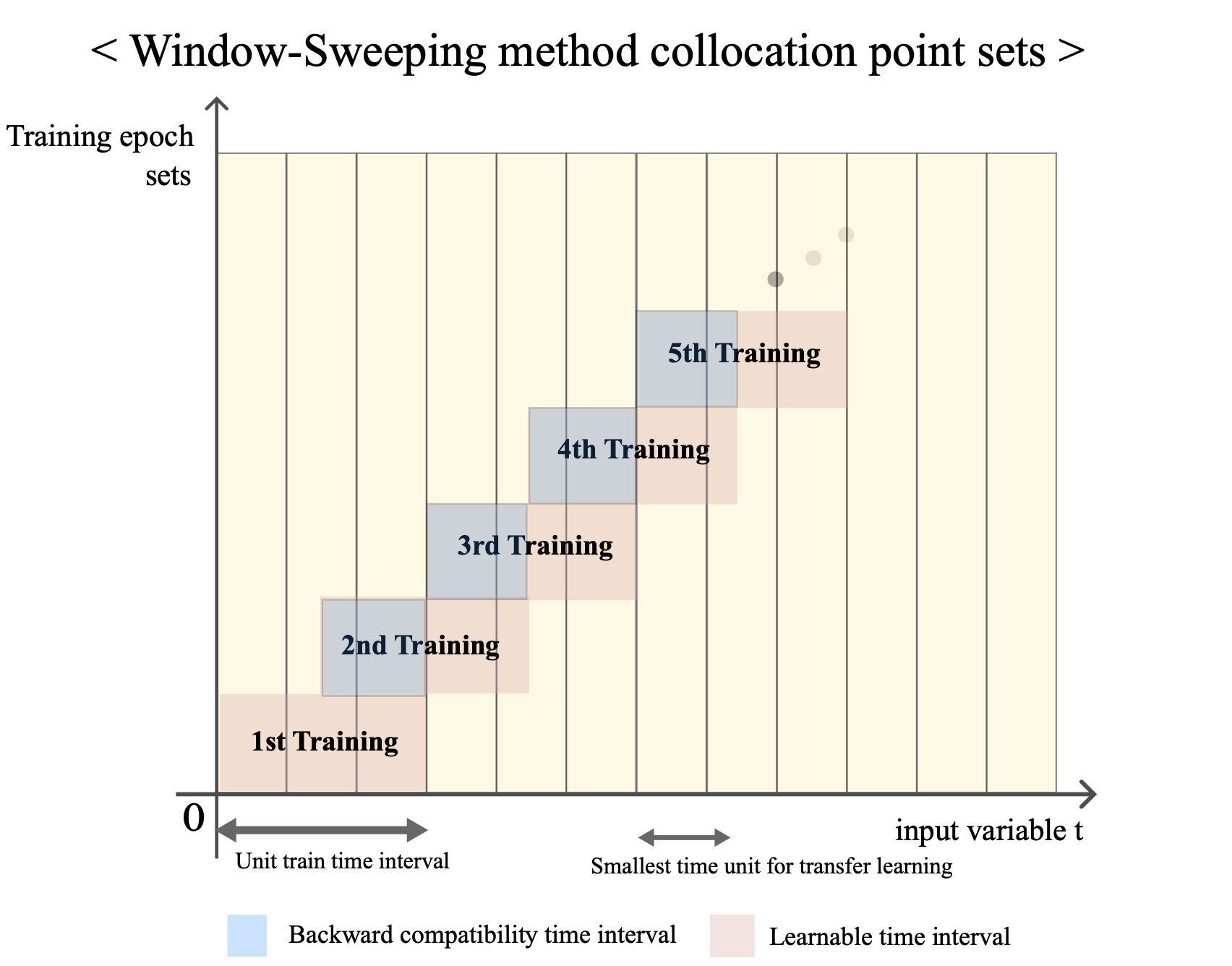}
    \end{subfigure}
    \caption{  }
    \label{sweeping}
\end{figure}
\\
Combining all methods we explained, we define the total loss for the current network parameters $\theta^+$ and the corresponding outputs $u_{\text{pred}}^{\theta^+}(x_i,t_i)$, $\phi_{\text{pred}}^{\theta^+}(x_i,t_i)$, $F_{\text{pred}}^{\theta^+}(x_i,t_i)$, and $p_{\text{pred}}^{\theta^+}(x_i,t_i)$ as follows.
\begin{equation}\label{total_loss}
\begin{aligned}
&Loss_{\text{tot}} := 1000 Loss_{\text{initial}} + 500 Loss_{\text{PDE}} + Loss_{\text{boundary}}
\\& \text{where}
\\&
Loss_{\text{PDE}} = \sum_{j = 1}^5 \sum_i|\text{Residual}_j(u_{\text{pred}}^{\theta^+}(x_i,t_i), \phi_{\text{pred}}^{\theta^+}(x_i,t_i),F_{\text{pred}}^{\theta^+}(x_i,t_i) )|^2/N_{\text{PDE}}
\\& \text{\qquad for the PDE equation residuals from } \eqref{governing_system}_1 - \eqref{governing_system}_4\text{, together with the constraint det}F = 1,
\\&
Loss_{\text{boundary}} = \sum_i(
|u_{\text{pred}}^{\theta^+}(x_i,t_i)- (0,0)|^2
+ |\partial_n \mu(u_{\text{pred}}^{\theta^+}, \phi_{\text{pred}}^{\theta^+}, F_{\text{pred}}^{\theta^+})- 0|^2
\\&+ |\partial_n \phi_{\text{pred}}^{\theta^+}(x_i,t_i)- 0|^2
+ |\partial_n \Delta F_{\text{pred}}^{\theta^+}(x_i,t_i)- 0|^2
+ |\partial_n  F_{\text{pred}}^{\theta^+}(x_i,t_i)- 0|^2
)/N_{\text{boundary}}
\\& \text{\qquad which correspond to the boundary condition residuals in } \eqref{bdry}, 
\\& Loss_{\text{initial}} = 
\begin{cases}
(|u_{\text{pred}}^{\theta^+}(x_i,t_i)- (0, 0)|^2
+ |\phi_{\text{pred}}^{\theta^+}(x_i,t_i)- \phi_0|^2
+ | F_{\text{pred}}^{\theta^+}(x_i,t_i)- I|^2)//N_{\text{initial}} &\\ \qquad \text{if the training time interval starts from zero},
\\
\\
\sum_i(\sum_j((u_j)_{\text{pred}}^ {\theta^-}(x_i,t_i) - (u_j)_{\text{pred}}^{\theta^+}(x_i,t_i))^2
+\sum_{j,k}^2 ((f_{jk})_{\text{pred}}^ {\theta^-}(x_i,t_i) - (f_{jk})_{\text{pred}}^{\theta^+}(x_i,t_i))^2
&\\  \qquad +(\phi_{\text{pred}}^ {\theta^-}(x_i,t_i) - \phi_{\text{pred}}^{\theta^+}(x_i,t_i))^2
+(p_{\text{pred}}^ {\theta^-}(x_i,t_i) - p_{\text{pred}}^{\theta^+}(x_i,t_i))^2
)/N_{\text{initial}}, 
&\\  \qquad \qquad \text{if the training time interval starts from nonzero time}
\end{cases}
\end{aligned}
\end{equation}
where $N$ is collocation points number for each subscript.
For the initial loss term, we applied window-sweeping method to use the transfer learning on the backward compatibility time interval as \eqref{sweeping_bc_loss}. In every numerical experiment, we set the initial conditions $u_0 = 0$ and $F_0 = I$ as reflected in the initial-condition loss term. For $\phi_0$, in one thrombus cases, we use
\begin{equation}
    \phi_0 = 0.5  (1 - (1- 10^{-12})*\text{tanh}(2.6(- \sqrt{(x - x_0)^2 + (y - y_0)^2} + R) / (\sqrt{8}  h)))
\end{equation}
where $x_0 = 1$, $y_0 = 0$ and $R = 0.25$. And, in two thrombi cases, we use scaler of the following smooth minimum function
\begin{equation}
    \phi_0 = (\zeta_0 \zeta_1) / (\zeta_0 + \zeta_1 - \zeta_0 \zeta_1 + 10^{-20})
\end{equation}
where $\zeta_0$ is $\phi_0$ where $x_0 = 1-0.23$, $y_0 = 0$, $R = 0.25$ and $\zeta_1$ is $\phi_0$ where $x_0 = 1+0.23$, $y_0 = 0$, $R = 0.25$.
\\
\par
The PDE solver PINN uses PyTorch. The numerical simulations were carried out on the BigRed200 supercomputer using NVIDIA A100 GPUs at Indiana University.
\\
\subsection{Total energy-adaptive sampling for Auto-adaptive PINN}\label{AAPINN}
Catching the movement of phase field variable $\phi$ on shock interface is one of the most challenging part in phase-field dynamics simulation. Total energy is important to simulate this since $\phi$ evolves toward to decreasing this energy though it has sharp change on some domain area. To address this problem, for some difficult cases, Auto-adaptive PINN in \cite{bk25} employed MCMC Metropolis Hasting approximation based on total energy function. This means that, rather than we use uniform sampling distribution, we use density function related with energy function to sample the training collocation points, and throw more collocation points on rapidly changing energy area. Hence, we address the difficult simulation on highly changing energy region of shock interface. Then, on PDE loss term on the domain $\Omega$, we use Metropolis-Hasting approximating on the density $\rho_{density}$ as
\begin{equation}
\int_0^T \int_{\Omega} Residual(w_{\text{pred}}^ {\theta^+}) \text{ }\rho_{density}(w_{\text{pred}}^ {\theta^+}) dx dt
\end{equation}
as the reference \cite{bk25} used this formula.
\\
\par
Now, we give specific energy formula as the reference which is approximated and related with $\rho_{density}$ on spatial sampling points. This is because, as application of this method on our NSCH thrombus system, we will use more spatial collocation points on the higher total energy \eqref{Total_Energy} for sampling the PINN training points.
\\
In our multi-variable coupling NSCH system, to prevent blowing up of simulation output variable $\phi$ on the thrombus-blood interface shock area, we should find the direction of $E(x,t)$ in which it decreases most quickly. To resolve this via differentiating $E(x,t)$, it is impossible since this is integrated value. In this situation, we conventionally use first variation of total energy functional for the $\phi$.
\\
When we denote first variation for $E(\mathbf{x},t)$ with respect to $\phi$ as $\frac{\delta E(\phi, \cdot)}{ \delta \phi}: \psi \rightarrow \frac{d}{d \epsilon} E(\phi + \epsilon \psi) \big|_{\epsilon = 0} $ where $\psi \in L^2(\Omega)$,
\begin{equation}\label{E_x_PINN}
    \begin{aligned}
        &\frac{\delta E(\phi, \psi)}{ \delta \phi} = \frac{d}{d \epsilon}E(\phi + \epsilon \psi)\big|_{\epsilon = 0} 
        \\&
        = \frac{1}{d \epsilon}(\int_{\Omega} \lambda |\nabla (\phi + \epsilon \psi)|^2 + 2 \lambda \gamma f(\phi + \epsilon \psi)  + { \nu(\phi)} tr(FF^T -I ) dx) \big|_{\epsilon = 0} 
        \\&
        = \int \lambda (2 \nabla \phi ) \nabla \psi  + 2 \lambda \gamma f'(\phi)  \psi+ \nu'(\phi) tr(FF^T -I ) \psi dx 
        \\&
        = \int (- \lambda 2 \Delta \phi     + 2 \lambda \gamma f'(\phi) + \nu'(\phi) tr(FF^T -I )) \psi dx 
        \\& = \int 2 \mu \psi dx.
    \end{aligned}
\end{equation}
In the result of the computation, $\frac{\delta E(\phi, \cdot)}{ \delta \phi} = 2(\mu, \cdot)$ and this means the change of $E(x,t)$ comes from perturb of $\phi$. In dynamics system, it is well known that chemical potential $\mu$ is driving force for the evolutionary dynamics of $\phi$. Thus, the last term above agrees with this meaning in physics.
\\
\par
Based on the result \eqref{E_x_PINN}, by focusing on mixed energy term of \eqref{E_decom}, we can take $\psi = \nabla \phi$ to see 
\begin{equation}\label{AA_e1}
    \begin{aligned}
        &\int_{\Omega} (- \lambda 2 \Delta \phi     + 2 \lambda \gamma f'(\phi) + \nu'(\phi) tr(FF^T -I )) \nabla \phi dx
        \\& =
        \int_{\Omega} 2 (\lambda \nabla \phi \nabla^2 \phi + \lambda \gamma f'(\phi) \nabla \phi + \nu'(\phi) \nabla \phi tr(FF^T -I )) dx
    \end{aligned}
\end{equation}
To determine input for this first variation functional to use it as the target energy in Metropolis hasting process for AA-PINN, observe gradient of total energy $E(x,t)$ with respect to each spatial variable is composed of
\begin{equation}\label{AA_e2}
    \begin{aligned}
        &\nabla E(x,t)
        \\&=\nabla (|u|^2 + \lambda |\nabla \phi|^2 + 2 \lambda \gamma f(\phi) + \nu'(\phi)tr(FF^T-I))
        \\&= 2 (u \nabla u +\lambda \nabla \phi \nabla^2 \phi + \lambda \gamma f'(\phi) \nabla \phi + \nu'(\phi) \nabla \phi tr(FF^T -I )+ \sum_{i,j} \nu (\phi) \nabla F^{i,j} F^{i,j} )
    \end{aligned}
\end{equation}
\\
Considering that the main changing part in the shock region comes from $\nabla \phi$, coincident part associated with $\nabla \phi$ in the formula \eqref{AA_e1} and \eqref{AA_e2} in weak sense is the target energy function in the Metropolis-Hasting process to approximate probability density on sampling training points in $\Omega$. More specifically, for each time uniformly distributed, the function 
\begin{equation}
    E(x,t)_{AA} := 2 (\lambda \nabla^2 \phi + \lambda \gamma f'(\phi)+ \nu'(\phi) tr(FF^T -I ))\nabla \phi 
\end{equation}
is used as point-wise energy density on spatial domain as in section 2.2 of the AA-PINN paper \cite{bk25}. Based on this modified formula, we follow PINN training steps for the Algorithm 1 in the paper.
\\
\par
By using this variation of AA-PINN method, in the challenging cases to simulate shock region of $\phi$ which are marked as (E) in the section \ref{case_study}, our PINN solver narrows down the difference between predicted output ${\phi}^{\theta}$ of neural network training and PINN output solution $\phi$ satisfying the governing system. In the case with two thrombi in the section, for example, we can check the initial $\phi$ configuration is challenging work to see the clean simulation output as fitted with initial condition for the $\phi$ (See Figure \ref{clots_initial_phi_E} (a)). This is because the two thrombi initial configuration is very far from the solution formula form for minimizing the mixed energy, which is known as the tangent hyperbolic function profile as shaping of one thrombus. The minimizing formula has the graph as we discussed in figure \ref{fig_compare_clot} (b).
\\
\\
To overcome this problem, we applied AA-PINN energy-adaptive sampling method as we discussed. Figure \ref{fig_sampling_E} show the sampling distribution, which approximates the energy function $E(x,t)_{AA}$ well.
\\
After applying this sampling, on Figure \ref{clots_initial_phi_E}, the $L^{\infty}(\Omega)$ error relative to initial configuration condition $\phi \big|_{t = 0}$ decreased {from $0.02560$ to $0.00418$, corresponding to a $75.0412 \%$ redduction.}
\begin{figure}[H]
  \centering
    \begin{subfigure}[t]{0.495\textwidth}
      \centering
      \includegraphics[width=\linewidth,trim=35 20 10 20]{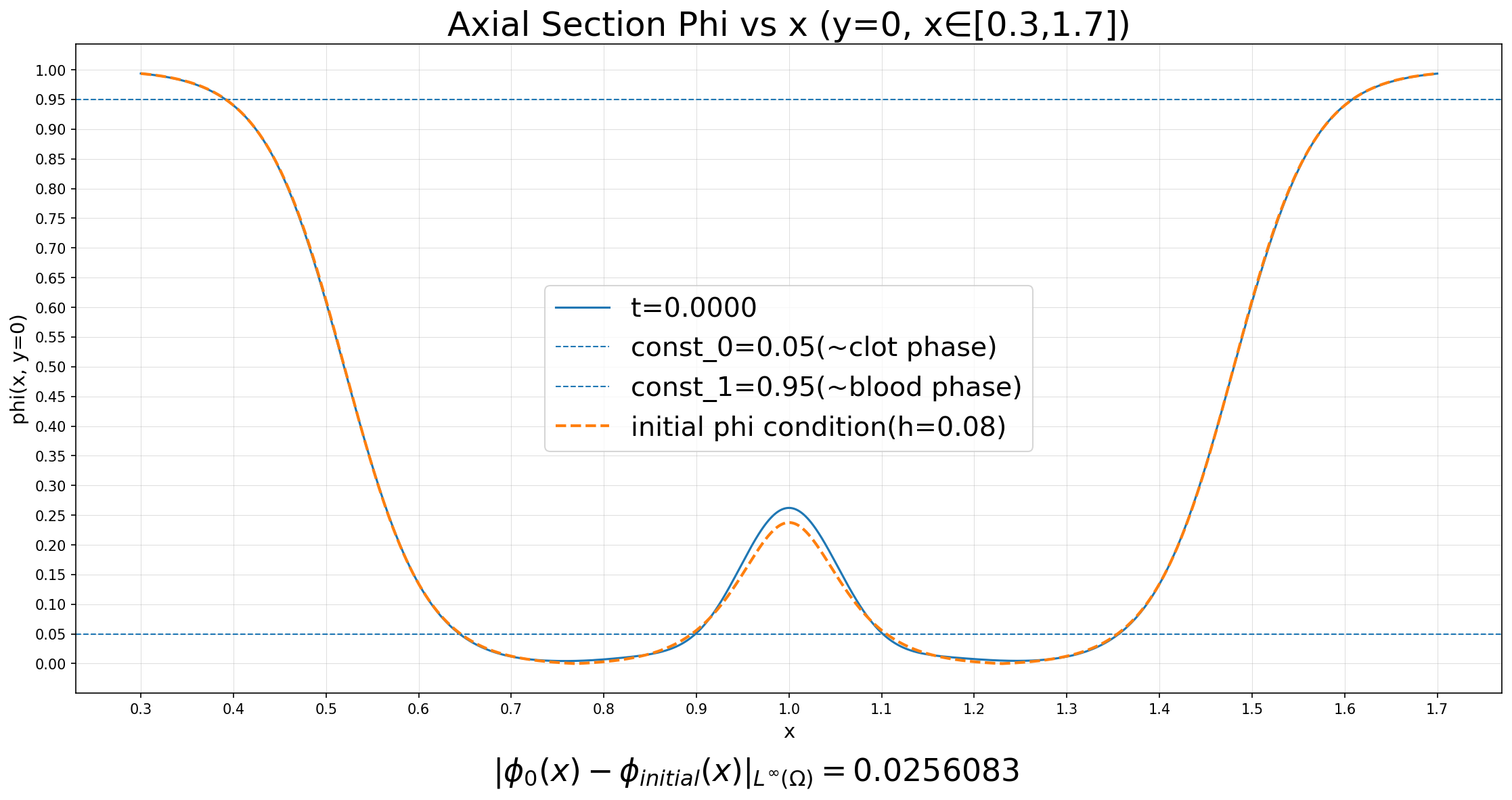}
      \caption{}
    \end{subfigure}
    \begin{subfigure}[t]{0.495\textwidth}
      \centering
      \includegraphics[width=\linewidth,trim=13 20 20 20]{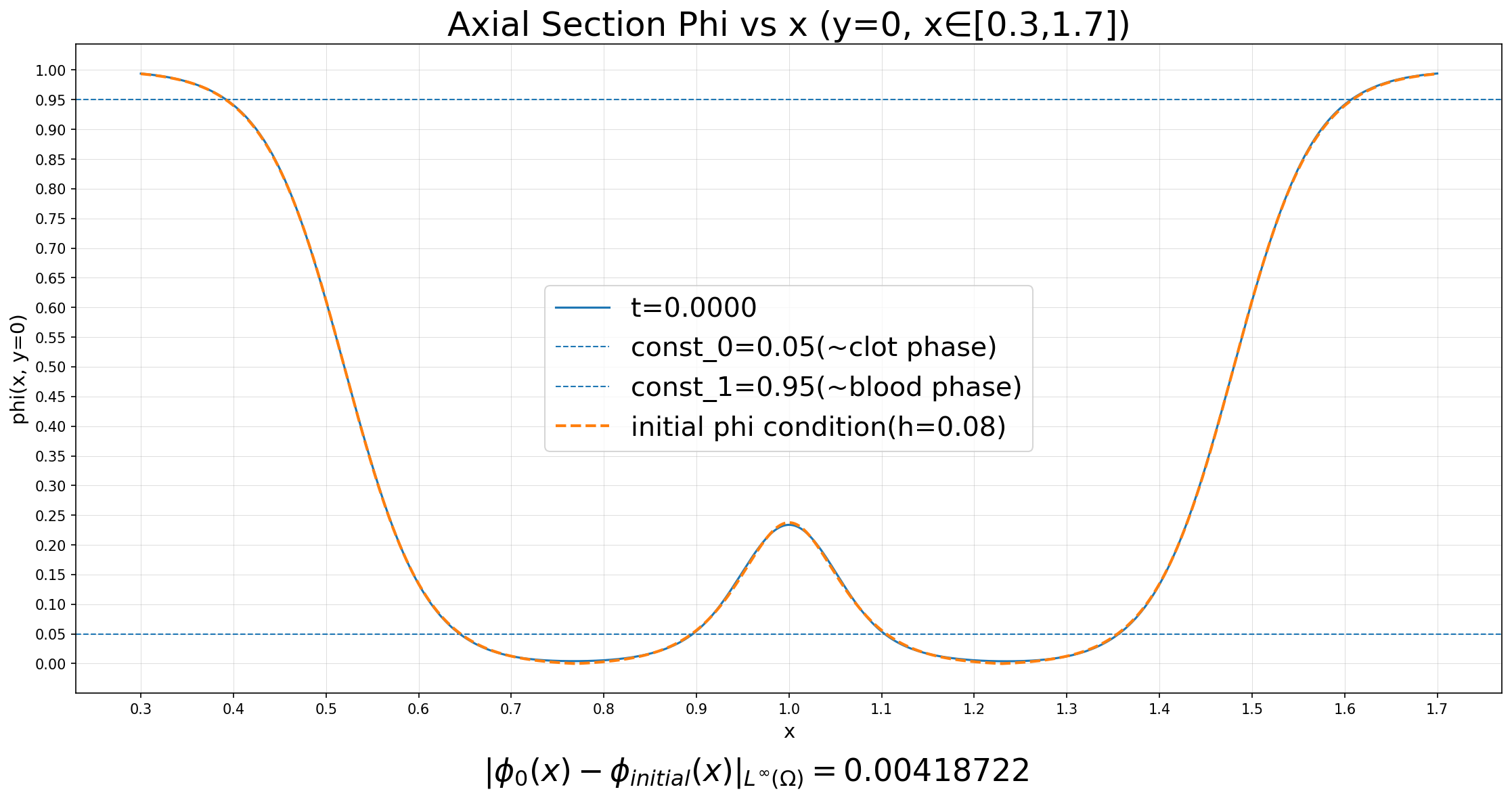}
      \caption{}
    \end{subfigure}
    \caption{(a) Before applying AA-PINN method and (b) after applying this on same experiment setting as Figure \ref{fig_sampling_E}.}
    \label{clots_initial_phi_E}
\end{figure}
\subsection{Parameter selection and case study}\label{case_study}
There are several different cases with different initial phase-field variable $\phi_0$ and various physical parameters on this system. For density $\rho$, we fixed this parameter as $1$ in this paper. Also, deformation variable diffusion term has coefficient $k = 10^{-5}$ in this numerical section.
\\
\par
To specify each experimental case, we first choose the physical parameters $\eta$, $\nu$ and $\tau$, which determine the strengths of the diffusive terms $\Delta u$, $\Delta F$ and $\Delta^2 \phi$. Based on this selection, secondly, other parameters $\kappa$, $\lambda$ and $\gamma$ are chosen to decide permeability or diffusive status of thrombus from initial thrombus configuration. We focus on the rapid variation of the phase-field variable 
$\phi$, since the interfacial region is the most dynamically active part of the system. For this rapid change from mixed energy $E_m(x,t)$, dominant change can result from the comparison of the coefficients $\lambda$ of $|\nabla \phi|^2$ term and $\lambda \gamma$ of $f (\phi) = \frac{1}{4h}\phi^2(1-\phi)^2$. To be specific, dissipative total energy drives whole system toward decreasing the energy and mixed energy $|\nabla \phi|^2$ term evolves toward diffusive $\phi$ profile to decrease steepness of interface. In the same way, double well potential term $f(\phi)$ of the mixed energy polarize the phase field $\phi$ as $0$ or $1$ to decrease itself. 
\\
\par
Therefore, a larger value of $\lambda$ leads to a more diffuse thrombus profile. For CASE B and B$^{\prime}$ in Table \ref{tab:cases-params}, compare them in Figure \ref{fig:diffuse}. 
Increasing $\lambda \gamma$ leads to a clearer separation between the blood and thrombus phases. For CASE C and C$^{\prime}$ in Table \ref{tab:cases-params}, compare them to see clear gathering of clots when we put two thrombi as initial configuration with higher $\lambda \gamma$ in Figure \ref{fig_sampling_E}.
{\footnotesize
\renewcommand{\arraystretch}{1.25} 
\begin{table}[H]
\centering
\begin{tabularx}{\linewidth}{P{2.5cm}|Y|Y|Y|Y|Y|Y|Y|Y|Y|Y}
\hline
 &$\eta_b$ & $\eta_t$&$\kappa_b$ &$\kappa_t$ & $\nu_b$&$\nu_t$ & $h$ & $\gamma$ & $\tau$& $\lambda$ \\
\hline\hline
Case A \qquad \quad Base line (static)& 
5*1e-1& 1& 5*1e4& 1e-4& 1e-5& 1& .08& 1e-1& 1e-4& 1e-3\\
\hline
Case B \qquad Diffusive thrombus & 
5*1e0& 1e1& 5*1e4&1&  5*1e-2& 1e-1& .08& 5*1e-2& 1e-2& 2*1e-3 \\
\hline
Case B\' \qquad Diffusive thrombus & 
5*1e0& 1e1& 5*1e4& 1& 5*1e-2& 1e-1& .08& 1e-1& 1e-2& 1e-3 \\
\hline
Case C (E) \qquad \quad Two thrombi & 
5*1e0& 1e1& 5*1e4& 1& 5*1e-2& 1e-1& .08& 5& 1e-2& 1e-3 \\
\hline
Case C\' ~(E) \qquad \quad Two thrombi & 
5*1e0& 1e1& 5*1e4& 1& 5*1e-2& 1e-1& .08& 1e-3& 1e-2& 1e-3 \\
\hline
Case D (E) \qquad \quad Thin interface& 
5*1e-1& 1& 5*1e4& 1e-4& 1e-5& 1& .035& 1e-1& 1e-4& 1e-3\\
\hline
Case D\' \qquad \quad Thin interface& 
5*1e-1& 1& 5*1e4& 1e-4& 1e-5& 1& .05& 1e-1& 1e-4& 1e-3\\
\hline
\end{tabularx}
\caption{Parameter settings for each case. For the challenging cases marked (E), PINN solver employed total energy-adaptive sampling from Auto-Adaptive PINN(\cite{bk25}).}
\label{tab:cases-params}
\end{table}
}
Also, note that we use $\nu(\phi)$ as cubic polynomial $\nu(0) + (\nu(1) - \nu(0) )\phi^2(3-2\phi)$ with positive viscoelasticity constants $\nu(1)$ on blood and $\nu(0)$ on thrombus without specific description. This is for simulating the system with higher distinguishable sensitivity from high-order differentiability. 
\\
The previous model used linear polynomial for this parameter in \cite{ktt22} and this brought catastrophic discontinuity on the $\nu'(\phi)$ since the range of this is discrete as three constants on pure blood region, whole interface region and pure clot region in split.
\\
\par
To format the result figures, we organize them using some of 1) bench marks error, 2) PINN training loss, 3) axial graph for phase field variable and 4) energy dissipation tracking, mostly until the system approach static status with small enough change of total energy at the last training time interval of length $\Delta_{unit} t = 0.05$ or $\Delta_{unit} t = 0.1$. 
\\
\par
\textbf{Case A. Static thrombus case (base line experiment)}
This is baseline case described in Table \ref{tab:cases-params}. It shows static movement for thrombus as time increases since all the training time interval $\Delta t$ resuls in the change of energy under $10^{-5}$(see Figure \ref{D1008_e}). Also, in Figure \ref{D1008}, it shows change of $\phi$ in $10^{-3}$ level without showing proper movement of axial section. This case takes a role of baseline experiment for solution dynamics of clots or benchmark result errors to compare with other thrombus cases. Simulation results can be described as the static axial evolution as Figure \ref{D1008}.
\\
\begin{figure}[H]
  \centering
    \begin{subfigure}[t]{0.49\textwidth}
      \centering
      \includegraphics[width=\linewidth,trim=35 3 5 20]{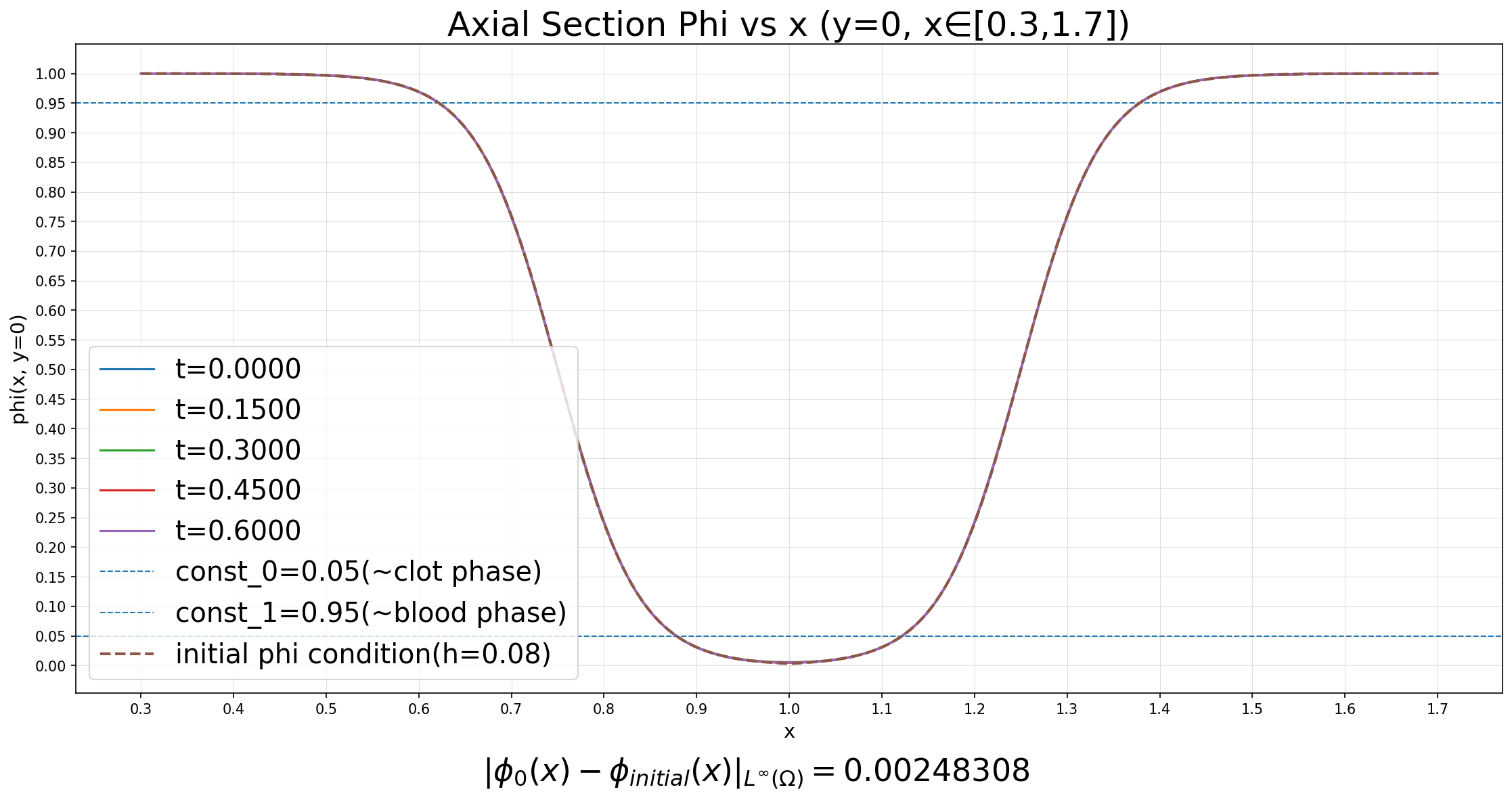}
      \caption{}
    \end{subfigure}
    \begin{subfigure}[t]{0.48\textwidth}
      \centering
      \includegraphics[width=\linewidth,trim=13 3 20 20]{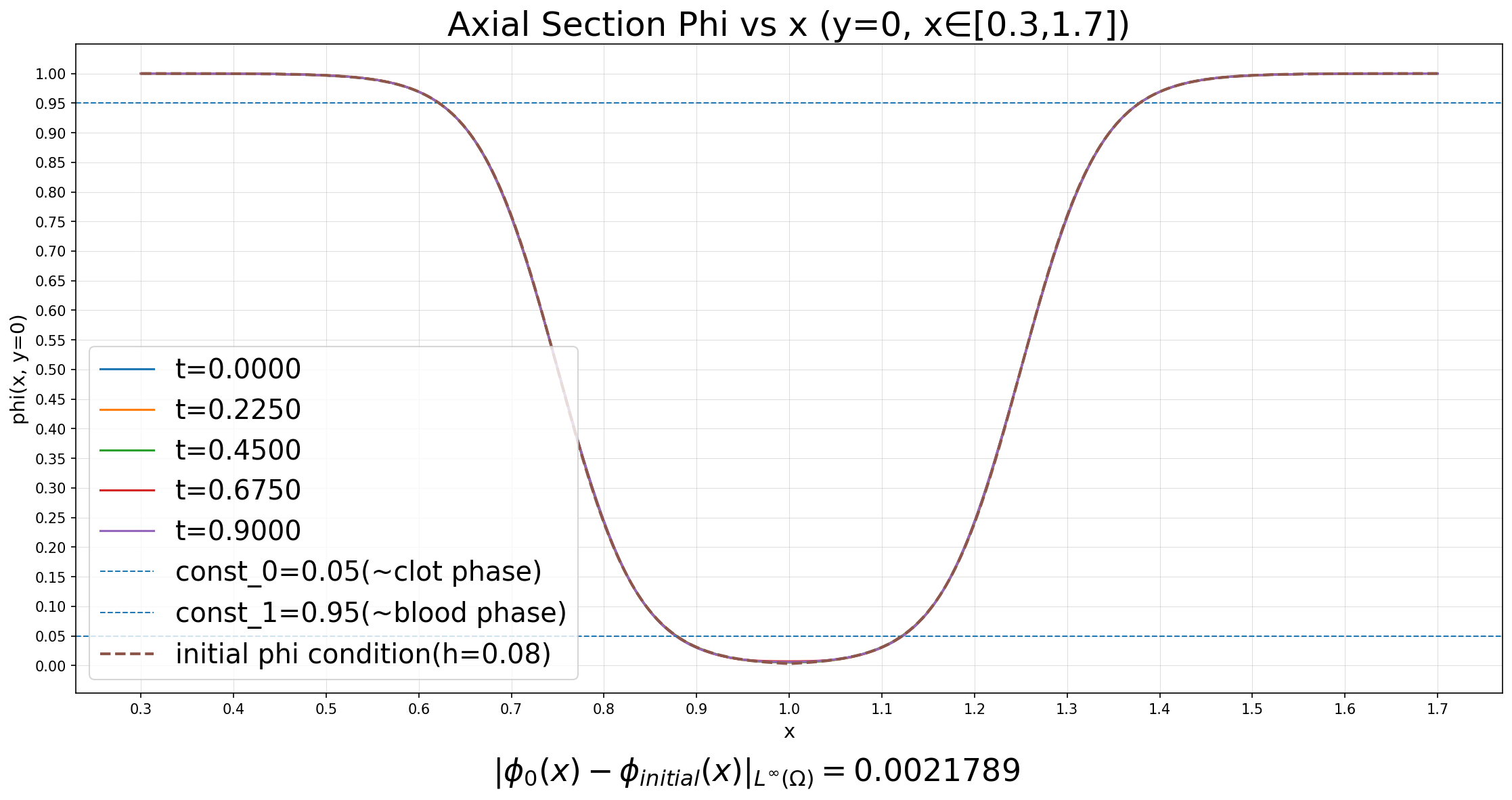}
      \caption{}
    \end{subfigure}
    \par\medskip 
    \begin{subfigure}[t]{0.48\textwidth}
      \centering
      \includegraphics[width=\linewidth,trim=13 20 10 20]{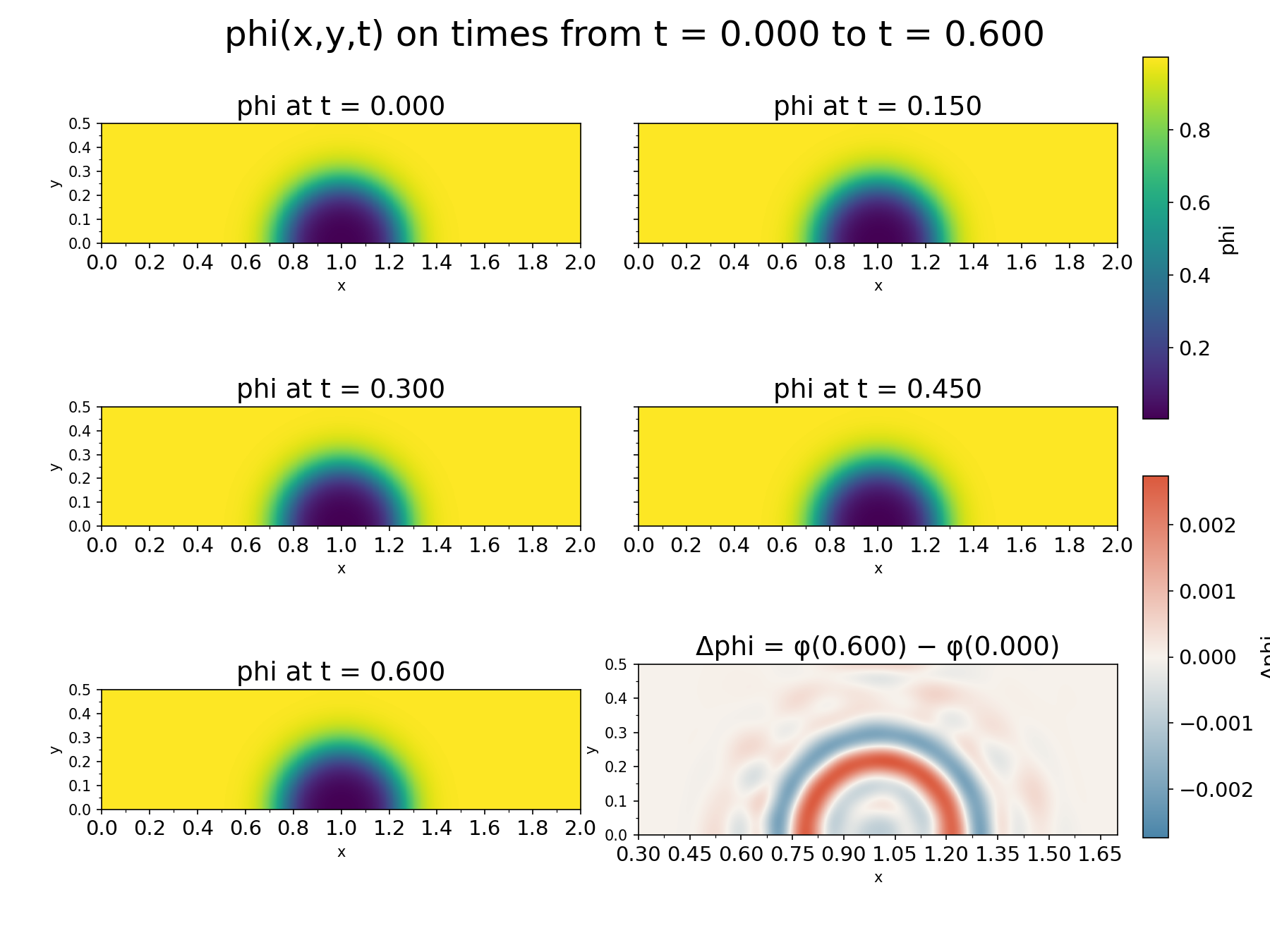}
      \caption{}
    \end{subfigure}
    \begin{subfigure}[t]{0.5\textwidth}
      \centering
      \includegraphics[width=\linewidth,trim=13 0 20 100]{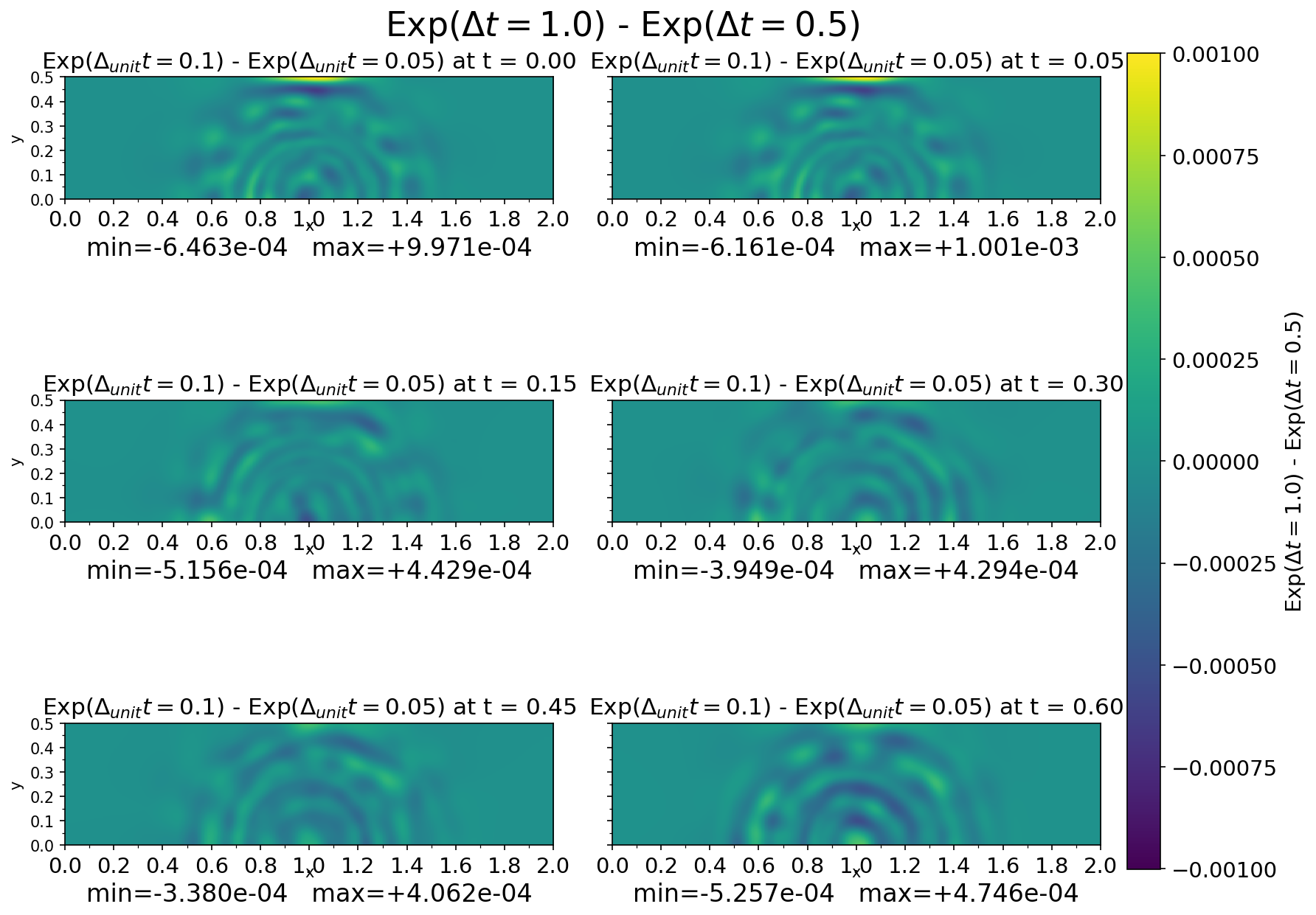}
      \caption{}
    \end{subfigure}
    \caption{Axial section figures of evolution for each unit train time interval $\Delta_{unit} t = 0.05$(a) and $\Delta_{unit} t = 1.0$(b). These have parameter as in the Case $A$ in Table \ref{tab:cases-params} and the clots are static with small enough change of $\phi$ value as we intended with the parameter setting. (c) is phi profile for For $\Delta_{unit} t = 0.05$ and (d) is difference for the both cases of $\Delta_{unit} t = 0.05$ and $\Delta_{unit} t = 0.1$. On plot (d), at $t = 0.05, 0.15, 0.45$, the both cases are in different in the view that one is at the end of the unit train time interval of $0.05$(Exp$(\Delta t = 0.05)$ and another is in the middle in the time interval of $0.1$(Exp$(\Delta t = 0.1)$. The separating interval of transfer learning points as $t = 0.15, 0.45$ of unit train time interval $\Delta_{unit} t = 0.05$ have small error bound compared with $\Delta_{unit} t = 0.1$ case which does not have transfer learning at the points. This means transfer learning perform well with small error.}
    \label{D1008}
\end{figure}
\begin{figure}[H]
  \centering
    \begin{subfigure}[t]{0.49\textwidth}
      \centering
      \includegraphics[width=\linewidth,trim=35 0 5 10]{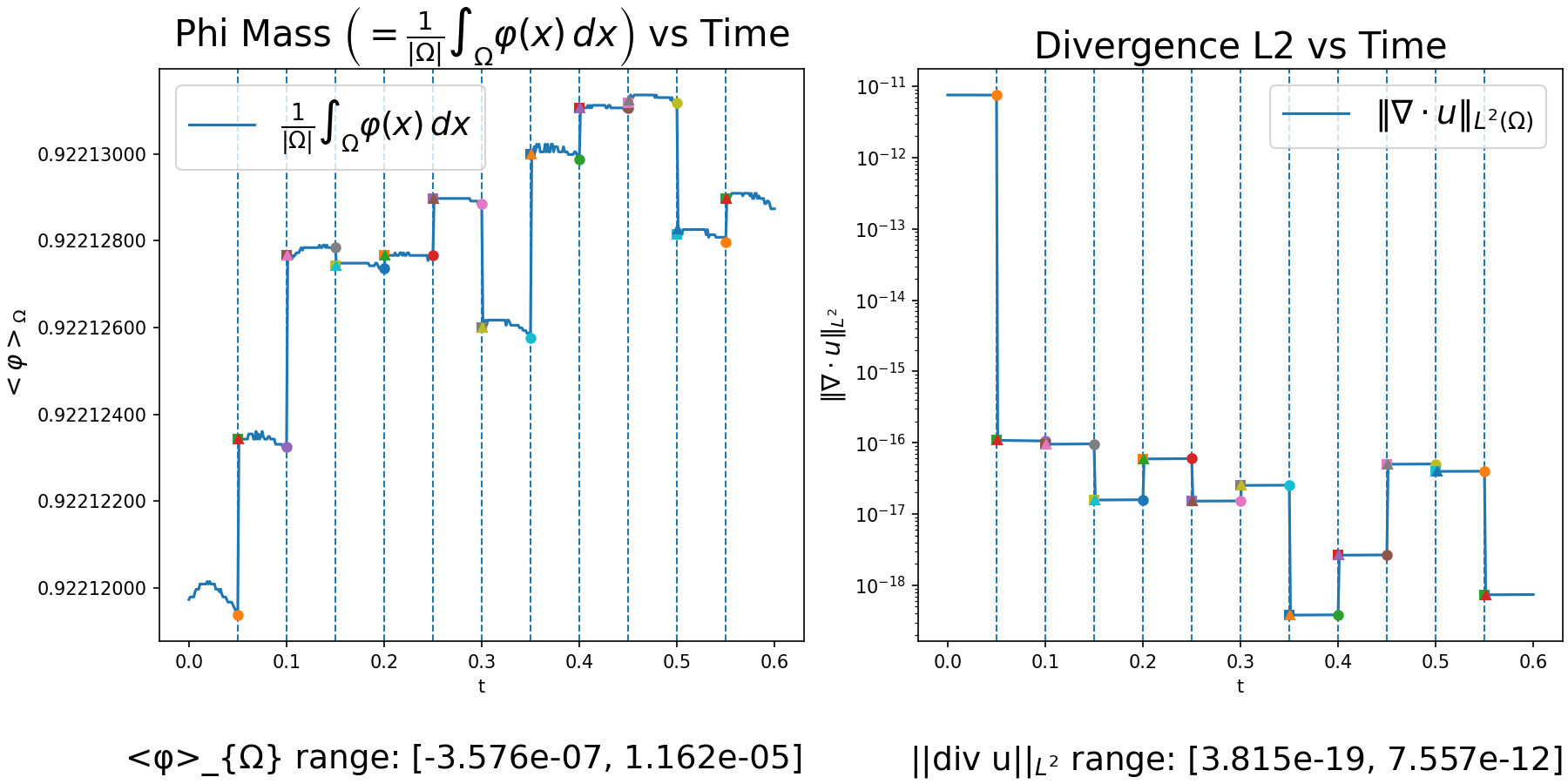}
      \caption{}
    \end{subfigure}
    \begin{subfigure}[t]{0.49\textwidth}
      \centering
      \includegraphics[width=\linewidth,trim=13 0 20 10]{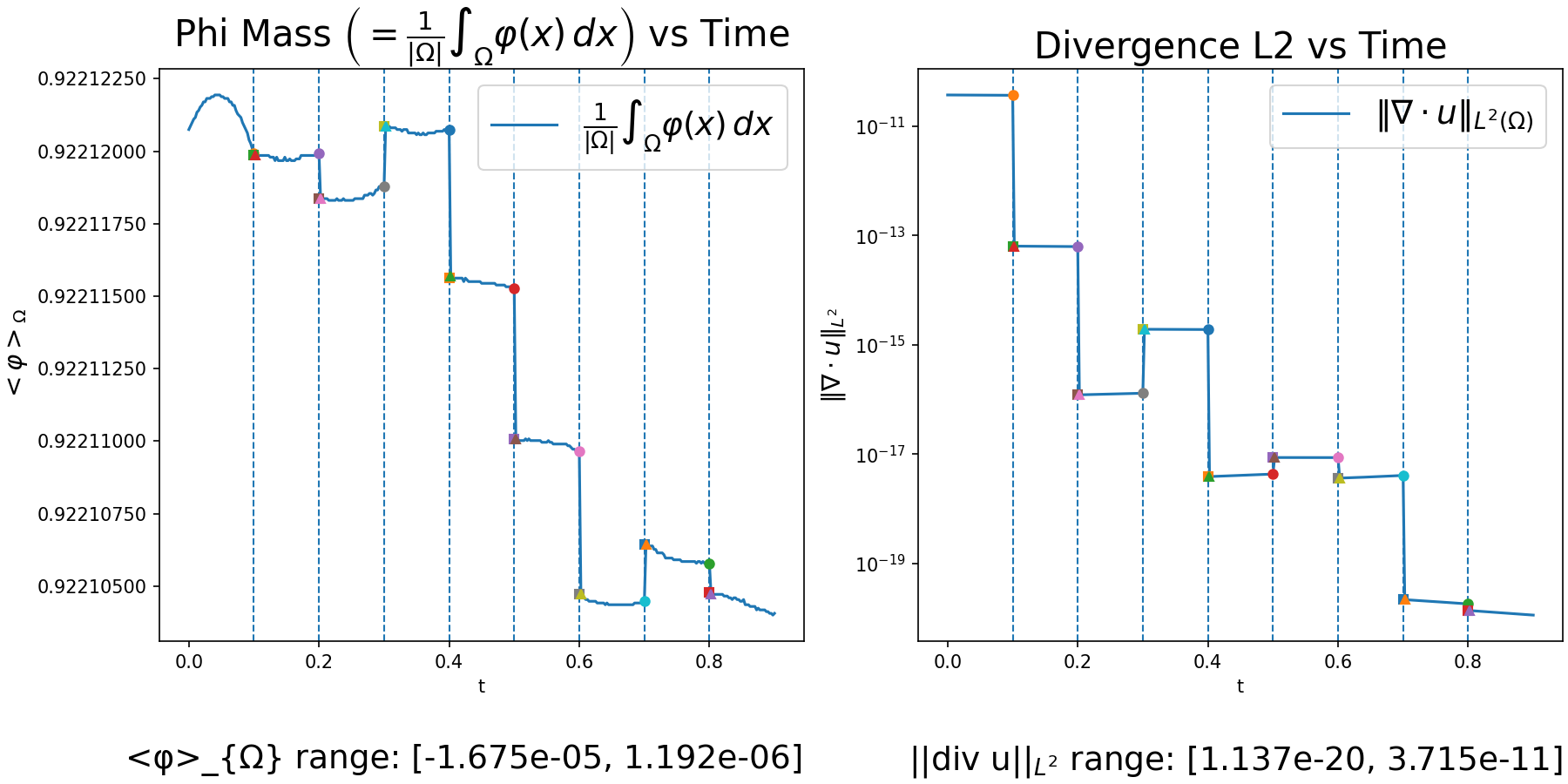}
      \caption{}
    \end{subfigure}
    \begin{subfigure}[t]{0.495\textwidth}
      \centering
      \includegraphics[width=\linewidth,trim=13 10 20 0]{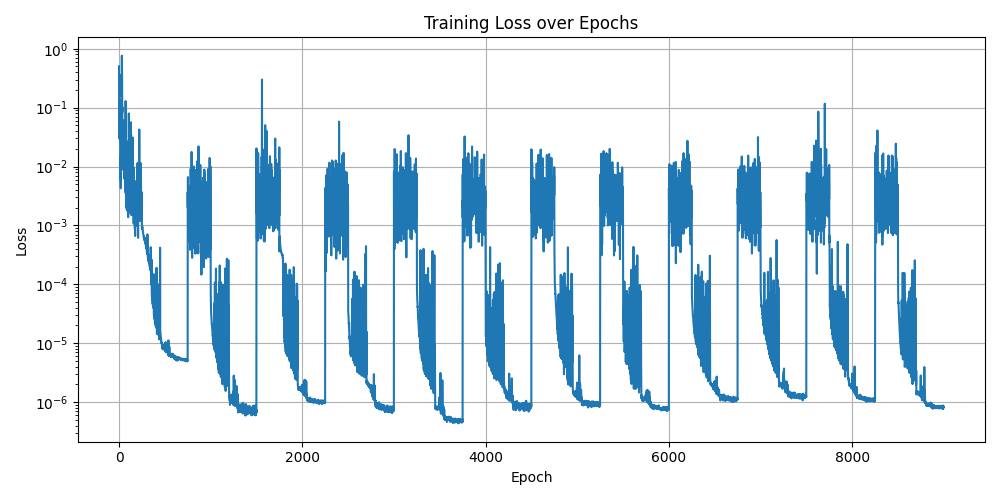}
      \caption{}
    \end{subfigure}
    \begin{subfigure}[t]{0.495\textwidth}
      \centering
      \includegraphics[width=\linewidth,trim=13 10 20 0]{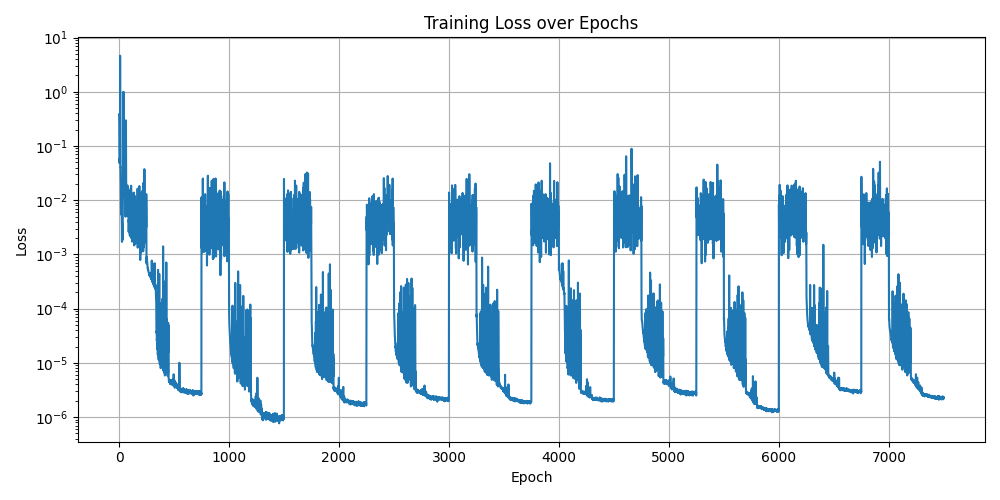}
      \caption{}
    \end{subfigure}
    \caption{Mass conservation error and divergence free condition error for each of unit train time intervals 0.5(a) and 1.0(b). The parameter setting comes from the Case $A$ in Table \ref{tab:cases-params}. We can check that the unit train tine interval Experiment of $\Delta_{unit} t = 0.05$ performs better than Experiment of $\Delta_{unit} t = 0.1$ case.}
    \label{D1008_2}
\end{figure}
\begin{figure}[H]
  \centering
    \begin{subfigure}[t]{0.49\textwidth}
      \centering
      \includegraphics[width=\linewidth,trim=35 10 5 20]{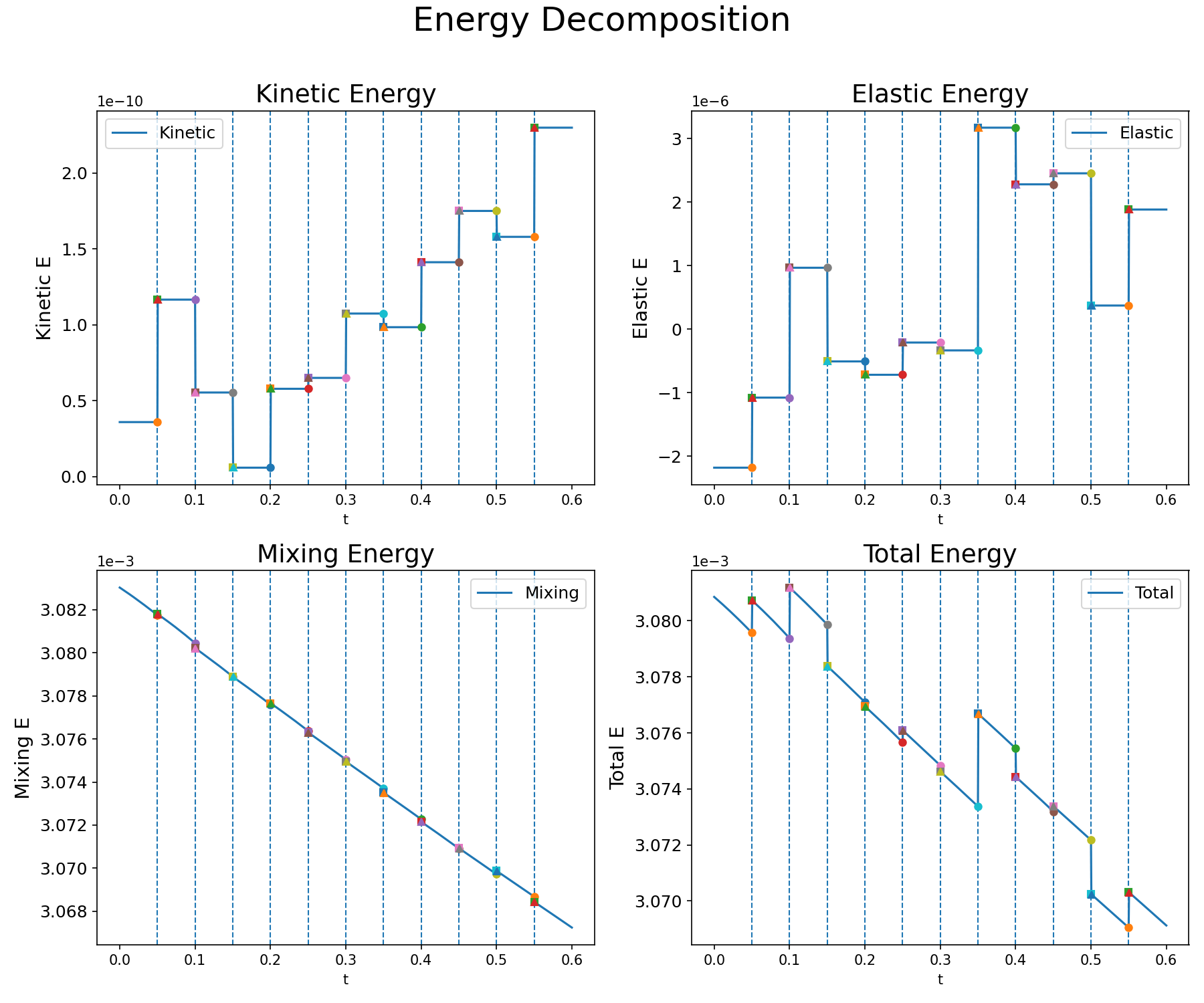}
      \caption{}
    \end{subfigure}
    \begin{subfigure}[t]{0.49\textwidth}
      \centering
      \includegraphics[width=\linewidth,trim=13 10 20 20]{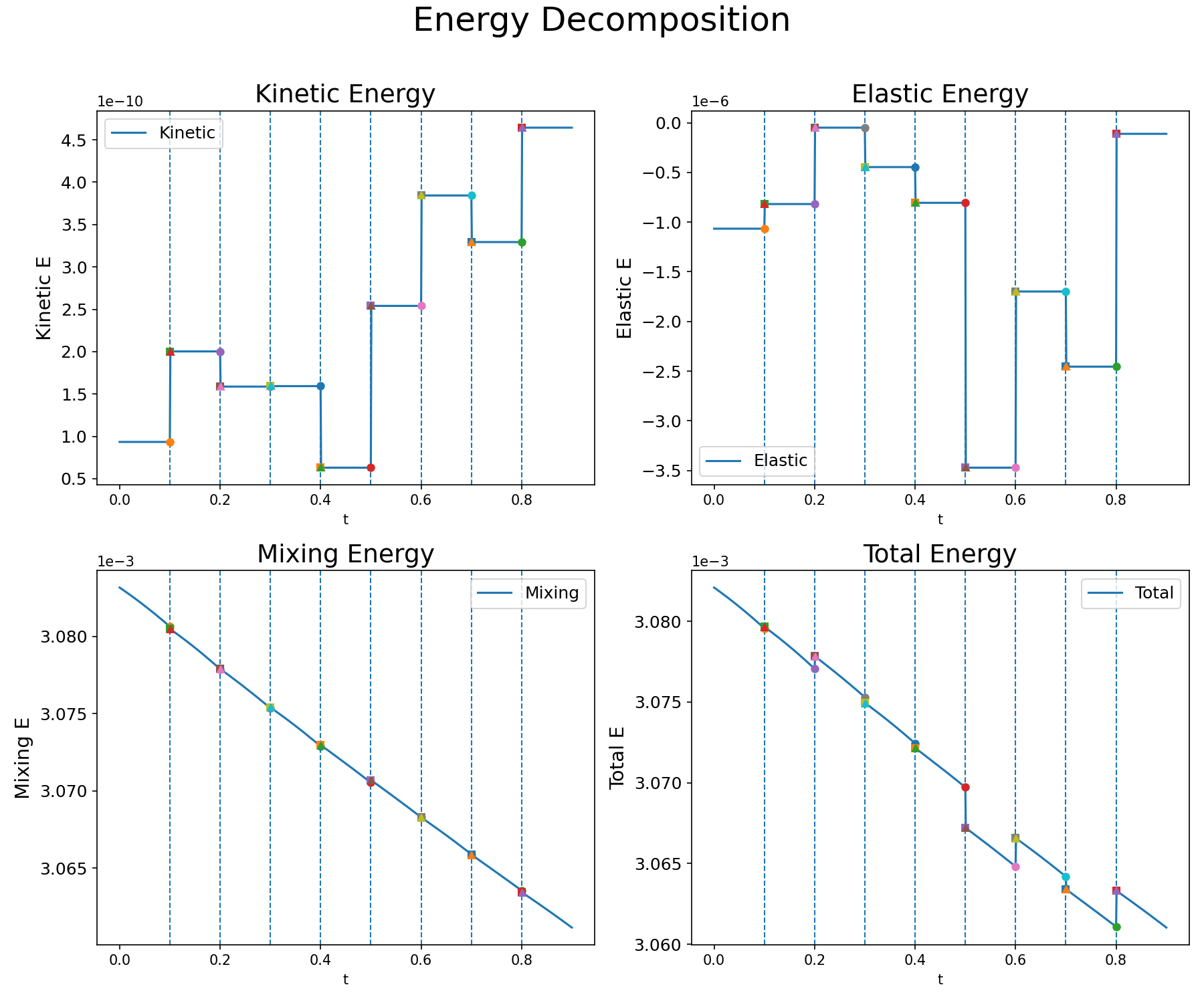}
      \caption{}
    \end{subfigure}
    \caption{Energy dissipation for each of unit train time intervals 0.05(a) and 0.1(b). The parameter setting comes from the Case $A$ in Table \ref{tab:cases-params}.}
    \label{D1008_e}
\end{figure}
We set the unit train time interval $\Delta_{unit} t$ as $0.05$ and $0.1$ for two experiments and compare the results for the benchmark errors as Figure \ref{D1008}. Since $\Delta_{unit} = 0.05 $ case is better than the other one in the Figure, we know that the experiments are performed in right way. Also, total energy \eqref{Total_Energy} keeps dissipative property for whole time domain as Figure \ref{D1008_e} under level of $10^{-5}$. 
\\
\par
Lastly in this case, in the Appendix \eqref{appendix:E_008}, we can check the energy dissipation of this these cases as we computed in \eqref{Total_Energy}.
\\
\par
\textbf{Case B. Diffusive thrombus cases}
In this case, by assigning proper parameters, we simulate diffusive thrombus as time goes on. For parameter setting, parameters are set as Case B and Case B' in the Table \ref{tab:cases-params}. These are intentionally set to increase permeability admitting the velocity change for blood and thrombus and increase the viscosity and viscoelasticity which also do not resist for the change of the variables and adhere to mix of the different values for corresponding variables streaming.
\begin{figure}[H]
  \centering
    \begin{subfigure}[t]{0.49\textwidth}
      \centering
      \includegraphics[width=\linewidth,trim=35 3 5 20]{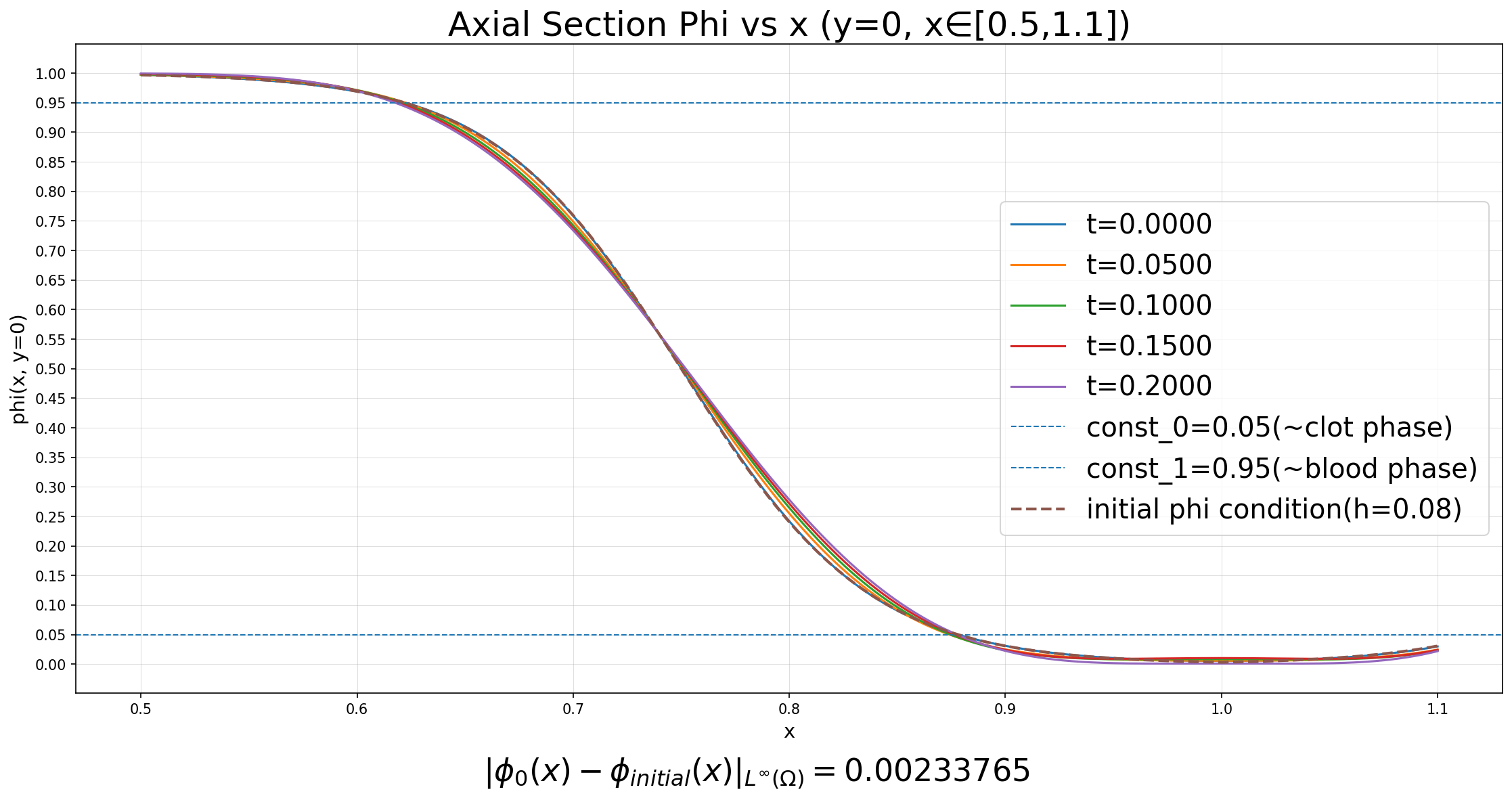}
      \caption{}
    \end{subfigure}
    \begin{subfigure}[t]{0.49\textwidth}
      \centering
      \includegraphics[width=\linewidth,trim=13 3 20 20]{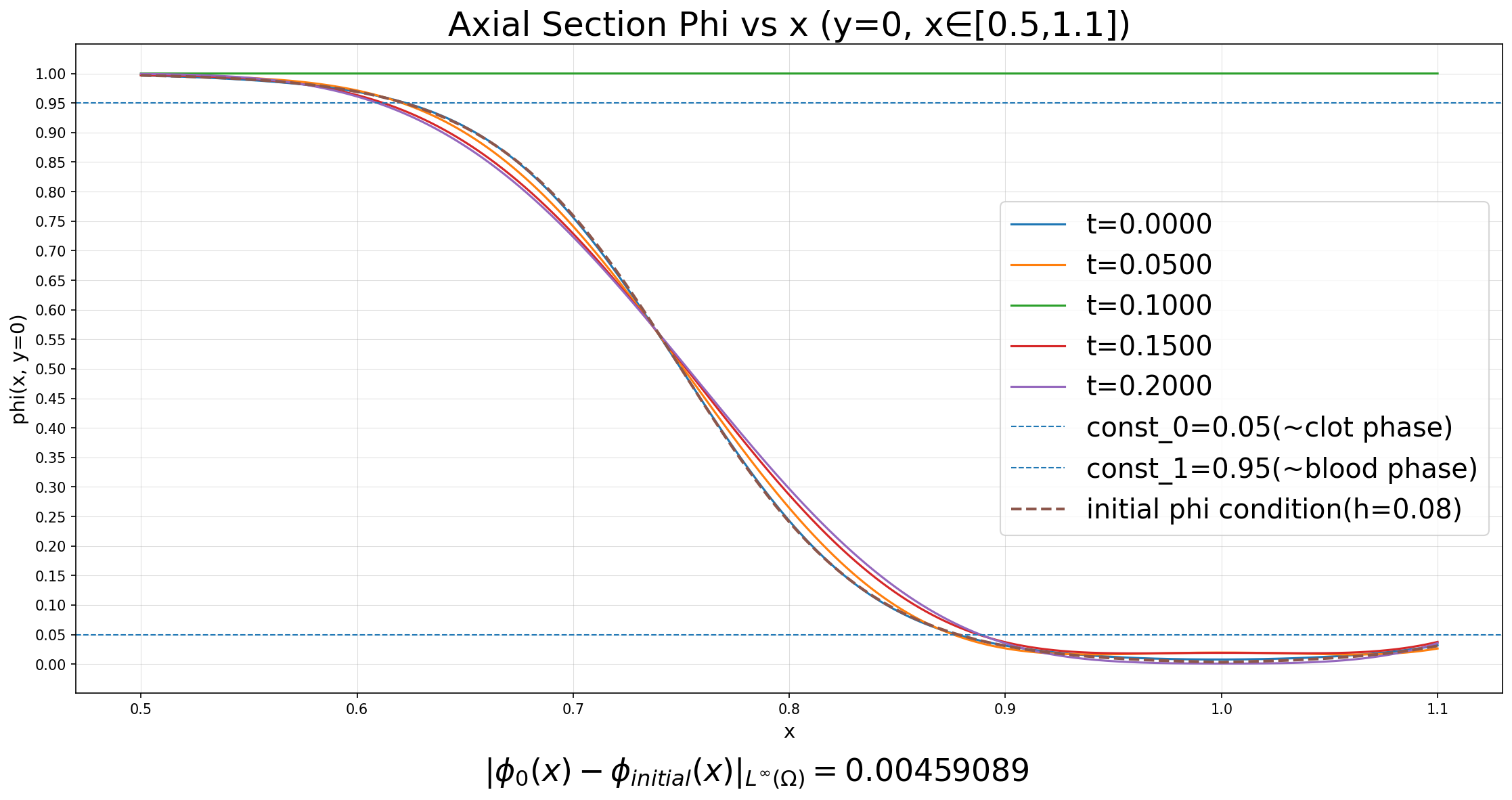}
      \caption{}
    \end{subfigure}
    \caption{Axial section figures of evolutionas (a) and (b). These (a) and (b) have parameter as in the Case $B'$ and Case $B$  in Table \ref{tab:cases-params} respectively. As in the more diffusive figure as (b), $\lambda$ of (b) is bigger than the one of (a) twice, while fixing the $\lambda \gamma$ in the mixed energy. $\lambda$ drives the dissipative energy to decrease $|\nabla \phi|^2$.}
    \label{fig:diffuse}
\end{figure}
On Figure \ref{fig:diffuse}, two cases of $B$ and $B^\prime$, different setting for $\lambda$ with fixed constant $\lambda \gamma$ of mixed energy term $\int_{\Omega} \lambda |\nabla \phi|^2 + 2\gamma \lambda f(\phi) dx$ have effect on extent of diffusion of the thrombus. This is because higher coefficient on $|\nabla \phi|^2$ dominate the dissipative total energy evolving as the main driving force on $\int_{\Omega} |\nabla \phi|^2 dx$. This energy term moves toward making lower $|\nabla \phi|$, which is equal to decreasing slope for phase field variable $\phi$ on mixture area of blood and thrombus.
\\
\par
Additionally, they showed also the total energy dissipation from the equation \eqref{2_27}. We was able to check this easily.
\\
\par
\textbf{Case C (E). Two thrombi cases}. These are the cases to observe thrombus gathering movement. To observe the clear movement of the two clots gathering, we adjust the $\lambda \times \gamma$ as more higher to make the double-well potential term in energy as main driving force. Since thrombus system total energy mainly moves from mixed energy on interfacial region, this adjusting makes sense. Also, this is because energy dissipative property on double-well potential term make the thrombus mixture polarize for the $\phi$ value to clot status $0$ or blood status $1$. As shown in Figure \ref{fig_sampling_E}, we observe that the clots gather toward the center more clearly in (b) than in (a). (a) shows more diffusive evolution on thrombus than (b) in the figure. Recall Figure \ref{clots_initial_phi_E} to see that the mid-point of the phase-field variable moves in polarized way in the plot(b) in Figure \ref{fig_sampling_E}, while remaining resistant to diffusion.
\\
\begin{figure}[H]
  \centering
    \begin{subfigure}[t]{0.38\textwidth}
      \centering
      \raisebox{0.3\height}{\includegraphics[width=\linewidth,trim=13 0 8 0]{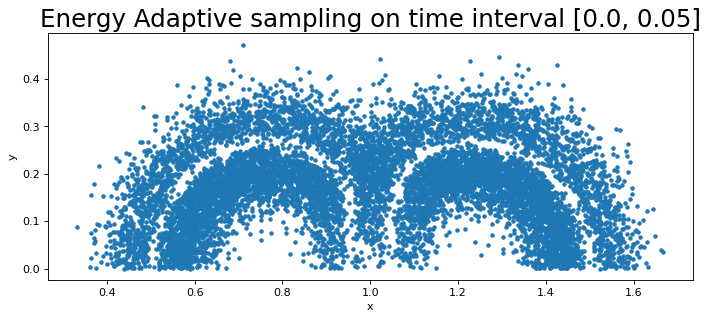}}
      \caption{}
    \end{subfigure}
    \begin{subfigure}[t]{0.6\textwidth}
      \centering
      \includegraphics[width=\linewidth,trim=8 0 20 0]{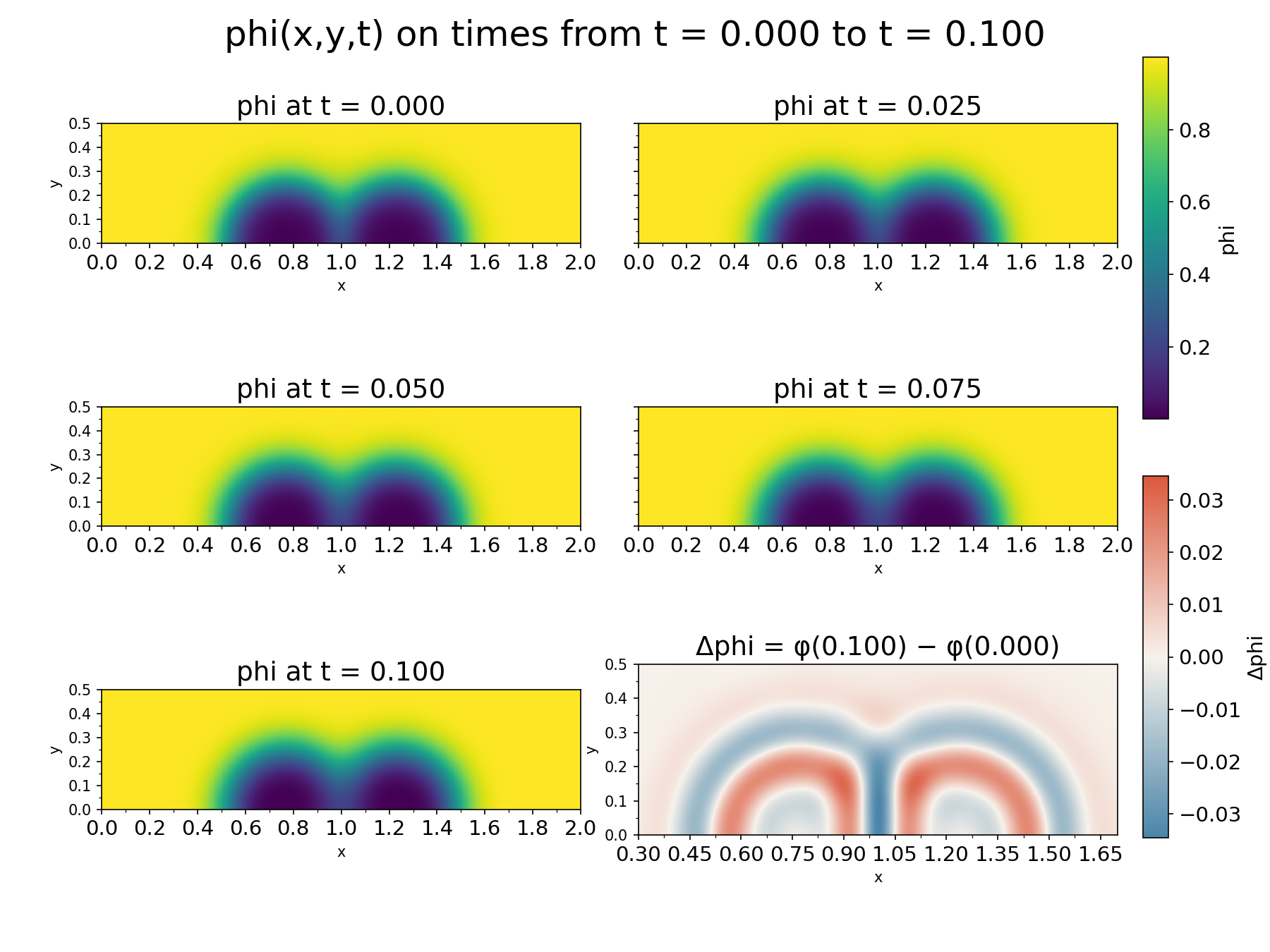}
      \caption{}
    \end{subfigure}
    \begin{subfigure}[t]{0.38\textwidth}
      \centering
      \raisebox{0.3\height}{\includegraphics[width=\linewidth,trim=13 0 8 0]{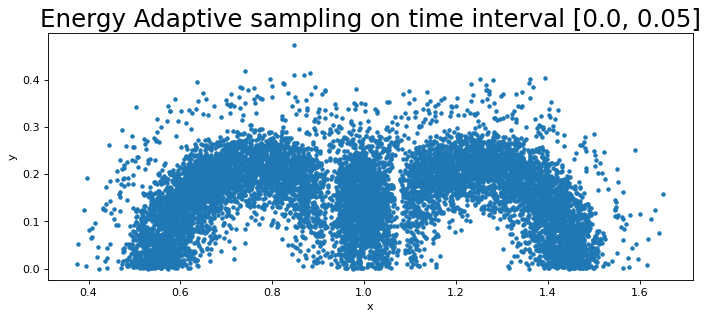}}
      \caption{}
    \end{subfigure}
    \begin{subfigure}[t]{0.6\textwidth}
      \centering
      \includegraphics[width=\linewidth,trim=8 0 20 0]{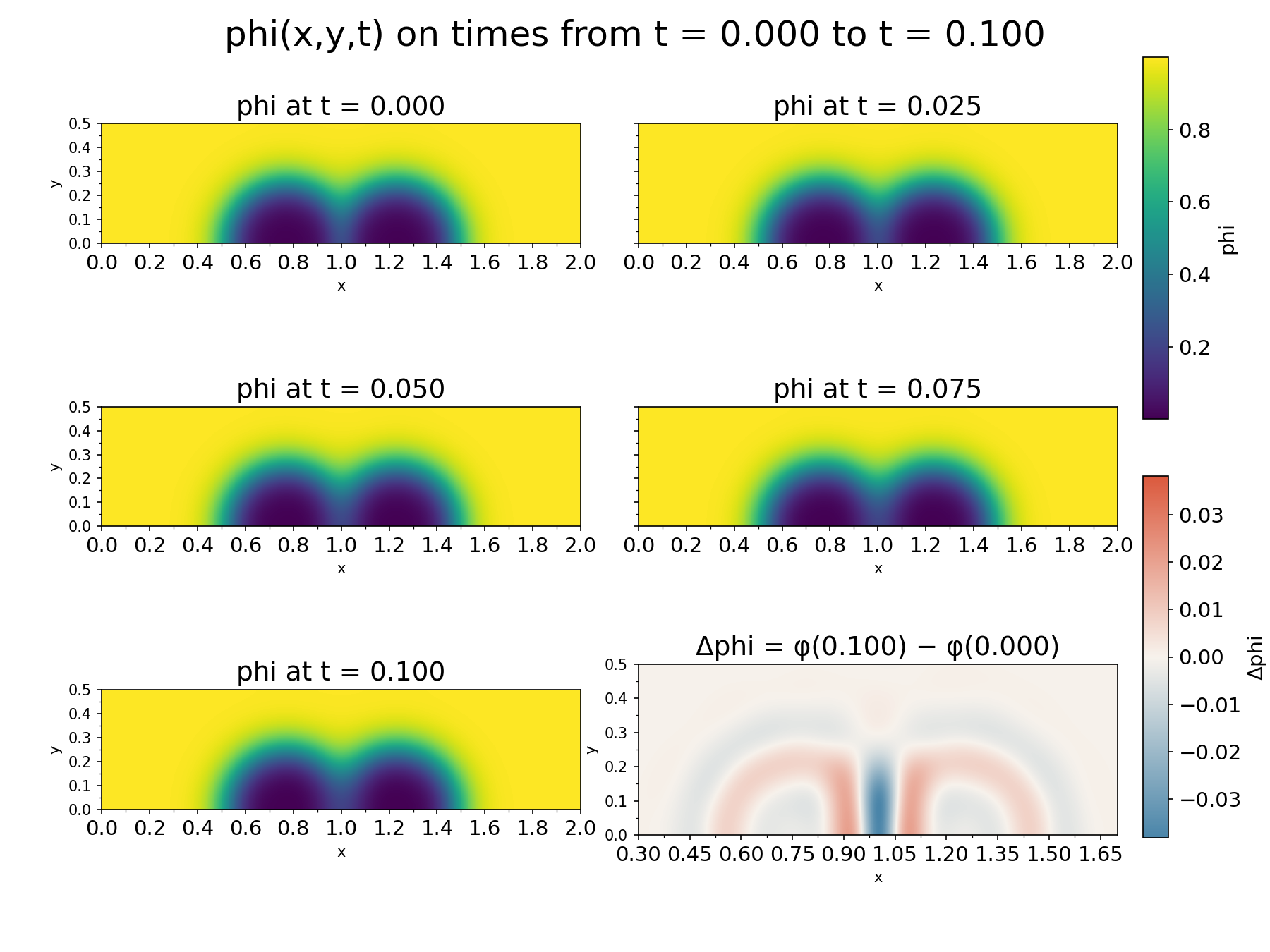}
      \caption{}
    \end{subfigure}
    \caption{Case $C'$ in (a) and (b) and Case $C$ in (c) and (d) in Table \ref{tab:cases-params}. First column plots are AA-PINN training points sampling and Second column plots are $\phi$ profile as time changes.
    Initial condition for $\phi$ profile is given as two thrombi as shown at $t = 0$ in phi profile (b) and (d). As in the more clear gathering of clots in (d) without boundary diffusive part which is shown in (b), $\lambda \gamma$ of (d) is bigger than the one of (b), while fixing the $\lambda $ in both cases. Hence, $\lambda \gamma$ drives the dissipative energy to decrease $|\phi(1-\phi)|^2$ more while separating phase-field.} 
    \label{fig_sampling_E}
\end{figure}
After evolving the Case $C$, as shown in Figure \ref{fig_clot_gathering}, we can see the total energy dissipation decreases to the level $10^{-4}$ for its change in the last training unit interval $\Delta_{unit} t = 0.1$. And, until $t = 0.5$, it shows that thrombi slowly overlap and gather to make them in one thrombus.
\begin{figure}[H]
  \centering
    \begin{subfigure}[t]{0.45\textwidth}
      \centering
      {\includegraphics[width=\linewidth,trim=13 0 8 0]{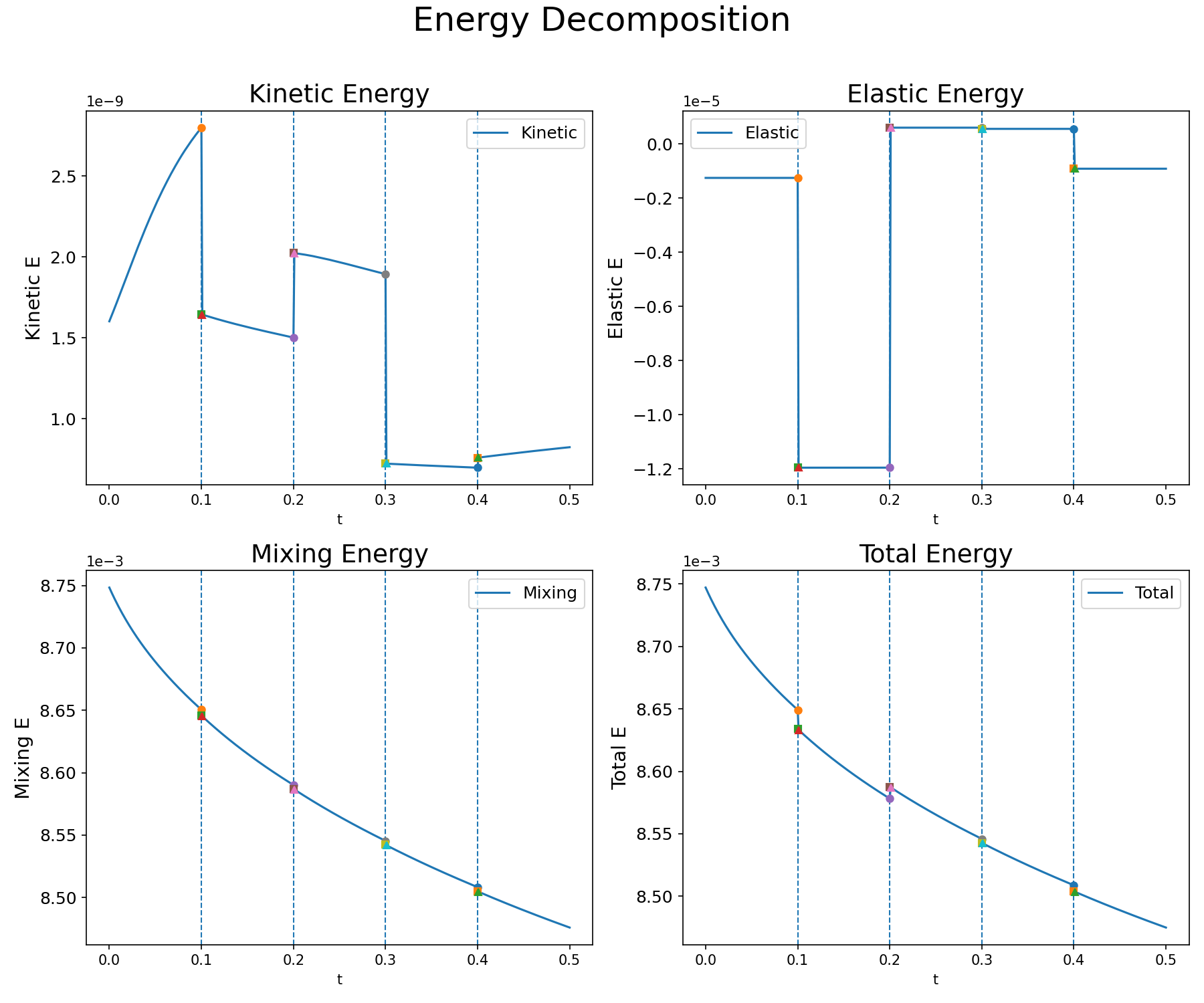}}
      \caption{}
    \end{subfigure}
    \begin{subfigure}[t]{0.5\textwidth}
      \centering
      \includegraphics[width=\linewidth,trim=8 0 20 0]{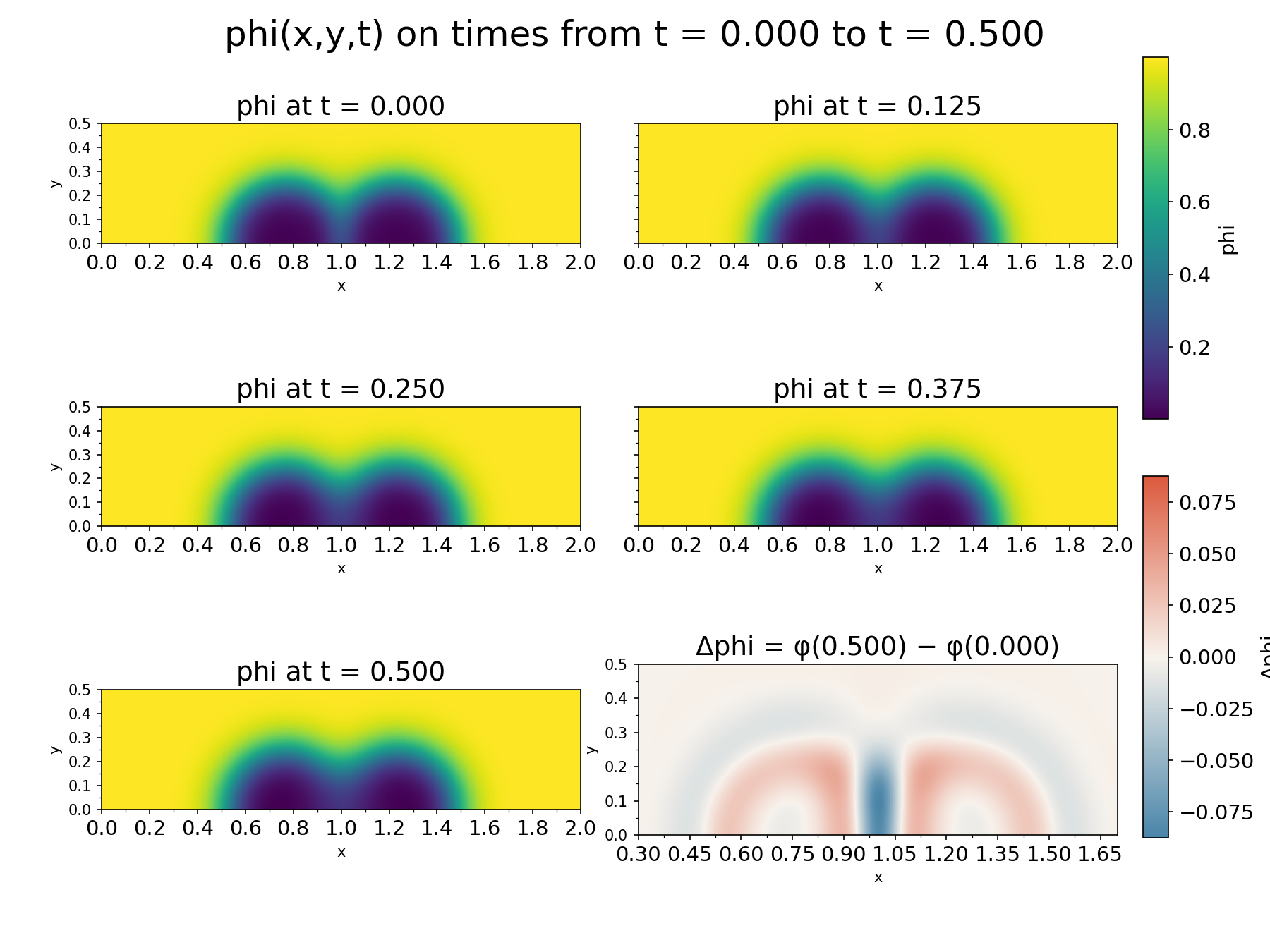}
      \caption{}
    \end{subfigure}
    \caption{Case $C'$ in (a) and (b) of Table \ref{tab:cases-params}.} 
    \label{fig_clot_gathering}
\end{figure}
\textbf{Case D (E). Thin interface cases} These are the cases for simulating thinner interface cases.
It is well known that the blood-thrombus interfacial region is difficult to simulate accurately because the phase-field variable develops very steep gradients there. In our governing system, numerical simulations face the same difficulty when the interface thickness $h$ is taken to be smaller. Therefore, we applied the AA-PINN sampling method to the case $h = 0.0035$, which reduced the absolute error relative to the initial profile $\phi \big|_{t = 0}$ from 0.0159 to 0.007109. This is $ 55.2893 \%$ decreasing and we can check this in Figure \ref{fig:e0035}.
\begin{figure}[H]
  \centering
    \begin{subfigure}[t]{0.49\textwidth}
      \centering
      \includegraphics[width=\linewidth,trim=35 3 5 20]{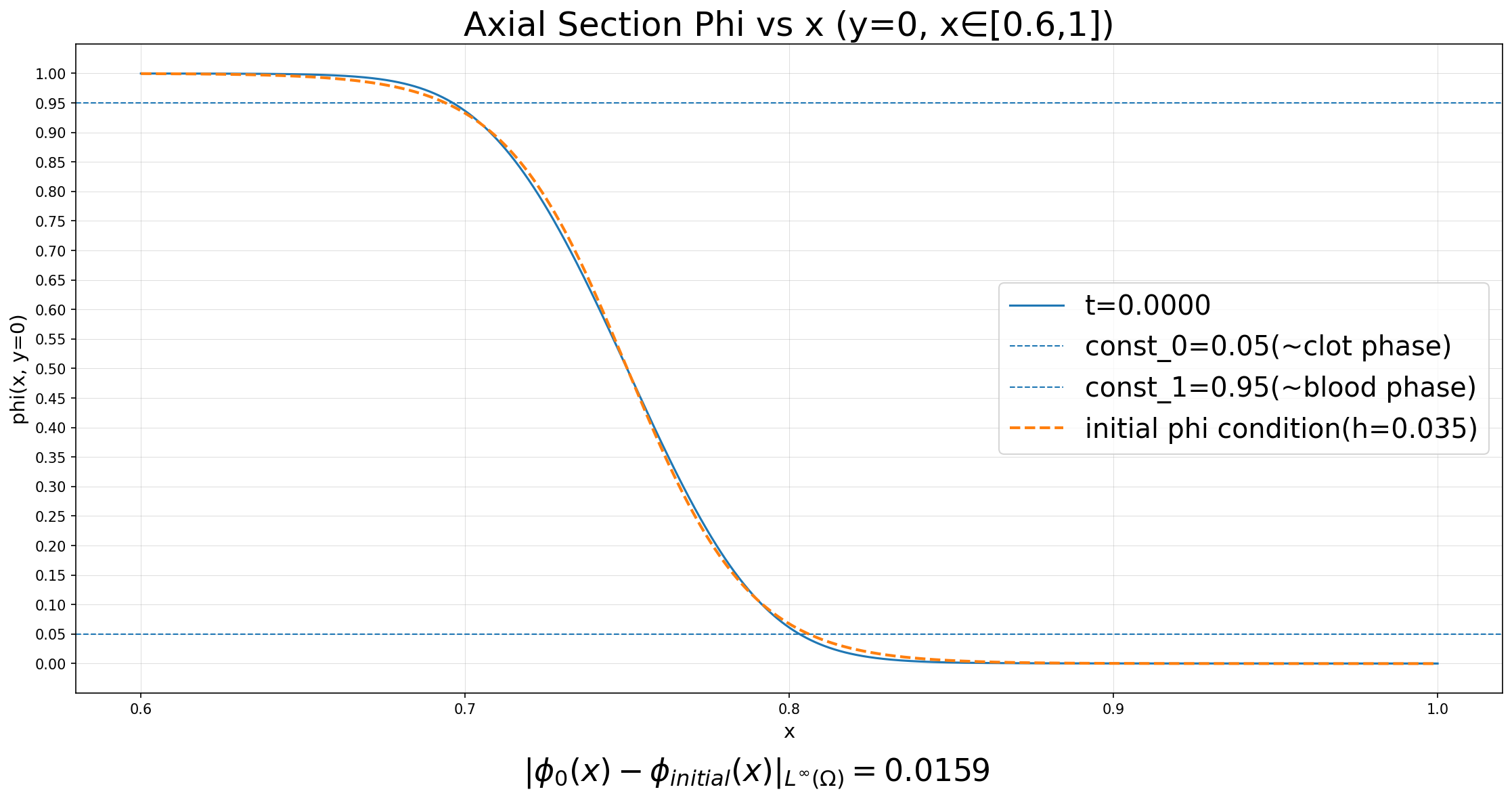}
      \caption{}
    \end{subfigure}
    \begin{subfigure}[t]{0.49\textwidth}
      \centering
      \includegraphics[width=\linewidth,trim=13 3 20 20]{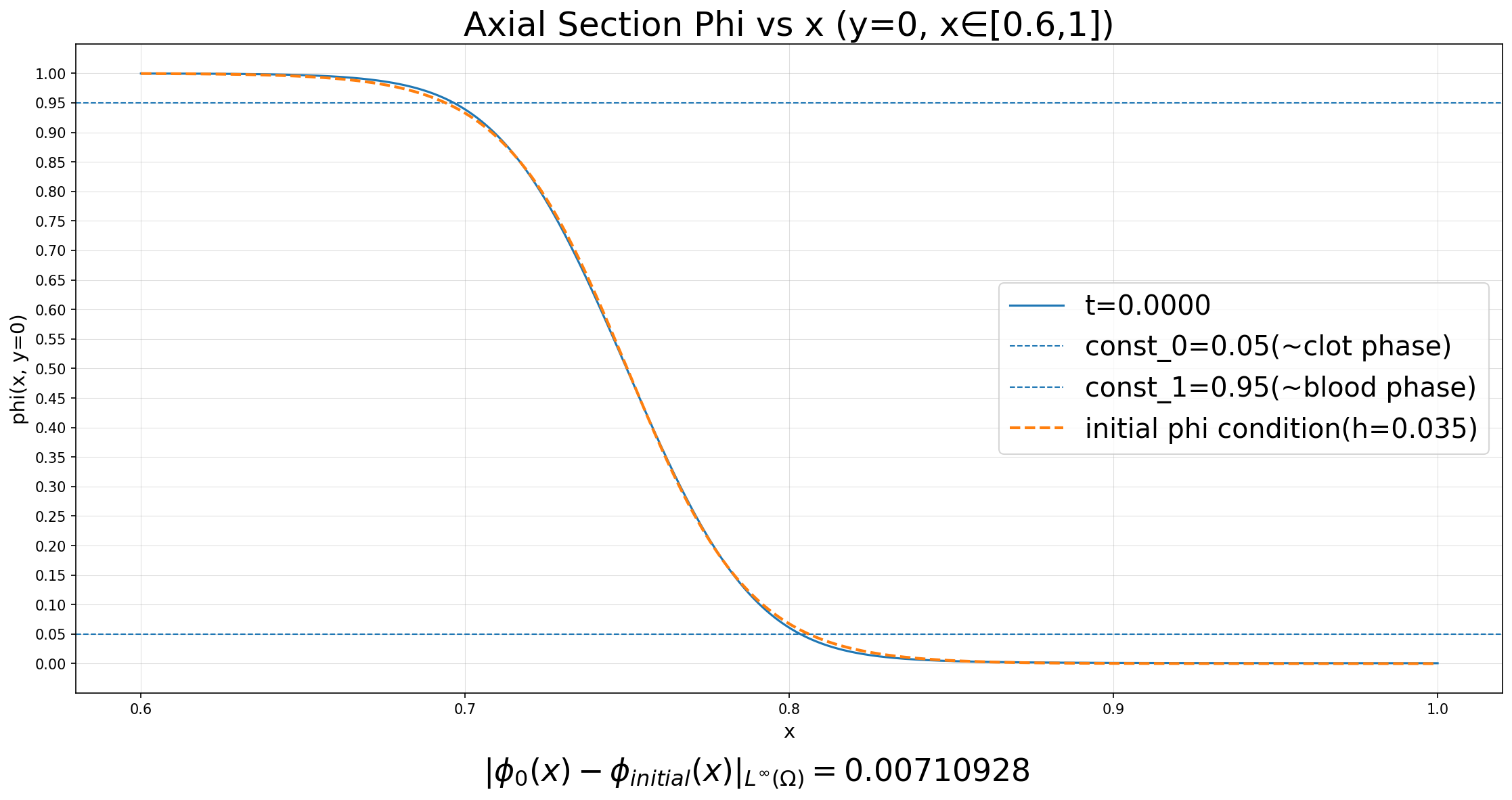}
      \caption{}
    \end{subfigure}
    \caption{Axial section figures of Case $D$ in Table \ref{tab:cases-params}. (a) does not use AA-PINN energy-adaptive sampling and (b) uses it.}
    \label{fig:e0035}
\end{figure}
Meanwhile, the thicker interface case $h = 0.05$ does not use the AA-PINN though it is more challenging case for simulation than $h=0.08$ Case A (Compare Figure \ref{fig:005} and Figure \ref{D1008}- \ref{D1008_2}). Axial section $L^{\infty}(\Omega)$ error for $\phi_0$ was 0.00227658 and thus this is slightly bigger than the 0.0021789 of (b) in Figure \ref{D1008}. Additionally, we can observe more detailed residual loss and benchmark results as follows to compare thinner interface Case $D'$ and baseline Case $A$ as Figure \ref{fig:005}.
\begin{figure}[H]
  \centering
    \begin{subfigure}[t]{0.49\textwidth}
      \centering
      \includegraphics[width=\linewidth,trim=35 3 0 20]{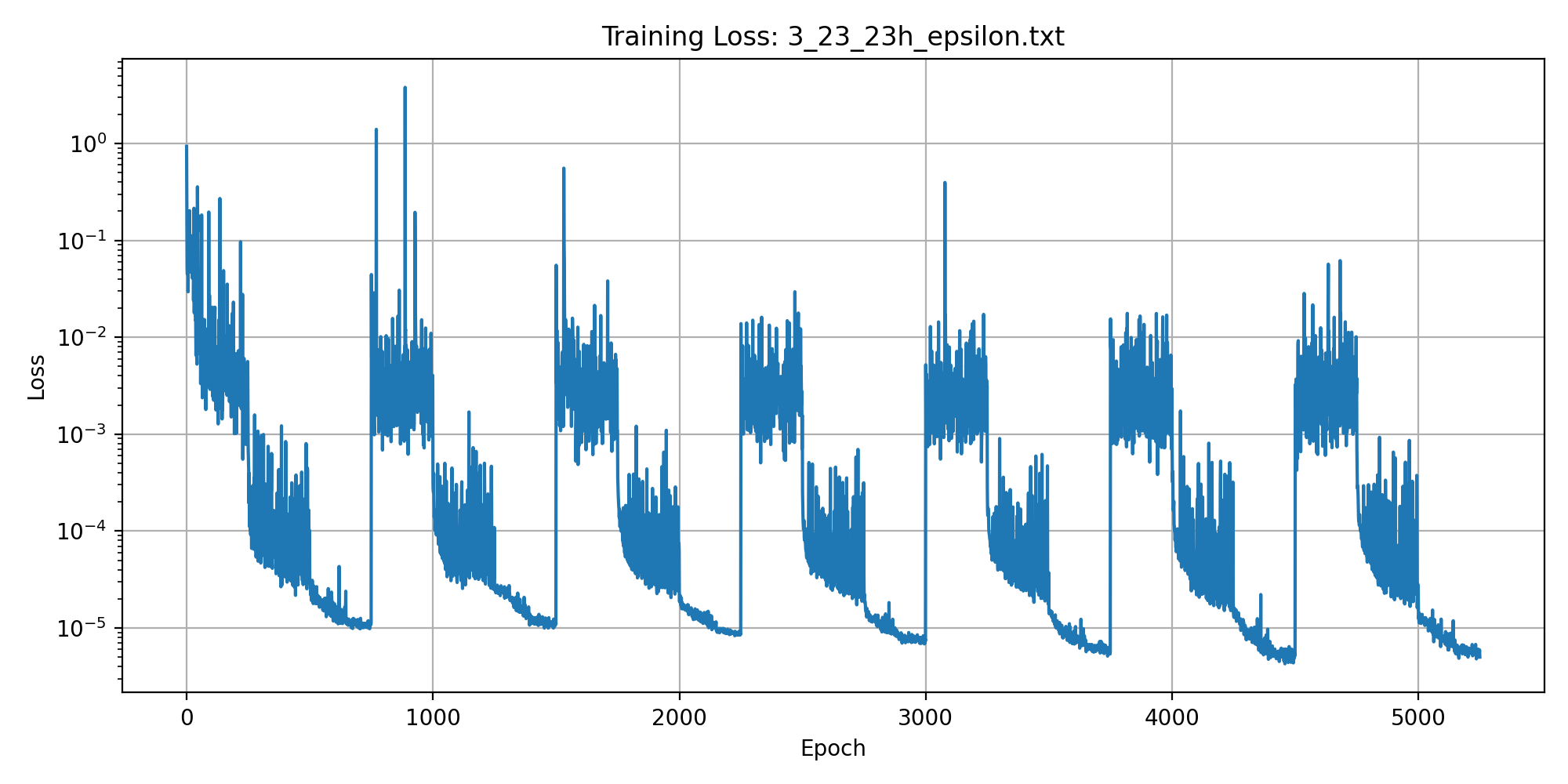}
      \caption{}
    \end{subfigure}
    \begin{subfigure}[t]{0.49\textwidth}
      \centering
      \includegraphics[width=\linewidth,trim=50 3 50 20]{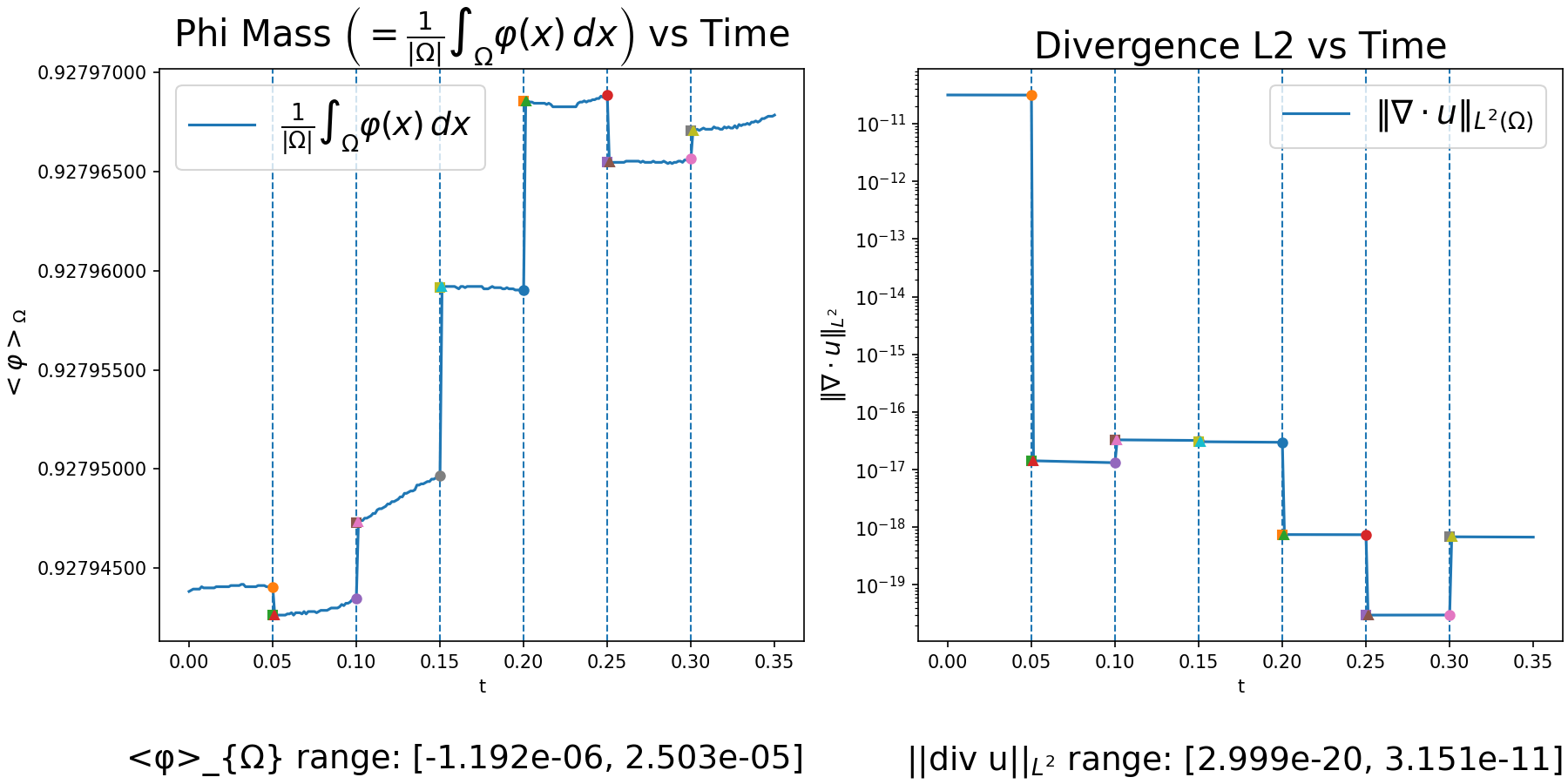}
      \caption{}
    \end{subfigure}
    \caption{$h = 0.05$ of Case $D'$ in Table \ref{tab:cases-params}. (a) is the residual loss graph and (b) is the mass conservation error and divergence free error. The loss and and errors are worse than $h = 0.08$ case in Figure \ref{D1008_2}. We observe that larger values of $h$ are numerically easier to simulate.}
    \label{fig:005}
\end{figure}
Since only interface thickness $h$ varies while the other parameters and initial configuration for $\phi_0$, $u_0$ and $F_0$ are the same as in Case A in the table, we can compare these two cases for $\phi_0$ graphs. Also, the smaller $h<0.08$ compared with case A gives more total energy via mixed energy term as we can compare in Figure \ref{D1008_e} and Appendix \ref{appendix:E_005}. This appendix also tracks the evolution of the energy over time.
\section{Conclusion}
In this paper, we studied a modified diffusion-enhanced NSCH–Oldroyd system motivated by thrombus modeling. For this modified system, we established a local well-posedness result and derived the associated energy-dissipation structure.
\\
\par
As a supplementary numerical component, we presented PINN-based numerical illustrations for representative thrombus cases. In particular, the window-sweeping training strategy together with the Metropolis–Hastings energy-adaptive sampling improved accuracy in challenging interfacial regimes.
\\
\par
Several directions remain for future work. In particular, the selection of physically calibrated model parameters and robust hyperparameter settings for the PINN framework requires further study. Applications to data-assimilation-based diagnostics also remain an interesting direction for future work.
\\
\par
\appendix
\section{PINN simulation parameter setting}
\subsection{Energy dissipation $E(x,t)$ of \eqref{Total_Energy} results for Case $A$}\label{appendix:E_008}
We can observe the energy dissipation as the Figure \ref{D1008_E}.
\begin{figure}[H]
  \centering
    \begin{subfigure}[t]{0.49\textwidth}
      \centering
      \includegraphics[width=\linewidth,trim=35 3 5 20]{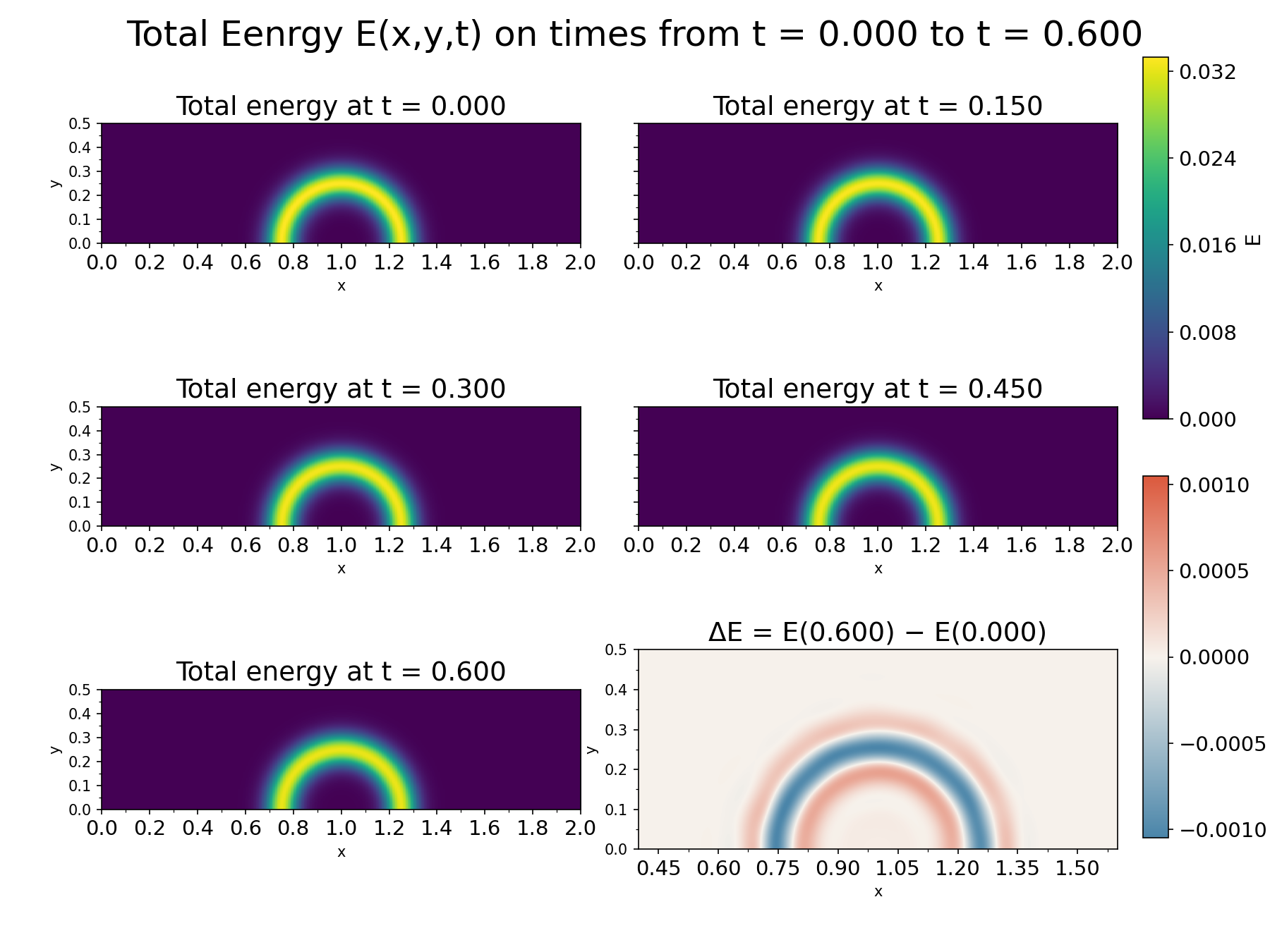}
      \caption{}
    \end{subfigure}
    \begin{subfigure}[t]{0.49\textwidth}
      \centering
      \includegraphics[width=\linewidth,trim=13 3 20 20]{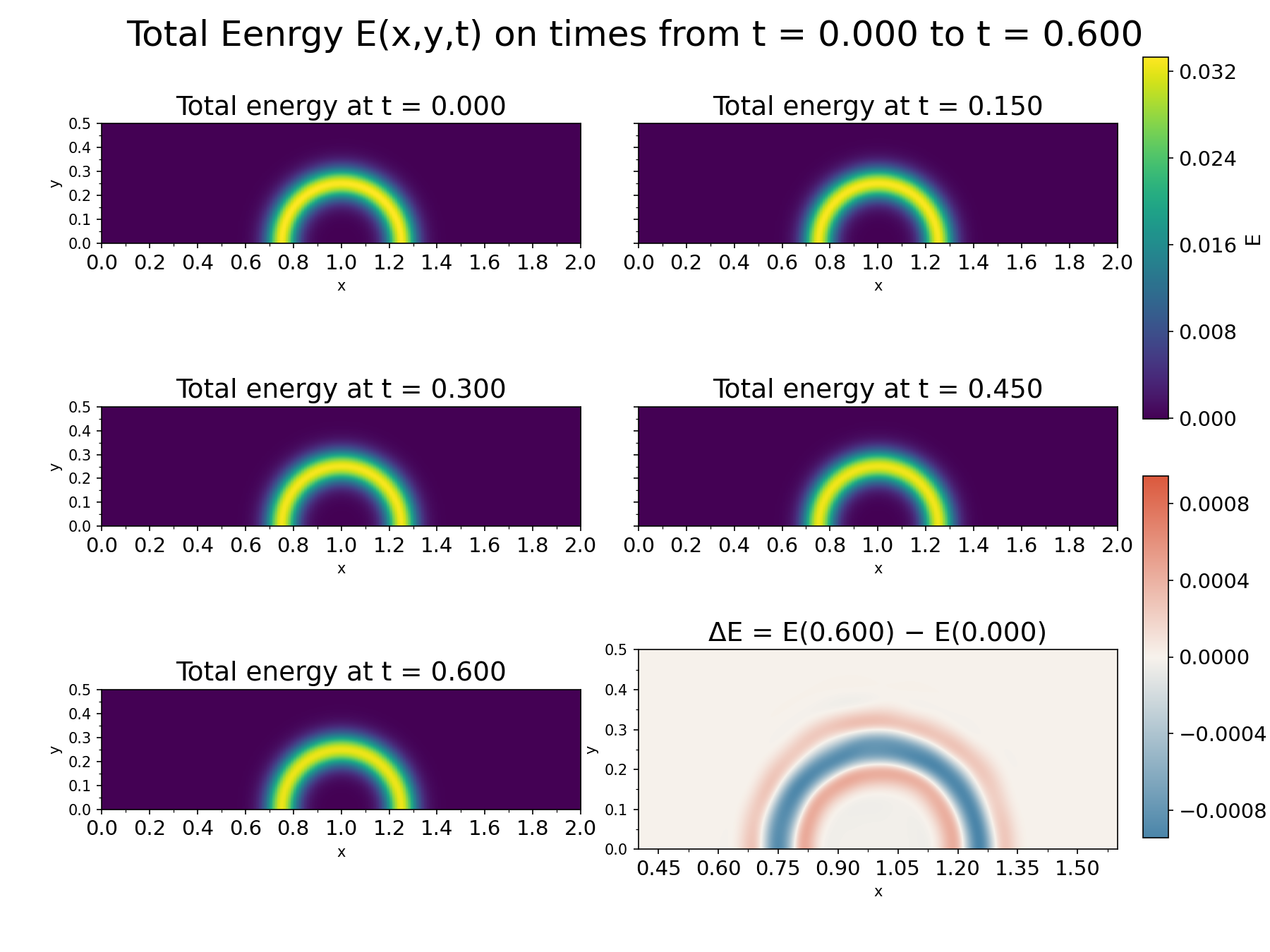}
      \caption{}
    \end{subfigure}
    \caption{Energy dissipation figures. These have parameter as in the Case $A$ in Table \ref{tab:cases-params} where (a) has $\Delta_{unit} t = 0.05$ and (b) has $\Delta_{unit} t = 0.1$.}
    \label{D1008_E}
\end{figure}

\subsection{Energy dissipation $E(x,t)$ of \eqref{Total_Energy} results for Case $D'$}\label{appendix:E_005}
As the previous appendix, there are the energy dissipation results as the Figure \ref{D1005_E}.
\begin{figure}[H]
  \centering
    \begin{subfigure}[t]{0.49\textwidth}
      \centering
      \includegraphics[width=\linewidth,trim=35 3 5 20]{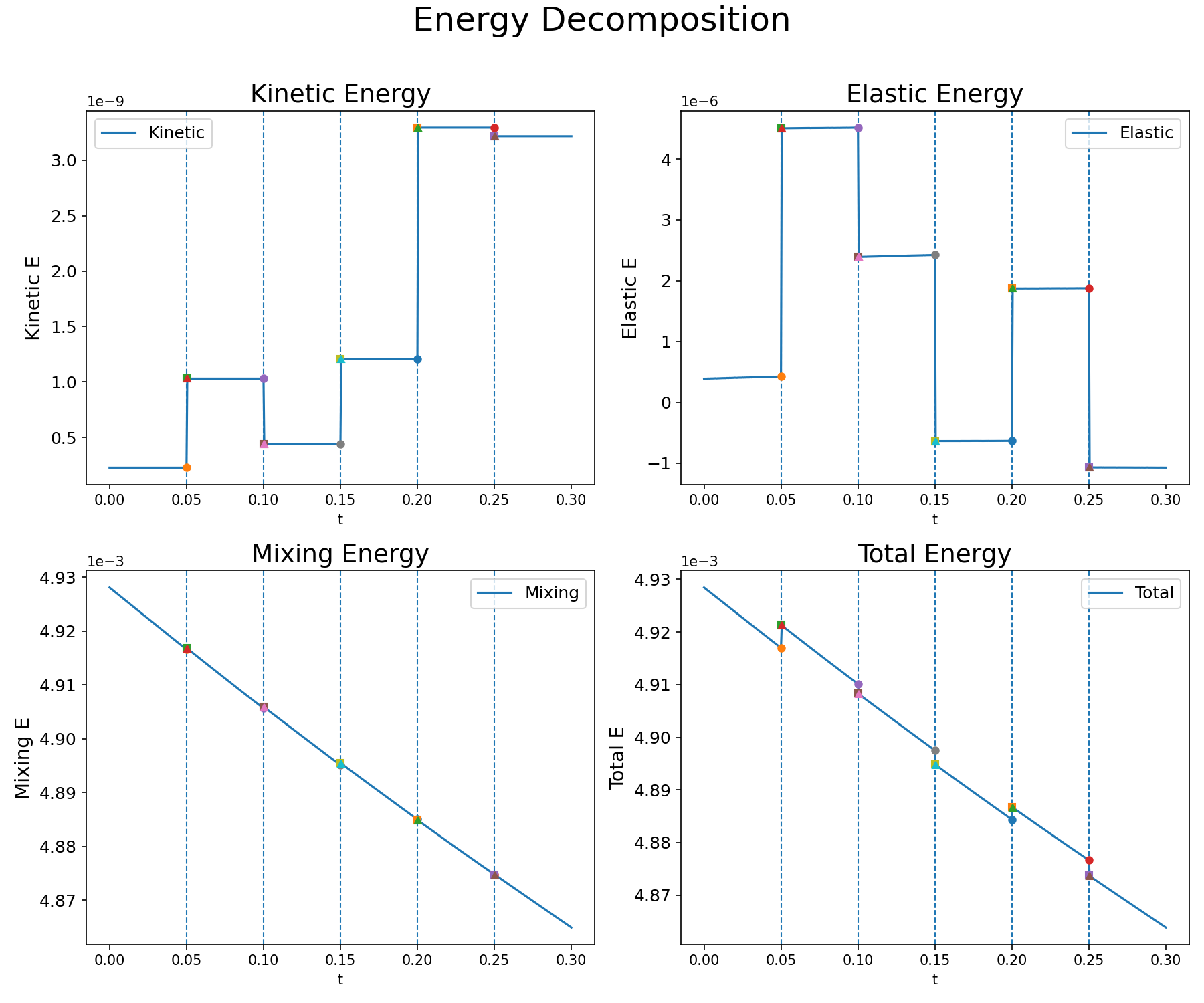}
      \caption{}
    \end{subfigure}
    \begin{subfigure}[t]{0.49\textwidth}
      \centering
      \includegraphics[width=\linewidth,trim=13 3 20 20]{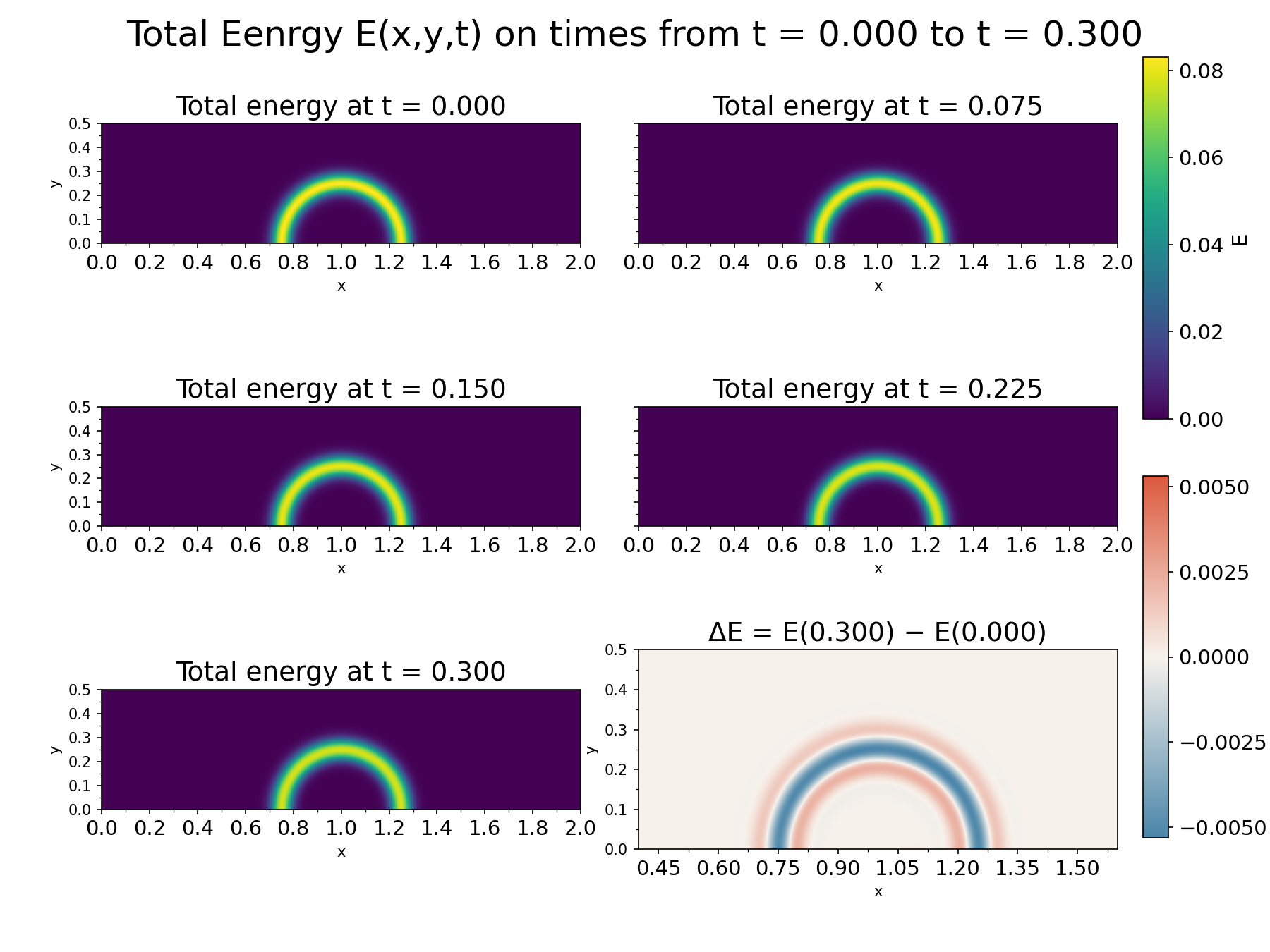}
      \caption{}
    \end{subfigure}
    \caption{Energy dissipation figures. These have parameter as in the Case $D'$ in Table \ref{tab:cases-params}.}
    \label{D1005_E}
\end{figure}

\end{document}